\documentclass[10pt]{amsart}
\usepackage{amssymb}
\usepackage{amsbsy}
\usepackage{amscd}
\usepackage[mathscr]{eucal}
\usepackage{verbatim}
\usepackage{version}
\usepackage{color}
\usepackage[matrix,arrow]{xy}
\usepackage{graphicx}

\def\cal{\mathcal}
\def\Bbb{\mathbb}
\def\frak{\mathfrak}

\newenvironment{NB}{
\color{red}{\bf NB}. \footnotesize 
}{}
\excludeversion{NB}
\newenvironment{NB2}{
\color{blue}{\bf NB}. \footnotesize
}{}

\newcommand{\SL}  {\operatorname{SL}}
\newcommand{ \Supp}{\operatorname{Supp}}
\newcommand{\Ext}{\operatorname{Ext}}
\newcommand{\Hom}{\operatorname{Hom}}

\newcommand{\im}{\operatorname{im}}

\newcommand{\rk}{\operatorname{rk}}

\newcommand{\NS}{\operatorname{NS}}
\newcommand{\coker}{\operatorname{coker}}
\newcommand{\Pic}{\operatorname{Pic}}
\newcommand{\ch}{\operatorname{ch}}

\newcommand{\Hilb}{\operatorname{Hilb}}

\newcommand{\Coh}{\operatorname{Coh}}

\newcommand{\Div}{\operatorname{Div}}

\newcommand{\alg}{\operatorname{alg}}

\newcommand{\f}{{\bf f}}

\newcommand{\End}{\operatorname{End}}

\newcommand{\Mod}{M}

\font\b=cmr10 scaled \magstep5
\def\bigzerou{\smash{\lower1.7ex\hbox{\b 0}}}

\numberwithin{equation}{section}
\theoremstyle{plain}
 \newtheorem{thm}{Theorem}[section]
 \newtheorem{lem}[thm]{Lemma}
 \newtheorem{prop}[thm]{Proposition}
 \newtheorem{cor}[thm]{Corollary}

\theoremstyle{definition}
 \newtheorem{defn}[thm]{Definition}
 
\theoremstyle{remark}
 \newtheorem{rem}[thm]{Remark}
 \newtheorem{ex}[thm]{Example}
 
\begin{document}

\title{Stability and Fourier-Mukai transforms on an eliptic surface.
}
\author{K\={o}ta Yoshioka}
\address{Department of Mathematics, Faculty of Science,
Kobe University,
Kobe, 657, Japan
}
\email{yoshioka@math.kobe-u.ac.jp}

\thanks{
The author is supported by the Grant-in-aid for 
Scientific Research (No. 21H04429, 23K03053, 26K06742), JSPS}

\begin{abstract}
We shall introduce a stability condition for a coherent sheaf associated to an elliptic surface. 
Then we study the behavior under relative Fourier-Mukai transforms.
\end{abstract}

\maketitle

\renewcommand{\thefootnote}{\fnsymbol{footnote}}

\section{Introduction}
The study of moduli spaces of stable sheaves on elliptic surfaces are started by fundamental works of  
Friedman (see \cite{F1} and \cite{F2}).
He obtained a birational description of the moduli spaces of rank 2 stable sheaves.
In particular it is proved that the moduli spaces are birationally equivalent to the 
Hilbert scheme of points, if the relative degree of the first Chern class
is odd.
Later it was generalized to higher rank cases by Bridgeland \cite{Br:1} and Yoshioka \cite{Y:Nagoya}. 
For the construction of the birational map,
Bridgeland's method \cite{Br:1} is notable, that is,
he developed a powerful method called relative Fourier-Mukai transforms.
They are equivalences of the derived category of coherent sheaves on elliptic surfaces by using moduli spaces of stable fiber sheaves on the elliptic fibration.
Then the birational maps for higher rank cases are easily follow.
Since then, 
many properties of moduli spaces are found by Bridgeland's method (cf. \cite{BBH}, \cite{BH}, \cite{HM}, \cite{JM}, \cite{Y:7}, \cite{Y:twist2}).

In \cite{Y:elliptic}, we introduced a notion of stability for coherent sheaves 
associated to relative Fourier-Mukai transforms
and studied the wall crossing behaviors. 
In this paper, we shall generalize the notion to the case where the relative degree and the rank is not 
relatively prime, and study some properties.

Let $\pi:X \to C$ be a minimal elliptic surface over ${\Bbb C}$. 
Let 
$$
H^{*}(X,{\Bbb Q})_{\alg}:=H^0(X,{\Bbb Q}) \oplus \NS(X)_{\Bbb Q} \oplus H^4(X,{\Bbb Q})
$$
be the algebraic part of the cohomology ring $H^*(X,{\Bbb Q})$, where
$\NS(X)_{\Bbb Q}=\NS(X) \otimes_{\Bbb Z} {\Bbb Q}$.
We define $\varrho_X \in H^4(X,{\Bbb Z})$ by
$\int_X \varrho_X=1$.
Then we have
$$
H^{*}(X,{\Bbb Q})_{\alg}={\Bbb Q} \oplus \NS(X)_{\Bbb Q} \oplus {\Bbb Q}\varrho_X.
$$ 
For $u=x_0+x_1+x_2 \varrho_X$ ($x_0,x_2 \in {\Bbb Q}$, $x_1 \in \NS(X)_{\Bbb Q}$),
we set $u^\vee:=x_0-x_1+x_2 \varrho_X$.
As in \cite{ellipticStab}, we introduce Mukai pairing and Mukai vector. 
\begin{enumerate}
\item[(1)]
We define the Mukai pairing $\langle \;\;,\;\; \rangle$ on 
$H^*(X,{\Bbb Q})_{\alg}$
by
\begin{equation}
\langle u,v \rangle=-\int_X u^\vee v=(x_1 \cdot y_1)-x_0 y_2-x_2 y_0,\;
u=x_0+x_1+x_2 \varrho_X, v=y_0+y_1+y_2 \varrho_X,
\end{equation} 
where $(x_1 \cdot y_1)$ is the intersection pairing of $\NS(X)_{\Bbb Q}$.
\item[(2)]
For $E \in {\bf D}(X)$, we define the Mukai vector $v(E)$ by
\begin{equation}
v(E):=\ch E (1+\tfrac{\chi({\cal O}_X)}{2}\varrho_X) \in H^*(X,{\Bbb Q})_{\alg}.
\end{equation}
\item[(3)]
If $v \in v({\bf D}(X))$, then we also call $v$ a Mukai vector.  
\end{enumerate}

For an object $E \in {\bf D}(X)$,
we define $(f,G,H)$-semistability in Definition \ref{defn:fGH} and \ref{defn:fGH2},
where $H$ is a ${\Bbb Q}$-divisor which is $\pi$-ample with
$(H^2)=0$ and $(H \cdot f)=1$.
Roughly speaking, it means that 
$E$ or the derived dual $E^{\vee}$ is a coherent sheaf with a possible torsion fiber sheaf,
where $G \in K(X)$ with the rank $\rk G>0$.

We shall study the moduli stack ${\cal M}_{(f,H)}^G(v)$ 
of $(f,G,H)$-semistable sheaves $E$ with $v(E)=v$.
The $(f,G,H)$-semistability depends on $c_1(G)/\rk G \mod {\Bbb Q}f$.
Thus we have a parameter space 
$$
{\cal G}(v) \subset \{(s,\alpha) \mid s \in {\Bbb R}, \alpha \in 
\NS(X)_{\Bbb R} \cap f^\perp \cap H^\perp \}
$$ 
for $(f,G,H)$-semistability.
As in slope-semistability or Bridgeland semistability,
we have a wall/chamber structure in ${\cal G}(v)$ (Definition \ref{defn:wall}).
We shall prove that there is a chamber parameterizing Gieseker semistable sheaves
with respect to $H_f:=H_f$ $(n \gg 0)$ (Proposition \ref{prop:Gieseker-chamber}). 
Thus our stability is a generalization of Gieseker semistability.
If $-K_X$ is nef, we can study wall crossing behaviors. In particular we 
can compare $(f,G,H)$-semistability with Gieseker stability.   

Let
$\Phi_{X \to X'}^{{\cal P}^{\vee}}:{\bf D}(X) \to {\bf D}(X')$ be 
a relative Fourier-Mukai transform, that is,
$X'$ is a moduli space of stable sheaves supported on fibers of $\pi$
and ${\cal P} \in \Coh(X \times X')$ the universal family.
Then there is an elliptic fibration $\pi':X' \to C$ and we can associate a ${\Bbb Q}$-divisor $H'$ which is
$\pi'$-ample with $({H'}^2)=0$ and $(H' \cdot f')=1$. 
Let $\overline{\Phi}_{X \to X'}^{{\cal P}^{\vee}}$ be its cohomological Fourier-Mukai
transform (see Lemma \ref{lem:cohFM}). 
For $E \in {\bf D}(X)$, we note that
$\rk \Phi_{X \to X'}^{{\cal P}^{\vee}[k]}(E)=
(-1)^{k+1} \langle v(E),v({\cal P}_{|X \times \{ x' \}}) \rangle$ (Lemma \ref{lem:RR}).
Then we have the following result which explains that our semistability behaves nicely under relative Fourier-Mukai transforms. 

\begin{thm}[{Theorem \ref{thm:FMpreserve}}]\label{thm:main}
Let $E$ be a $(f,G,H)$-semistable object.
Assume that $G$ is general and 
$\rk \Phi_{X \to X'}^{{\cal P}^{\vee}}(E) \ne 0$. Then there is a triple $(f',G',H')$ such that 
$\Phi_{X \to X'}^{{\cal P}^{\vee}}(E)$ is $(f',G',H')$-semistable up to shift.
\end{thm}
Theorem \ref{thm:main} shows that $\Phi_{X \to X'}^{{\cal P}^{\vee}}$ induces a bijective 
correspondence of the set of chambers. 
Hence by studying chamber structure, 
we can study the behavior of Gieseker semistable objects
under relative Fourier-Mukai transforms. 
In particular we get the following result.
\begin{thm}[{cf. Theorem \ref{thm:FM2}}]\label{thm:m>>0}
Let $E$ be a Gieseker semistable sheaf with respect to $H_f$. 
Then
$\Phi_{X \to X'}^{{\cal P}^{\vee}[1]}(E(mH))$ is Gieseker semistable with respect to
$H'_{f'}$
if $m$ is sufficiently large.  
\end{thm}
By studying the wall crossing behavior, we can show that the topological invariants
do not depend on the choice of the chamber if $(K_X \cdot H) \leq 0$.
Thus we get the following result.

\begin{prop}[{cf. Proposition \ref{prop:indep}}]
Assume that $(K_X \cdot H) \leq 0$.
\begin{enumerate}
\item[(1)]
The vertual Hodge number $e({\cal M}_{(f,H)}^G(v))$
does not depend on the choice of chambers.
\item[(2)]
$\Phi_{X \to X'}^{{\cal P}^{\vee}}$ preserves virtual Hodge polynomials
$e({\cal M}_{H_f}(v))$
of the moduli of Gieseker semistable sheaves.
\end{enumerate}
\end{prop}
We remark that (2) is a consequence of (1) and Theorem \ref{thm:main}.
In particular,  
$\Phi_{X \to X'}^{{\cal P}^{\vee}}$ preserves virtual Hodge polynomials
$e({\cal M}_{H_f}(v))$
for a rational elliptic surface and an elliptic ruled surface.

\begin{NB}
\begin{prop}
Let
$\Phi_{X \to X'}^{{\cal P}^{\vee}}:{\bf D}(X) \to {\bf D}(X')$ be 
a relative Fourier-Mukai transform.
For a Mukai vector $v$ with $\rk v>0$,
assume that $v':=\Phi_{X \to X'}^{{\cal P}^{\vee}[k]}(v)$ satisfies $\rk v'>0$.
Then the virtual Hodge polynomials are the same, that is,
 $e({\cal M}_H(v)^{ss})=e({\cal M}_{H'}(v')^{ss})$, where
$H$ and $H'$ are sufficiently close to fibers. 
\end{prop}
\end{NB}

In \cite{Y:elliptic}, 
we also studied wall crossing behaviours for the moduli of $G$-twisted semi-stable 1-dimensional sheaves.
The following result explain the relation.

\begin{thm}[{Theorem \ref{thm:FMpreserve0}}]
Let $v=r+\xi+a \varrho_X$ $(r>0, \xi \in \NS(X))$ be a Mukai vector 
such that $\langle v,v({\cal P}_{|X \times \{ x' \}}) \rangle = 0$.
We take $G \in K(X)_{\Bbb Q}$ such that
$\rk G>0$ and $(f,G,H)$ is general with respect to $v$.
For an object $E \in {\bf D}(X)$ with $v(E)=v$, the following conditions are equivalent.
\begin{enumerate}
\item
$E$ is $(f,G,H)$-semistable.
\item
$F:=\Phi_{X \to X'}^{{\cal P}^{\vee}[2]}(E)$ is a $-G'$-twisted semistable 1-dimensional
sheaf with respect to $H'_{f'}$.
\end{enumerate}
\end{thm}

Thus wall-crossings for $(f,G,H)$-semistability is the same as the one for stable 1-dimensional sheaves
(see Remark \ref{rem:reduce} for a generalization). 

\begin{NB}
Then $(f,G,H)$-semistability corresponds to 
a twisted semi-stability for 1-dimensional stable sheaves
(see Proposition \ref{prop:FMstability4})..
Thus two wall-crossings in \cite{Y:elliptic} are the same. 
\end{NB}

\begin{NB}
Assume that there is no multiple fiber.
Then for a primitive Mukai vector $v_0=rf+d\varrho_X$,
$X':=M_H^G(v_0)$ is a smooth projective surface and there is a universal family
${\cal P}$ as a $(1_X \times {\alpha'}^{-1})$-twisted sheaf on $X \times X'$, where
$\alpha'$ is a 2-cocycle of ${\cal O}_{X'}^{\times}$.
Then we have a Fourier-Mukai transform
$\Phi_{X \to X'}^{{\cal P}^{\vee}}:
{\bf D}(X) \to {\bf D}^{\alpha'}(X')$.
Hence we can find a suitable $v_0$
such that $\Phi_{X \to X'}^{{\cal P}^{\vee}}(E)$ is a $G'$-twisted semistable sheaf 
for any $E \in M_{(f,H)}^G(v)$. Thus our wall-crossing is regarded as a wall-crossing 
of $G'$-twisted semistability.
\end{NB}

In \cite{ellipticStab}, we studied Bridgeland stability conditions on elliptic surfaces and their relations with 
relative Fourier-Mukai transforms.
For a chamber ${\cal C}$ of $(f,G,H)$-semistability, there is a  
relative Fourier-Mukai transform under which ${\cal C}$ is transformed to a chamber 
parameterizing Gieseker semistable sheaves (Corollary \ref{cor:Gieseker}).
Since Gieseker semistable objects are examples of stable objects in the sense of Bridgeland \cite{Br:3},
$(f,G,H)$-semistability can be understood by using Bridgeland stability conditions.
On the other hand, a Bridgeland stability $\sigma$ (\cite{Br:stability}) 
consists of a collection of semistable objects ${\cal P}_\sigma(\phi)\; (\phi \in {\Bbb R})$ with
a stability function $Z_\sigma:{\bf D}(X) \to {\Bbb C}$ which 
has much more information, and the parameter space
(that is, the space of stability conditions) is too big for our study of wall crossing
for $(f,G,H)$-semistability. 
Therefore we do not use Bridgeland stability conditions to study our $(f,G,H)$-semistability. 
For relations with Bridgeland stability, we just mention some references \cite{LLM}, \cite{Lo}, \cite{Lo-M}. 

Let us explain the organization of this paper.
In section \ref{sect:stability}, we define our stability and study some properties.
In particular we prove a Bogomolov type inequality for semistable objects.
In section \ref{sect:wall}, we introduce wall and chamber in our parameter space of stability 
conditions.
In section \ref{sect:wall-properties}, we study walls and chambers.
In particular we show that there is a chamber parameterising Gieseker semistable sheaves.
In section \ref{sect:FM}, we study the behavior of our stability under relative Fourier-Mukai 
transforms.
In section \ref{sect:example}, we give some examples.

\section{Pleliminaries}\label{sect:pre}

\subsection{Moduli of stable sheaves.}

Let $X$ be a smooth projective surface and $H$ an ample ${\Bbb Q}$-divisor on $X$.  
For $G \in K(X)_{\Bbb Q}$ with $\rk G>0$,
$G$-twisted semistability of $E$ is defined by replacing the Hilbert polynomial
$\chi(E(nH))$ by $\chi(G,E(nH))$. 
$G$-twisted semistability depends only on $\beta:=c_1(G)/\rk G$ and it is the same as the $\beta$-twisted semistability 
in the sense of Matsuki and Wentworth \cite{MW}.

\begin{defn}\label{defn:moduli}
Let $v \in H^*(X,{\Bbb Q})_{\alg}$ be a Mukai vector (of a coherent sheaf).
Let $\overline{M}_H^G(v)$ be the moduli scheme of
$S$-equivalence classes of 
$G$-twisted semi-stable sheaves $E$ with
the Mukai vector $v(E)=v$
and $M_H^\beta(v)$ the open subscheme consisting of $G$-twisted stable sheaves.
Let ${\cal M}_H^\beta(v)^{ss}$ be the moduli stack of
$G$-twisted  semi-stable sheaves $E$ with $v(E)=v$. 
For $\beta:=c_1(G)/\rk G \in \NS(X)_{\Bbb Q}$,
we also define $\overline{M}_H^\beta(v)$, $M_H^\beta(v)$ and ${\cal M}_H^\beta(v)^{ss}$ similarly.
\end{defn}

$\overline{M}_H^\beta(v)$ is a projective scheme (see \cite{MW} and \cite{Y:twist2}).
If $H$ is general in the ample cone,
$\Mod_H^\beta(v)$ and ${\cal M}_H^\beta(v)^{ss}$ are
 independent of the choice of $\beta$. 
If $v$ is primitive in the free ${\Bbb Z}$-module
$v({\bf D}(X))=v(K(X)) \subset H^*(X,{\Bbb Q})$ and $H$ is general with
respect to $v$, then $\overline{M}_H^\beta(v)=M_H^\beta (v)$.

Let $\pi:X \to C$ be an elliptic surface.
For a $\pi$-ample ${\Bbb Q}$-divisor $H$,
(twisted) Gieseker semistability with respect to $H+nf$ is independent of the choice of $n \gg 0$.
We denote the polarization $H+nf$ $(n \gg 0)$ by $H_f$. 
So ${\cal M}_{H_f}^G(v)^{ss}={\cal M}_{H+nf}^G(v)^{ss}$,
$\overline{M}^G_{H_f}(v)=\overline{M}^G_{H+nf}(v)$ and $M^G_{H_f}(v)=M^G_{H+nf}(v)$,
where $n \gg 0$.

\subsection{Fourier-Mukai transforms}
For smooth projective varieties $X, Y$ and ${\bf P} \in {\bf D}(X \times Y)$,
$\Phi_{X \to Y}^{{\bf P}}:{\bf D}(X) \to {\bf D}(Y)$ is an integral functor defined by
\begin{equation}
\Phi_{X \to Y}^{{\bf P}}(E):={\bf R}p_{Y*}({\bf P} \otimes p_X^*(E)),\; E \in {\bf D}(X),
\end{equation}
where $p_X:X \times Y \to X$ and $p_Y:X \times Y \to Y$ are projections.
If $\Phi_{X \to Y}^{{\bf P}}$ is an equivalence, then it is called a Fourier-Mukai transform.
For the equivalence $\Phi_{X \to Y}^{{\bf P}}$,
\begin{equation}
\begin{split}
\Phi_{X \to Y}^{{\bf P}} \circ \Phi_{Y \to X}^{{\bf P}^{\vee}}=& \otimes p_Y^*({\cal O}_Y(-K_Y))[-n],\\
\Phi_{Y \to X}^{{\bf P}^{\vee}} \circ \Phi_{X \to Y}^{{\bf P}}=& \otimes p_X^*({\cal O}_X(-K_X))[-n],
\end{split}
\end{equation}
where $n=\dim X=\dim Y$.

\begin{NB}
For a spherical object $E_0 \in {\bf D}(X)$,
we have an autoequivalence called spherical twist:
$$
R_{E_0}(E):={\bf R}p_{2*}(p_1^*(E) \otimes {\cal E})[1],\; E \in {\bf D}(X),
$$
where
\begin{equation}
{\cal E}:=\ker(p_1^*(E_0^{\vee}) \otimes p_2^*(E_0) \to {\cal O}_\Delta)
\end{equation}
and $p_i:X \times X \to X$ ($i=1,2$) are $i$-th projections.
\end{NB}

For a spherical object $E_0 \in {\bf D}(X)$,
let $R_{E_0}:{\bf D}(X) \to {\bf D}(X)$ be the spherical twist defined by
\begin{equation}
R_{E_0}(E):=\mathrm{cone}({\bf R}\Hom(E_0,E)\otimes E_0  \to E).
\end{equation}
If $E_0, E \in \Coh(X)$ satisfies $\Ext^2(E_0,E)=0$,
then we have an exact sequence 
$$
0 \to \Hom(E_0,E) \otimes E_0 \to E \to R_{E_0}(E) \to \Ext^1(E_0,E) \otimes E_0 \to 0.
$$

Assume that $X$ is an Enriques surface.
For an exceptional object $E_0 \in {\bf D}(X)$, 
we also have an equivalence 
$R_{E_0}:{\bf D}(X) \to {\bf D}(X)$ such that
\begin{equation}
R_{E_0}(E):=\mathrm{cone}({\bf R}\Hom(E_0,E)\otimes E_0  \oplus 
{\bf R}\Hom(E_0 (K_X),E) \otimes E_0 (K_X) \to E).
\end{equation}
Since $\langle v(E_0)^2 \rangle=-\chi(E_0,E_0)=-1$ and
$v(R_{E_0}(E))=v(E)+2\langle v(E),v(E_0) \rangle v(E_0)$,
$R_{E_0}$ indices a $(-1)$-reflection on the Mukai lattice.

For $E \in {\bf D}(X)$, we set 
\begin{equation}
D_X(E):={\bf R}{\cal H}om_{{\cal O}_X}(E,{\cal O}_X) \in {\bf D}(X).
\end{equation}
Then $D_X$ defines a contravariant functor ${\bf D}(X) \to {\bf D}(X)$. 
For simplicity, we introduce the following notation.
\begin{equation}
\begin{split}
E^*:=&H^0(D_X(E))={\cal H}om_{{\cal O}_X}(E,{\cal O}_X),\\
E^{\vee}:=&D_X(E).
\end{split}
\end{equation}
In particular the double dual $E^{**}$ of $E$ is a locally free sheaf. 
$E^*$ is the usual dual of $E$ and $E^{\vee}$ is the derived dual of $E$.

\begin{rem}
For $E \in {\bf D}(X)$,
$$
(\Phi_{X \to Y}^{{\cal P}^{\vee}[1]}(E))^{\vee}(-K_Y)={\bf R}\Hom_{p_Y} (p_X^*(E),{\cal P}[1])
$$
by Grothendieck duality.
\end{rem}

\subsection{Stable sheaves on a fiber.}\label{subsect:multiple}

Let $\pi:X \to C$ be an elliptic surface such that $R^1 \pi_* {\cal O}_X 
\not \cong {\cal O}_C$.  
\begin{NB}

\begin{lem}[{\cite[Lem. 1.17]{Y:elliptic}}]\label{lem:multiple1}
Let $mf_0$ be a fiber of $\pi$, where $m \geq 1$ is the multiplicity.
We take $E \in {\cal M}_H^\alpha (lrf_0+ld \varrho_X)^s$, where $\gcd(r,d)=1$.
\begin{enumerate}
\item[(1)]
Assume that 
$\Div E=lrf_0$.
Then $l=1$ and $E$ is an ${\cal O}_{f_0}$-module.
\item[(2)]
Assume that $\Supp E \ne f_0$ and $\Div E$ is algebraically equivalent to $lrf_0$.
Then $m \mid lr$ and $\gcd(\frac{lr}{m},ld)=1$.  
\end{enumerate}
\end{lem}

\begin{lem}[{\cite[Lem. 1.18]{Y:elliptic}}]\label{lem:multiple2}
Let $mf_0$ be a fiber with multiplicity $m$. We set $v:=(0,rf_0,d)$, where
$\gcd(r,d)=1$ and $r>0$.
If $m \nmid r$, then 
$\dim M_H^\alpha(0,rf_0,d)=1$ for a general $\alpha$.
\end{lem}

\begin{rem}[{\cite[Rem. 1.19]{Y:elliptic}}]
Let $mf_0$ be a fiber with multiplicity $m$. 
We take $E \in {\cal M}_H^\alpha (0,lrf_0,ld)^s$.
\begin{enumerate}
\item[(1)]
Assume that $m \nmid lr$ and $E$ is a locally free ${\cal O}_{f_0}$-module.
Then $l=1$ and $\Hom(E,E(K_X))=0$. In particular 
${\cal M}_H^\alpha (0,lrf_0,ld)^s$ is smooth of dimension 0 at $E$.
\item[(2)]
Assume that $m \mid lr$. Then  
${\cal M}_H^\alpha (0,lrf_0,ld)^s$ is smooth of dimension 1 at $E$.
Moreover $\Supp E \ne f_0$ for a general $E$ if $l>1$.
\end{enumerate}
\end{rem}

\begin{defn}\label{defn:multiple}
Let $mf_0$ be a multiple fiber. We set $v:=(0,rf_0,d)$, where $\gcd(r,d)=1$.
We set
$$
{\cal M}_H^\alpha(lv,lrf_0)^{ss}:=\{E \in {\cal M}_H^\alpha(lv)^{ss}
\mid \Div E=lr f_0 \}.
$$
\end{defn}

\begin{prop}[{\cite[Prop. 1.21]{Y:elliptic}}]\label{prop:multiple}
Assume that $(H,\alpha)$ is general. Then 
$\dim {\cal M}_H^\alpha(l v,lr f_0)^{ss} \leq 0$. In particular
$\dim {\cal M}_H^\alpha(l v)^{ss} \leq 0$ if $l \gcd(r,m)<m$.
\end{prop}

For the proof of this claim, we start with the following definition.

\begin{defn}\label{defn:J}
For $E_0 \in {\cal M}_H^\alpha (l_0 v)^s$,
we set
\begin{equation}
{\cal J}(l, E_0):=\{E \in {\cal M}_H(l v)^{ss} \mid 
\text{ $E$ is generated by $E_0(p K_X)$, $p \in {\Bbb Z}$ } \},
\end{equation}
where $l_0 \mid l$.
\end{defn}

\begin{lem}[{\cite[Lem. 1.23]{Y:elliptic}}]\label{lem:fiber-dim}
$\dim {\cal J}(l,E_0) \leq -1$.
\end{lem}

{\it Proof of Proposition \ref{prop:multiple}.}

We note that $E \in {\cal M}_H^\alpha(l v,lr f_0)^{ss}$ is generated by
members in ${\cal M}_H^\alpha(v,r f_0)^{ss}$ (Lemma \ref{lem:multiple1} (1)).
There are $F_i \in {\cal M}_H^\alpha(v)^s$ and 
$E_i \in {\cal J}(l_i, F_i)$ such that $E \cong \oplus_i E_i$,
where $\sum_i l_i=l$.
By Lemma \ref{lem:fiber-dim} and Lemma \ref{lem:multiple2}, 
we get our claim. For more details, see \cite[sect. 5.3]{K-Y}.
\qed

\end{NB}
Let $m_1 f_1,m_2 f_2,...,m_s f_s$ be multiple fibers of $\pi$.
For a class $r' f+d' \varrho_X$ $(r' \in {\Bbb Z}_{>0}, d' \in {\Bbb Z})$ with $\gcd(r',d')=1$,
let us consider $r_i f_i+d_i \varrho_X$ ($r_i \in {\Bbb Z}_{>0}, d_i \in {\Bbb Z}$) such that $\gcd(r_i,d_i)=1$ and
${\Bbb Q}(r_i f_i+d_i \varrho_X)={\Bbb Q}(r' f+d' \varrho_X)$.
We set $p_i:=\gcd(r_i,m_i)$.
Then $r' m_i d_i=r_i d'$,
$r_i=p_i r'$ and $m_i=p_i \frac{d'}{d_i}$, where $d_i \mid d'$.
We also have
$r' f+d' \varrho_X=\frac{d'}{d_i}(r_i f_i+d_i \varrho_X)$.
We set
\begin{equation}\label{eq:isotropic}
{\bf f}:=\sum_i l_i(r_i f_i+d_i \varrho_X)+l(r' f+d' \varrho_X)
\end{equation}
where $l_i,l \in {\Bbb Z}$ and
$0 \leq l_i <\frac{d'}{d_i}$.
Then $(l_1,...,l_s,l)$ is uniquely determined by ${\bf f}$.
\begin{lem}[{\cite[Lem. 1.24]{Y:elliptic}}]\label{lem:isotropic-dim}
$\dim {\cal M}_H^\alpha({\bf f})^{ss}=l$.
\end{lem}

\begin{lem}[{\cite[Lem. 1.25]{Y:elliptic}}]\label{lem:spherical2}
Let $D$ be an effective divisor such that $(D^2)=-2$ and $\pi(D)$ is a point.
Then ${\cal M}_H^\alpha(0,D,a)^s$ consists of a spherical object for 
a general $(H,\alpha)$. 
\end{lem}

\subsection{Mukai pairing and Fourier-Mukai transform on an elliptic surface.}\label{subsect:ellipticMukai} 

Let $\pi:X \to C$ be an elliptic surface over a curve $C$
and $f$ a fiber of $\pi$.
We denote the generic point of $C$ by $\eta$.
Then the generic fiber $\pi^{-1}(\eta)$ is a smooth projective curve with genus 1. 
We recall some results in \cite{ellipticStab}.

\begin{lem}\label{lem:RR}
For $E, F \in {\bf D}(X)$, 
\begin{equation}
-\langle v(E),v(F) \rangle=\chi(E,F(\tfrac{1}{2}K_X)).
\end{equation}
\end{lem}

\begin{defn}
For elliptic surfaces $X \to C$, $Y \to C$
and ${\bf P} \in {\bf D}(X \times Y)$, 
we set
\begin{equation}
v({\bf P}):=\ch ({\bf P})p_X^*((1+\tfrac{\chi({\cal O}_X)}{2}\varrho_X))
p_Y^*((1+\tfrac{\chi({\cal O}_Y)}{2}\varrho_Y)) \in H^*(X \times Y,{\Bbb Q}).
\end{equation}
\end{defn}
As a consequence of Grothendieck Riemann-Roch theorem, we have the following.
\begin{lem}\label{lem:cohFM}
For $\Phi_{X \to Y}^{{\bf P}^{\vee}}$,
we have a commutative diagram

\begin{equation}
\begin{CD}
{\bf D}(X) @>{\Phi_{X \to Y}^{{\bf P}^{\vee}}}>> {\bf D}(Y)\\
@V{v}VV @VV{v}V \\
H^*(X,{\Bbb Q})_{\alg} @>{\overline{\Phi}_{X \to Y}^{{\bf P}^{\vee}}}>> H^*(Y,{\Bbb Q})_{\alg}
\end{CD}
\end{equation}
where 
\begin{equation}
\overline{\Phi}_{X \to Y}^{{\bf P}^{\vee}}(\alpha):=p_{Y*}(v({\bf P}^{\vee})p_X^*(e^{-\tfrac{1}{2}K_X})p_X^*(\alpha)).
\end{equation}
\end{lem}

\begin{lem}\label{lem:fiber}
Assume that $\Phi_{X \to Y}^{{\cal P}^{\vee}}$ is a relative Fourier-Mukai transform, that is,
${\bf P}_{|X \times \{ y \}}$ $(y \in Y)$ is a stable 1-dimensional sheaf on a fiber of $\pi$. 
Then
for $E \in {\bf D}(X)$,
$$
\overline{\Phi}_{X \to Y}^{{\cal P}^{\vee}}(v(E)e^{\lambda f})=
e^{\lambda f}\overline{\Phi}_{X \to Y}^{{\cal P}^{\vee}}
(v(E)),\, \lambda \in {\Bbb Q}.
$$
\end{lem}

\begin{prop}[{\cite{ellipticStab}}]\label{prop:isometry}
Let $\Phi_{X \to Y}^{{\cal P}^{\vee}}$ be an equivalence in Lemma \ref{lem:fiber}. Then 
$\overline{\Phi}$ preserves
$\langle\;\;,\;\; \rangle$.
Thus 
$$
\langle \overline{\Phi}_{X \to Y}^{{\cal P}^{\vee}}(v(E_1)),\overline{\Phi}_{X \to Y}^{{\cal P}^{\vee}}(v(E_2)) \rangle
=\langle v(E_1),v(E_2) \rangle,\;
E_1,E_2 \in {\bf D}(X).
$$ 
\end{prop}

\section{Stability associated to an elliptic fibration.}\label{sect:stability}

\subsection{$(f,G,H)$-semistability}

Let $\pi:X \to C$ be an elliptic surface over a curve $C$
and $f$ a fiber of $\pi$.
We assume that all fibers are irreducible.
Let $H$ be a relatively ample ${\Bbb Q}$-divisor such that
$(H \cdot f)=1$.
In this section, we take an ample ${\Bbb Q}$-divisor $H+mf$ ($m \gg 0$)
for the Gieseker stability of coherent sheaves on $X$.
In order to define our stability,
we first introduce several definitions. 

\begin{defn}
\begin{enumerate}
\item[(1)]
We set
\begin{equation}
N_H:=\{D \in \NS(X)_{\Bbb Q} \mid (D \cdot f)=(D \cdot H)=0 \}.
\end{equation}
\item[(2)]
We have a orthogonal decomposition
\begin{equation}
H^*(X,{\Bbb Q})_{\alg}=({\Bbb Q}1+{\Bbb Q}H+{\Bbb Q}f+{\Bbb Q}\varrho_X) \oplus_{\perp} N_H.
\end{equation}
For $v \in H^*(X,{\Bbb Q})_{\alg}$, $D(v)$ denotes the $N_H$-component of $v$.
\end{enumerate}
\end{defn}

\begin{rem}\label{rem:mu}
We take a positive integer $\mu$ such that $\mu H \in \NS(X)$.
For $\xi \in \NS(X)$,
we set 
$$
\xi=dH+kf+D, D \in N_H.
$$ 
Then
$d=(\xi \cdot f) \in {\Bbb Z}$,
$\mu k=(\xi \cdot \mu H) \in {\Bbb Z}$.
Hence $\mu D=\mu \xi-d \mu H-k \mu f \in \NS(X)$. 
\end{rem}

We take $G \in K(X)_{\Bbb Q}$ with $\rk G>0$.

\begin{defn}
For $E \in {\bf D}(X)$ with $\rk E>0$, we set
$$
\mu_{f}(E):=\frac{(c_1(E) \cdot f)}{\rk E},\; \mu_{\beta,f}(E):=\mu_f(E(-\beta))
=\frac{((c_1(E)-\rk E \beta) \cdot f)}{\rk E},
$$
where $\beta \in \NS(X)_{\Bbb Q}$.
\end{defn}

\begin{defn}\label{defn:F+F-}
\begin{enumerate}
\item
We set
\begin{equation}
\begin{split}
{\cal F}_G^+ :=&\{A \mid \text{ $A$ is a $G$-twisted stable fiber sheaf with $\chi(G,A) \geq 0$ } \},\\
{\cal F}_G^-:=&\{A \mid \text{ $A$ is a $G$-twisted stable fiber sheaf with $\chi(G,A)<0$ } \}.
\end{split}
\end{equation}
\item
$\langle {\cal F}_G^+ \rangle$ and $\langle {\cal F}_G^- \rangle$ are subcategory
of $\Coh(X)$ generated by ${\cal F}_G^+$ and ${\cal F}_G^-$ respectively.
\end{enumerate}
\end{defn}


Let ${\cal C}_G$ be a subcategory of $\Coh(X)$ such that
$E \in {\cal C}_G$ if 
\begin{enumerate}
\item
$E_{|\pi^{-1}(\eta)}$ is a semistable vector bundle with $\mu_f (E) \geq \mu_f (G)$.
\begin{NB}
or a torsion sheaf.
\end{NB}
\item
$\Hom(A,E)=\Hom(E,B)=0$ for all $A \in {\cal F}_G^+, B \in {\cal F}_G^-$.
\end{enumerate}

Now we define our $(f,G,H)$-semistability as follows (Definition \ref{defn:fGH} and Definition \ref{defn:fGH2}).  

\begin{defn}\label{defn:fGH}
Let $E$ be a coherent sheaf with $\mu_f (E)>\mu_f (G)$.
Then $E$ is $(f,G,H)$-semistable if
$E_{|\pi^{-1}(\eta)}$ is semi-stable vector bundle and
for any subsheaf $E_1$ of $E$ with
$(c_1(E_1) \cdot f)=\frac{\rk E_1}{\rk E}(c_1(E)\cdot f)$,
\begin{enumerate}
\item
$\chi(G,E_1) <\frac{\rk E_1}{\rk E}\chi(G,E)$ or
\item 
$\chi(G,E_1) =\frac{\rk E_1}{\rk E}\chi(G,E)$ and
$(c_1(E_1) \cdot H) \leq \frac{\rk E_1}{\rk E}(c_1(E) \cdot H)$.
\end{enumerate}
\end{defn}

\begin{rem}\label{rem:fGH}
Let $E$ be a $(f,G,H)$-semistable sheaf and $E_1$ the subsheaf in
Definition \ref{defn:fGH}.   
\begin{enumerate}
\item[(1)]
If $\rk E_1=0$, then $E_1$ is a fiber sheaf with 
$\chi(G,E_1) \leq 0$.
If $\chi(G,E_1) = 0$, then $E_1$ is not 0-dimensional and
$(c_1(E_1) \cdot H) \leq 0$. By the relative ampleness of $H$, 
this case does not occur. Thus
$\chi(G,E_1) < 0$.
\item[(2)]
If $\rk E_1=\rk E$, then $E/E_1$ is a fiber sheaf and
$\chi(G,E/E_1) \geq 0$.
\item[(3)]
By (1), (2), we have $E \in {\cal C}_G$.
Moreover the torsion submodule of $E$ is purely 1-dimensional.
Hence there is a locally free resolution of length two for $E$:
$$
0 \to V_{-1} \to V_0 \to E \to 0.
$$ 
\end{enumerate}
\end{rem}

\begin{defn}\label{defn:fGH2}
Let $E \in {\bf D}(X)$ be a two-term complex of locally free sheaves.
Assume that $\mu_f (E)< \mu_f (G)$.
Then $E$ is $(f,G,H)$-semistable if $E^{\vee}(K_X)$ is a coherent sheaf which is
$(f,G^{\vee},H)$-semistable.
\end{defn}


\begin{lem}\label{lem:fGH-indep}
$(f,G,H)$-semistability is determined by $\frac{c_1(G)}{\rk G} \mod {\Bbb Q}f$
and $H$.
In particular $(f,G,H)$-semistability is the same as $(f,G(kf),H)$-semistability
$(k \in {\Bbb Q})$.
\end{lem}

\begin{proof}
For a subsheaf $E_1$ of $E$ with $(c_1(E_1) \cdot f)=\frac{\rk E_1}{\rk E}(c_1(E) \cdot f)$,
$$
\chi(G,E_1)-\frac{\rk E_1}{\rk E}\chi(G,E)=
-(c_1(G) \cdot (c_1(E_1)-\tfrac{\rk E_1}{\rk E}c_1(E)))+
\rk G \left(\chi(E_1)- \tfrac{\rk E_1}{\rk E}\chi(E) \right).
$$
Hence the claim holds.
\end{proof}


\subsection{Some properties of $(f,G,H)$-semistability.}

We first prove the following Bogomolov type inequality.

\begin{lem}\label{lem:Bogomolov}
We set $c_1(G):=\rk G(\alpha+sH)$, $\alpha \in N_H$.
Assume that $s<\frac{d}{r}$.
Let $E$ be a coherent sheaf of $v(E)=v$ such that
$E \in {\cal C}_G$.
\begin{NB}
\begin{enumerate}
\item
$E_{|\pi^{-1}(\eta)}$ is a semi-stable vector bundle,
\item
$\Hom(E,A)=0$ for any $A \in {\cal F}_G^-$
and 
\item
$\Hom(A,E)=0$ for any $A \in {\cal F}_G^+$.
\end{enumerate}
\end{NB}
Then $\langle v(E)^2 \rangle \geq -r^2 \chi({\cal O}_X)$.

In particular if $E$ is $(f,G,H)$-semistable sheaf, 
then $\langle v(E)^2 \rangle \geq -r^2 \chi({\cal O}_X)$.
\end{lem}

\begin{proof}
Let $T$ be the torsion subsheaf of $E$ and set
$E':=E/T$.
Then $T$ is a fiber sheaf.
We shall prove that $\langle v(E')^2 \rangle \leq \langle v(E)^2 \rangle$.
\begin{NB}
Let 
$$
0 \subset F_1(T) \subset F_2(T) \subset \cdots \subset F_t(T)=T
$$
be the Harder-Narasimhan filtration of $T$ with respect to $G$-twisted semi-stability.
\end{NB}
Let $A$ be a subsheaf of $T$ which is a stable factor of the first filter of the Harder-Narasimhan filtration of $T$ with respect to $G$-twisted semistability.
We set $n:=\dim \Hom(A,E)=\dim \Hom(A,T)$.
Since $\chi(G,A(K_X))=\chi(G,A)<0$ by (iii), (ii) implies $\Hom(E,A(K_X))=0$.
Then 
$$
\Ext^2(A,E)=\Hom(E,A(K_X))^{\vee}=0
$$
 and
$$
\phi:\Hom(A,T) \otimes A \to T
$$
 is injective.
In particular
$n \geq \chi(A,E)=-\langle v(A),v(E) \rangle$.
We set $T_1:=\im \phi \cong A^{\oplus n}$.
\begin{NB}
we have the Harder-Narasimhan filtration  
$$
0 \subset F_1(T)/A^{\oplus n} \subset F_2(T)/A^{\oplus n} \subset \cdots 
\subset F_t(T)/A^{\oplus n}=T/A^{\oplus n}.
$$
\end{NB}
Assume that $\langle v(A)^2 \rangle=-2$.
Then 
\begin{equation}
\begin{split}
\langle v(E/T_1)^2 \rangle=& \langle v(E)^2 \rangle-2n \langle v(A),v(E) \rangle-2n^2\\
\leq &  \langle v(E)^2 \rangle.
\end{split}
\end{equation}
Assume that $\langle v(A)^2 \rangle=0$.
Then we have $v(A)=(0,r_1 f,d_1)$, where $r_1 \in {\Bbb Q}$.
Since $0>\chi(G,A)=-\rk G(r_1 s -d_1)$,
$\langle v(A),v(E) \rangle=r_1 d-rd_1>0$.
Hence $\langle v(E/T_1)^2 \rangle< \langle v(E)^2 \rangle$.

Since $T/T_1$ is the torsion subsheaf of $E/T_1$ and $\chi(G,B)<0$ for all subsheaves $B$ of $T/T_1$,
we apply the same procedure to $T/T_1$. Then 
we get a filtration
$$
0 \subset T_1 \subset T_2 \subset \cdots \subset T_t=T
$$
such that 
$$
\langle v(E)^2 \rangle \geq \langle v(E/T_1)^2 \rangle \geq \langle v(E/T_2)^2 \rangle \geq \cdots \geq \langle v(E/T_t)^2 \rangle.
$$
Since $E'$ is torsion free, Bogomolov inequality implies $\langle v(E')^2 \rangle \geq -r^2 \chi({\cal O}_X)$
(cf. \cite[Lem. 3.3]{Y:Enriques}).
Therefore we get our claim.
\end{proof}

\begin{defn}\label{defn:T_B}
Let $B$ be a compact subset of $(-\infty,\tfrac{d}{r}) \times (N_H) \otimes_{\Bbb Q} {\Bbb R}$.
Let ${\cal T}_B(v)$ be the set of Mukai vectors 
$v_1=(0,\xi_1,a_1)$ such that $(\xi_1 \cdot H)>0$,
$(\xi_1 \cdot f)=0$,
$\langle e^{\alpha+sH},v_1 \rangle \geq 0$ for an $(s,\alpha) \in B$,
$\langle (v-v_1)^2 \rangle \geq -r^2 \chi({\cal O}_X)$.
\end{defn}

\begin{lem}\label{lem:bddv}
${\cal T}_B(v)$ is a finite set.
\end{lem}


\begin{proof}
We set
$s_0:=\max \{ s \mid (s,\alpha) \in B \}$.
For $v_1=(0,\xi_1,a_1) \in {\cal F}_B(v)$,
\begin{equation}
\langle (v-v_1)^2 \rangle=\langle v^2 \rangle+(\xi_1^2)-2(\xi \cdot \xi_1)+2ra_1.
\end{equation}
We set 
$$
\xi_1=k_1 f+D_1, \quad k_1 \in {\Bbb Q},\; D_1 \in N_H.
$$
By
$0 \leq \langle e^{\alpha+sH},v_1 \rangle=((\alpha+s H)\cdot (k_1 f+D_1))-a_1$,
$a_1-sk_1-(\alpha \cdot D_1) \leq 0$.
Since $s \leq s_0 <\frac{d}{r}$ and $k_1=(\xi_1 \cdot H)>0$, we get
$a_1-\frac{d}{r}k_1<(\alpha \cdot D_1)$.
Hence
\begin{equation}\label{eq:bddv}
\begin{split}
-r^2 \chi({\cal O}_X)-\langle v^2 \rangle \leq &
(\xi_1^2)-2(\xi \cdot \xi_1)+2r a_1\\
=& (D_1^2)-2(D \cdot D_1)-2dk_1+2ra_1\\
<& 2r(\alpha \cdot D_1)-2(D \cdot D_1)
+(D_1^2)\\
=& ((D_1+r\alpha-D)^2)-((r\alpha-D)^2).
\end{split}
\end{equation}
Since $N_H$ is negative definite and $B$ is bounded,
the set of $D_1$ is bounded, which implies the choice of $D_1$ is finite.
Then the choice of $dk_1-ra_1$ is also finite by \eqref{eq:bddv}.
Since 
$$
0<(d-rs_0)k_1 \leq (d-rs)k_1 \leq r(\alpha \cdot D_1)-(ra_1-dk_1),
$$
the choice of $(k_1,a_1)$ is also finite.
\end{proof}

\begin{lem}\label{lem:F^-}
Let $E$ be a coherent sheaf such that $E_{|\pi^{-1}(\eta)}$ is a semi-stable vector bundle,
$\mu_f(E) \geq \mu_f(G)$ and
$\Hom(A,E)=0$ for all $A \in {\cal F}_G^+$.  
Then there is an exact sequence
$$
0 \to E_1 \to E \to E_2 \to 0
$$ 
such that $\Hom(E_1, A)=0$ for 
all $A \in \langle {\cal F}_{G}^- \rangle$
and $E_2 \in \langle {\cal F}_{G}^- \rangle$.
\end{lem}

\begin{proof}
Assume that there is a sequence of subsheaves 
\begin{equation}
F_n \subset F_{n-1} \subset \cdots \subset F_1 \subset F_0=E
\end{equation}
such that $F_{i-1}/F_i$ are fiber sheaves with
$\chi(G,F_{i-1}/F_i)<0$.
Let $T_n$ be the torsion subsheaf of $F_n$.
Then $\chi(G,T_n)<0$ and 
$$
v(E)=\sum_{i=1}^n v(F_{i-1}/F_i)+v(T_n)+v(F_n/T_n).
$$
By Lemma \ref{lem:bddv},
the choice of $\sum_{i=1}^n v(F_{i-1}/F_i)+v(T_n)$ is finite.

Assume that there is a non-zero homomorphism $\phi:F_n \to A$ $(A \in {\cal F}_G^-)$.
We set $F_{n+1}=\ker \phi$. Then
$F_n/F_{n+1}$ is a fiber sheaf with $\chi(G,F_n/F_{n+1})<0$. 
Therefore there is $F_n$ such that
$\Hom(F_n,A)=0$ for all $A \in {\cal F}_G^-$.
\end{proof}

\begin{lem}\label{lem:F^+}
Let $E$ be a coherent sheaf such that 
$E_{|\pi^{-1}(\eta)}$ is a semistable vector bundle.
For a locally free sheaf $G$ with $\mu_f (E) \geq \mu_f (G)$,
there is an exact sequence
$$
0 \to E_1 \to E \to E_2 \to 0
$$
such that $E_1 \in \langle {\cal F}_{G}^+ \rangle$
and
$\Hom(A,E_2)=0$ for all $A \in \langle {\cal F}_{G}^+ \rangle$.
\end{lem}

\begin{proof}
Let $T$ be the torsion submodule of $E$.
Then there is a subsheaf $E_1$ of $T$ such that
$E_1 \in \langle {\cal F}_G^+ \rangle$ and $T/E_1 \in \langle {\cal F}_G^- \rangle$.
Then $E_2:=E/E_1$ satisfies $\Hom(A,E/E_2)=0$ for all $A \in \langle {\cal F}_{G}^+ \rangle$.
\end{proof}
\begin{rem}
In the notation of \cite[Defn. 3.3.3]{PerverseII},
$T \in \widehat{\frak T}, F \in \widehat{\frak F}$.
\end{rem}

\begin{lem}\label{lem:C_G}
For $E \in {\cal C}_G$,
let $E_1$ be a subsheaf of $E$.
Then there is a subsheaf $F$ of $E$ such that
\begin{enumerate}
\item
$F \in {\cal C}_G$ and
$E/F \in {\cal C}_G$,
\item
$F_{|\pi^{-1}(\eta)}=E_{1|\pi^{-1}(\eta)}$ in
$E_{|\pi^{-1}(\eta)}$,
\item
$\chi(G,E_1) \leq \chi(G,F)$.
\end{enumerate}
\end{lem}

\begin{proof}
We have a subsheaf $E_1'$ of $E_1$ such that
$\Hom(E_1',A)=0$ for all $A \in {\cal F}_G^-$ and 
$E_1/E_1' \in \langle {\cal F}_G^- \rangle$ by Lemma \ref{lem:F^-}.
By Lemma \ref{lem:F^+},
we have a subsheaf  $T$ of $E/E_1'$ such that
$T \in  \langle {\cal F}_G^+ \rangle$ and
$\Hom(A,(E/E_1')/T)=0$ for all $A \in  \langle {\cal F}_G^+ \rangle$.
Since $E \in {\cal C}_G$, $\Hom((E/E_1')/T,B)=0$ for all
$B \in {\cal F}_G^-$.  Hence we get $(E/E_1')/T \in {\cal C}_G$.
Let $F$ be a subsheaf of $E$ such that $E_1' \subset F$ and $F/E_1'=T$.
\begin{NB}
Let $E$ be a coherent sheaf fitting in an exact sequence
$$
0 \to E_1 \to E \to T \to 0
$$
such that $\Hom(E_1,A)=0$ for any $G$-twisted stable fiber sheaf
with $\chi(G,A)<0$ and $T \in \langle {\cal F}_G^+ \rangle$.
Then $\Hom(E,A)=0$ for any $G$-twisted stable fiber sheaf
with $\chi(G,A)<0$. 
\end{NB}

Then $F \in {\cal C}_G$,
$E/F \in {\cal C}_G$ and $F_{|\pi^{-1}(\eta)}=E_{1|\pi^{-1}(\eta)}$ in
$E_{|\pi^{-1}(\eta)}$.
Since 
$$
\chi(G,F)=\chi(G,E_1')+\chi(G,T)=\chi(G,E_1)-\chi(G,E_1/E_1')+\chi(G,T),
$$
we also get 
$\chi(G,F) \geq \chi(G,E_1)$.
\end{proof}

\begin{rem}
$F$ is uniquely determined:
Let $F'$ be a subsheaf of $E$ satisfying (i), (ii), (iii) in Lemma \ref{lem:C_G}.
Since $F_{|\pi^{-1}(\eta)}=E_{1|\pi^{-1}(\eta)}$ in
$E_{|\pi^{-1}(\eta)}$,
$\phi:F' \to E \to E/F$ is a 0-map for a general fiber of $\pi$.
Then we have $\phi(F') \in \langle {\cal F}_G^+ \rangle$. 
Since $\Hom(A,E/F)=0$ for $A \in {\cal F}_G^+$, $F' \to E \to E/F$ is a 0-map. 
Hence $F' \subset F$. We also see that $F \subset F'$. Therefore $F=F'$.
\end{rem}

\begin{lem}\label{lem:EF}
Let $E \in {\cal C}_G$ and $F \in {\cal C}_G$ be $(f,G,H)$-semistable sheaves
such that 
$$
\rk E>0, \rk F>0,\; 
\mu_f (E)=\mu_f (F),\;
\frac{\chi(G,E)}{\rk E}>\frac{\chi(G,F)}{\rk F}.
$$ 
Then $\Hom(E,F)=0$.
\end{lem}

\begin{proof}
For a nonzero homomorphism $\phi:E \to F$,
we have
$$
\frac{\chi(G,E)}{\rk E} \leq \frac{\chi(G,\phi(E))}{\rk \phi(E)} \leq
\frac{\chi(G,F)}{\rk F}
$$
if $\rk \phi(E)>0$.
Hence $\phi(E)$ is a torsion sheaf.
Then 
$$
\rk \phi(E) \frac{\chi(G,E)}{\rk E} \leq \chi(G,\phi(E)) \leq
\rk \phi(E) \frac{\chi(G,F)}{\rk F}.
$$
By Remark \ref{rem:fGH} (1),
we get a contradiction. Therefore $\phi=0$.  
\end{proof}

\begin{lem}\label{lem:EF2}
Assume that $(K_X \cdot H) \leq 0$.
Let $E \in {\cal C}_G$ and $F \in {\cal C}_G$ be $(f,G,H)$-semistable sheaves
such that 
$$
\rk E>0, \rk F>0,\; 
\mu_f(E)=\mu_f(F)>\mu_f(G),\;
\frac{\chi(G,E)}{\rk E} \geq \frac{\chi(G,F)}{\rk F}.
$$ 
Then
$\Ext^2(F,E)=0$, if $(K_X \cdot H)<0$
or $\frac{\chi(G,E)}{\rk E}>\frac{\chi(G,F)}{\rk F}$.
\end{lem}

\begin{proof}
Since $\mu_f(F) > \mu_f(G)$, we get
$$
\chi(G,F(K_X))=\chi(G,F)+\rk G \rk F ((\tfrac{c_1(F)}{\rk F}-\tfrac{c_1(G)}{\rk G}) \cdot K_X) \leq \chi(G,F).
$$
Moreover the inequality is strict if $(K_X \cdot H)<0$.
By Lemma \ref{lem:EF}, $\Ext^2(F,E)=\Hom(E,F(K_X))^{\vee}=0$.
\end{proof}

\begin{defn}
Let ${\cal S}(v)$ be the set of pairs of coherent sheaves $(E_1,E)$ such that
\begin{enumerate}
\item
$E \in {\cal C}_G$ with $v(E)=v$
\item
$E_1$ is a subsheaf of $E$ with $((\rk E_1 c_1(E)-\rk E c_1(E_1))\cdot f)=0$
\item
$E_1, E/E_1 \in {\cal C}_G$.
\end{enumerate} 
\end{defn}

\begin{lem}\label{lem:wall1}
$$
\# \{ (\langle v_1^2 \rangle,D(v_1)) \mid v_1=v(E_1), (E_1,E) \in {\cal S}(v) \} <\infty.
$$
\end{lem}


\begin{proof}
For $(E_1,E) \in {\cal S}(v)$,
we set 
\begin{equation}
v_1:=v(E_1)=r_1+\xi_1+a_1 \varrho_X, v_2:=v-v_1=r_2+\xi_2+a_2 \varrho_X,\; \xi_1,\xi_2 \in \NS(X).
\end{equation}
Since
$$
\frac{\langle v_i^2  \rangle}{r_i^2}=\frac{(\xi_i^2)}{r_i^2}-2\frac{a_i}{r_i},\;i=1,2,
$$
we see that
\begin{equation}
\begin{split}
\frac{\langle v^2 \rangle}{r_1r_2}
&=\frac{\langle v_1^2 \rangle+\langle v_2^2 \rangle+
2\langle v_1,v_2 \rangle}{r_1r_2}\\
&=\frac{1}{r_1 r_2}\langle v_1^2 \rangle+
\frac{1}{r_1r_2}\langle v_2^2 \rangle+
2\left(\frac{\xi_1}{r_1},\frac{\xi_2}{r_2} \right)-
2\left(\frac{a_1}{r_1}+\frac{a_2}{r_2} \right)\\
&=-\left(\left(\frac{\xi_1}{r_1}-\frac{\xi_2}{r_2} \right)^2 \right)+
\left(\frac{r_1+r_2}{r_1^2r_2}\langle v_1^2 \rangle+
\frac{r_1+r_2}{r_1r_2^2}\langle v_2^2 \rangle \right).
\end{split}
\end{equation}
Hence 
\begin{equation}\label{eq:Bog}
\begin{split}
\frac{\langle v^2 \rangle+r^2 \chi({\cal O}_X)}{r_1r_2}
&=-\left(\left(\frac{\xi_1}{r_1}-\frac{\xi_2}{r_2} \right)^2 \right)+
\left(\frac{r}{r_1^2r_2}\langle v_1^2 \rangle+
\frac{r}{r_1r_2^2}\langle v_2^2 \rangle \right)+\frac{r^2}{r_1 r_2}\chi({\cal O}_X)\\
&=-\left(\left(\frac{\xi_1}{r_1}-\frac{\xi_2}{r_2} \right)^2 \right)+
\frac{r}{r_1 r_2} \left(\frac{\langle v_1^2 \rangle+r_1^2 \chi({\cal O}_X)}{r_1}+
\frac{\langle v_2^2 \rangle+r_2^2 \chi({\cal O}_X)}{r_2} \right).
\end{split}
\end{equation}
We set $\xi_i:=d_i H+k_i f+D_i$, $D_i \in N_H$.
Then we have
$$
\left(\left(\frac{\xi_1}{r_1}-\frac{\xi_2}{r_2}\right)^2 \right)=\frac{((rD_1-r_1 D(v))^2)}{r_1^2 r_2^2}.
$$
Since $N_H$ is negative definite,
by the Bogomolov inequality,
we see that the choice of $(\langle v_1^2 \rangle,D_1)$ is finite.
\end{proof}

\begin{NB}
$\frac{\xi_1}{r_1}-\frac{\xi_2}{r_2}=\frac{r\xi_1-r_1 \xi}{r_1 (r-r_1)}$.
\end{NB}

\begin{rem}
By the proof of Lemma \ref{lem:wall1} (see \eqref{eq:Bog} and 
Lemma \ref{lem:Bogomolov}), we also have 
$$
0 \leq \langle v_1^2 \rangle+r_1^2 \chi({\cal O}_X) \leq \frac{r_1}{r}(\langle v^2 \rangle+r^2 \chi({\cal O}_X)).
$$
\end{rem}

\begin{lem}\label{lem:max}
For $E \in {\cal C}_G$, there is a subsheaf $F$ of $E$ such that
$\rk F >0$, $\mu_f(F)=\mu_f(E)$ and
$$
\chi(G,F') \leq \frac{\rk F'}{\rk F}\chi(G,F)
$$
for any subsheaf $F'$ of $E$ with $(c_1(F') \cdot f)=\rk F' \mu_f(E)$.
\end{lem}

\begin{proof}
We may assume that 
$$
v(G(\tfrac{1}{2}K_X))=1+(xH+\alpha), \alpha \in N_H.
$$
Then 
\begin{equation}\label{eq:(G,F')}
\chi(G,F')=-\langle 1+(xH+\alpha),v(F') \rangle.
\end{equation}
If $E$ is $(f,G,H)$-semistable, then $F=E$. So
we may assume that $E$ is not $(f,G,H)$-semistable.
Let $F'$ be a subsheaf of $E$ such that
$$
\chi(G,F') > \frac{\rk F'}{\rk E}\chi(G,E).
$$
By Lemma \ref{lem:C_G}, we may assume that $F',E/F' \in {\cal C}_G$.

We shall prove that $\chi(G,F')$ is bounded above.
We set
$$
v(F'):=r_1+d_1 H+k_1 f+D_1+a_1 \varrho_X, \; D_1 \in N_H.
$$
We may assume that $v(G)=\rk G e^{xH+\alpha}$ and
$\alpha \in N_H$.
By \eqref{eq:(G,F')},
\begin{equation}
\chi(F' (-xH-\alpha))=
a_1-k_1 x-(\alpha \cdot D_1)=
\frac{(D_1^2)-\langle v_1^2 \rangle}{2r_1}+\frac{(d_1-r_1 x)}{r_1} k_1-(\alpha \cdot D_1).
\end{equation}
We note that $d_1-r_1 x>0$.
Since the choice of $(\langle v(F')^2 \rangle,D_1)$ is finite by Lemma \ref{lem:wall1}, 
it is sufficient to show that $k_1$ is bounded above.

We take a ${\Bbb Q}$-ample divisor $H+mf$.
By the boundedness of Grothendieck,
$$
\{(c_1(F') \cdot (H+mf)) \mid F' \subset E \}
$$
is bounded above.
Hence
$$
\{(c_1(F') \cdot H) \mid F' \subset E, ((\rk E c_1(F')-\rk F c_1(E)) \cdot f)=0 \}
$$
is also bounded above.
Therefore $\chi(G,F')$ is bounded above.

We take $(F,E) \in {\cal S}(v)$ such that $\rk F>0$ and
$$
\frac{\chi(G,F')}{\rk F'} \leq \frac{\chi(G,F)}{\rk F}
$$
for all $(F',E) \in {\cal S}(v)$ with $\rk F'>0$.
Then $F$ satisfies the required property.
\end{proof}

\begin{prop}\label{prop:HNF}
For $E \in {\cal C}_G$, 
there is a filtration  
\begin{equation}\label{eq:HNF}
0 \subset F_1 \subset F_2 \subset \cdots \subset F_s = E
\end{equation}
such that
\begin{enumerate}
\item
$E_i:=F_i/F_{i-1}$ are $(f,G,H)$-semistable for $0< i <s$,
\item
$$
\frac{(c_1(E_i) \cdot f)}{\rk E_i}=\frac{(c_1(E_{i+1} \cdot f)}{\rk E_{i+1}},\;
\frac{\chi(G,E_i)}{\rk E_i}>\frac{\chi(G,E_{i+1})}{\rk E_{i+1}}
$$
 for
$0<i<s$.
\end{enumerate}
\end{prop}

\begin{proof}
We take a subsheaf $F$ of $E$ such that
$$
\frac{\chi(G,F)}{\rk F} \geq \frac{\chi(G,F')}{\rk F'}
$$
for any subsheaf $F'$ of $E$ (Lemma \ref{lem:max}).
We may assume that $\rk F$ is maximal.
By Lemma \ref{lem:C_G}, $F \in {\cal C}_G$ and $E/F \in {\cal C}_G$.
Moreover $F$ is $(f,G,H)$-semistable.
We set $F_1:=F$. Applying the same procedure to $E/F_1$, we have a required filtration for
$E/F_1$. 
So we get a desired filtration for $E$. 
\end{proof}

\begin{defn}\label{defn:HNF}
Let $E$ be a coherent sheaf such that $E_{|\pi^{-1}(\eta)}$ is $\mu$-semistable
and $(c_1(G^{\vee} \otimes E) \cdot f) \geq 0$.
The Harder-Narasimhan filtration of $E$ is the filtration
\begin{equation}
0 \subset F_0 \subset F_1 \subset F_2 \subset \cdots \subset F_s \subset E
\end{equation}
such that 
\begin{enumerate}
\item
$F_0 \in \langle {\cal F}_G^+ \rangle$, $E/F_s \in \langle {\cal F}_G^- \rangle$
and 
\item
$$
0 \subset F_1/F_0 \subset F_2/F_0 \subset \cdots \subset F_s/F_0 
$$
is the filtration \eqref{eq:HNF} for $F_s/F_0 \in {\cal C}_G$. 
\end{enumerate}
\end{defn}

\begin{defn}\label{defn:fGHstack}
Let ${\cal M}_{(f,H)}^G(v)^{ss}$ be the stack of $(f,G,H)$-semistable objects $E$ with $v(E)=v$, where 
$\rk v>0$.
For the Mukai vector $v$ of a fiber sheaf,
we set ${\cal M}_{(f,H)}^G(v)^{ss}:={\cal M}_{H_f}^G(v)^{ss}$.
\end{defn}
.

\section{Walls and chambers for $(f,G,H)$-semistability.}\label{sect:wall}

\subsection{Dependence on fiber sheaves.}

Let $B$ be a compact subset of $(-\infty,\tfrac{d}{r}) \times (N_H) \otimes_{\Bbb Q} {\Bbb R}$.
We take $G_{s,\alpha} \in K(X) \otimes_{\Bbb Q} {\Bbb R}$ with 
$c_1(G_{s,\alpha})=\rk G_{s,\alpha}(sH+\alpha)$.
 
\begin{defn}\label{defn:hatF+F-}
\begin{enumerate}
\item
We set
\begin{equation}
\begin{split}
\widehat{\cal F}_G^+ :=&\{A \mid \text{ $A$ is a $G$-twisted stable fiber sheaf with $\chi(G,A) > 0$ } \},\\
\widehat{\cal F}_G^-:=&\{A \mid \text{ $A$ is a $G$-twisted stable fiber sheaf with $\chi(G,A) \leq 0$ } \}.
\end{split}
\end{equation}
\item
$\langle \widehat{\cal F}_G^+ \rangle$ and $\langle \widehat{\cal F}_G^- \rangle$ are subcategory
of $\Coh(X)$ generated by $\widehat{\cal F}_G^+$ and $\widehat{\cal F}_G^-$ respectively.
\end{enumerate}
\end{defn}

\begin{lem}
Let $E$ be a coherent sheaf with $v(E)=v$ such that
\begin{enumerate}
\item
$E_{|\pi^{-1}(\eta)}$ is a semi-stable vector bundle,
\item
$\Hom(A,E)=0$ for any $A \in \widehat{\cal F}_{G_{s_0,\alpha_0}}^+$.
\end{enumerate}
Let $T$ be the torsion subsheaf of $E$.
Then $v(T) \in {\cal T}_B(v)$ (see Definition \ref{defn:T_B}).
In particular the choice of $v(T)$ is finite.
\end{lem}

\begin{proof}
Since $E/T$ is torsion free, $\langle v(E/T)^2 \rangle \geq -r^2\chi({\cal O}_X)$.
By (ii), $T$ is purely 1-dimensional and $\chi(G_{s_0,\alpha_0},T) \leq 0$.
Hence $v(T) \in {\cal T}_B(v)$.
By Lemma \ref{lem:bddv}, $\# {\cal T}_B(v)<\infty$.
\end{proof}

\begin{defn}
Let ${\cal U}_B(v)$ be the set of 
Mukai vectors $u=(0,\eta,b)$
such that
\begin{enumerate}
\item
$\langle u^2 \rangle=0,-2$,
\item
$\langle e^{\alpha+sH},u \rangle=0$ for some $(s,\alpha) \in B$,
\item
there is a Mukai vector $(0,\xi_1,a_1) \in {\cal T}_B(v)$ and
$\xi_1-\eta$ and $\eta$ are represented by effective divisors.
\end{enumerate}
\end{defn}

\begin{rem}
Assume that $\langle u^2 \rangle=0$.
Then $\langle v,u \rangle>0$. 

Assume that $\langle u^2 \rangle=-2$.
If $\langle v,u \rangle<0$, then $u$ defines a totally semistable wall.
If $\langle v,u \rangle \geq 0$, then
$\langle v-u,u \rangle \geq 2$. If there is a stable sheaf with the Mukai vector 
$v-u$, then $u$ defines a wall.   
\end{rem}

\begin{lem}
${\cal U}_B(v)$ is a finite set.
\end{lem}

\begin{proof}
Since ${\cal T}_B(v)$ is a finite set, 
the choice of $\xi_1$ is finite, and hence the choice of 
$\eta$ is also finite.
Since $B$ is bounded,
$b-(\eta \cdot(\alpha+sH))=0$ implies
the choice of $b$ is also finite.
Therefore ${\cal U}_B(v)$ is a finite set.
\end{proof}

Let us consider the open subset 
$$
B_1:=\{(s,\alpha) \in B \mid \text{$\langle e^{\alpha+sH},u \rangle \ne 0$ for all $u \in {\cal U}_B(v)$ } \}.
$$
Let $B_1^0$ be a connected component of $B_1$.

\begin{lem}
We take $(s_0,\alpha_0) \in B_1^0$.
Assume that a coherent sheaf $E$ satisfies
\begin{enumerate}
\item
$E_{|\pi^{-1}(\eta)}$ is a semi-stable vector bundle,
\item
$\Hom(A,E)=0$ for any $A \in \widehat{\cal F}_{G_{s_0,\alpha_0}}^+$.
\end{enumerate}
Then
$\Hom(A,E)=0$ for any $A \in \widehat{\cal F}_{G_{s,\alpha}}^+$ and $(s,\alpha) \in B_1^0$.
\end{lem}

\begin{proof}
We take $(s,\alpha) \in B_1^0$ and
assume that $\Hom(A,E) \ne 0$ with 
$A \in \widehat{\cal F}_{G_{s,\alpha}}^+$.
By taking its stable factor, we may assume that $A$ is a subsheaf of $E$.
Let $I$ be a segment connecting $(s_0,\alpha_0)$ and $(s,\alpha)$.
Since $\chi(G_{s_0,\alpha_0},A)<0$ by (ii),
there is $(s_1,\alpha_1) \in I$ such that $\chi(G_{s_1,\alpha_1},A) = 0$.
Assume that there is a subsheaf $A_1 \subset A$ such that $\chi(G_{s_1,\alpha_1},A_1) > 0$.
Then we have $0<(c_1(A_1) \cdot H)<(c_1(A) \cdot H)$. 
Since $\chi(G_{s_0,\alpha_0},A_1)<0$, we can apply the same argument to the subsheaf $A_1$.
Thus there is $(s_2,\alpha_2) \in I$ such that
$\chi(G_{s_2,\alpha_2},A_1)=0$. 
If there is a subsheaf $A_2 \subset A_1$ such that $\chi(G_{s_2,\alpha_2},A_2) > 0$, then 
we continue the same procedure, and we finally get a subsheaf $A_n \subset E$ and
$(s_{n+1},\alpha_{n+1}) \in I$ 
such that $A_n$ is a $G_{s_{n+1},\alpha_{n+1}}$-semistable sheaf with 
$\chi(G_{s_{n+1},\alpha_{n+1}},A_n)=0$.
\begin{NB}
We have an exact sequence
\begin{equation}
0 \to T' \to E \to E/T' \to 0
\end{equation}
such that $\Hom(A,E/T')=0$ for any semi-stable fiber sheaf $A$ with $\chi(G_{s_n,\alpha_n},A) \geq 0$
and $T'$ is generated by semi-stable fiber sheaves $A$ with $\chi(G_{s_n,\alpha_n},A) \geq 0$.
Then $A_n$ is a subsheaf of $T'$.
We have $\langle v(E/T')^2 \rangle \geq -r^2 \chi({\cal O}_X)$ and 
$\langle e^{\alpha_0+s_0 H},v(T') \rangle \geq 0$.
Hence $v(T') \in {\cal T}_B(v)$, which implies
$v(A_n') \in {\cal U}_B(v)$, where $A_n'$ is a stable factor of $A_n$.
\end{NB}
Since $A_n$ is a subsheaf of $T$ and $v(T) \in {\cal T}_B(v)$, we get
$v(A_n') \in {\cal U}_B(v)$, where $A_n'$ is a stable factor of $A_n$.
Then $(s_{n+1},\alpha_{n+1}) \not \in B_1$. 
Therefore 
$\Hom(A,E)=0$ for any $A \in \widehat{\cal F}_{G_{s,\alpha}}^+$.
\end{proof}

\begin{lem}
Assume that a coherent sheaf $E$ satisfies
\begin{enumerate}
\item
$E_{|\pi^{-1}(\eta)}$ is a semi-stable vector bundle,
\item
$\Hom(A,E)=0$ for any $A \in \widehat{\cal F}_{G_{s_0,\alpha_0}}^+$ and
$\Hom(E,A)=0$ for any $A \in {\cal F}_{G_{s_0,\alpha_0}}^-$.
\end{enumerate}
Then
$\Hom(A,E)=0$ $(A \in \widehat{\cal F}_{G_{s,\alpha}}^+)$ and
$\Hom(E,A)=0$ $(A \in {\cal F}_{G_{s_0,\alpha_0}}^-)$ for all $(s,\alpha) \in B_1^0$.
\end{lem}

\begin{proof}
We take $(s,\alpha) \in B_1^0$ and
assume that there is a quotient $\phi:E \to A$ such that
$\chi(G_{s,\alpha},A) \leq 0$.
Let $I$ be a segment connecting $(s_0,\alpha_0)$ and $(s,\alpha)$.
Since $\chi(G_{s_0,\alpha_0},A)>0$,
there is $(s_1,\alpha_1) \in I$ such that $\chi(G_{s_1,\alpha_1},A) = 0$.
If there is a quotient $A \to A_1$ such that $\chi(G_{s_1,\alpha_1},A_1) < 0$,
then we replace $\phi$ by the quotient $\phi_1:E \to A_1$.
Continuing this procedure, we finally get $(s_n,\alpha_n) \in I$ and
a quotient $\phi_n:E \to A_n$
such that $A_n$ is a $G_{s_n,\alpha_n}$-semistable sheaf with $\chi(G_{s_n,\alpha_n},A_n)=0$.
We have an exact sequence
\begin{equation}
0 \to E' \to E \to C \to 0
\end{equation}
such that $\Hom(E',A)=0$ for all $A \in \widehat{\cal F}_{G_{s_n,\alpha_n}}^-$
and $C \in \langle \widehat{\cal F}_{G_{s_n,\alpha_n}}^- \rangle$.
Then $A_n$ is a quotient of $C$.
We have $\langle v(E')^2 \rangle \geq -r^2 \chi({\cal O}_X)$ and $\langle e^{\alpha+sH},v(C) \rangle \geq 0$.
Hence $v(C) \in {\cal T}_B(v)$, which implies
$v(A_n') \in {\cal U}_B(v)$, where $A_n'$ is a stable factor of $A_n$.
Then $(s_n,\alpha_n) \not \in B_1$. 
Therefore 
$\Hom(E,A)=0$ for any semi-stable fiber sheaf $A$ with $\chi(G_{s,\alpha},A) < 0$.
\end{proof}

\subsection{Wall and chamber}

We shall study the dependence of $(f,G,H)$-semistability on $G$.
By the definition of $(f,G,H)$-semistability (Definition \ref{defn:fGH2}),
we may assume that $\mu_f (E) > \mu_f (G)$.
We note that

\begin{equation}\label{eq:DD_1}
\begin{split}
&-(D^2)+\langle v^2 \rangle+r^2 \chi({\cal O}_X)\\
=& r \left(\frac{-(D_1^2)+\langle v_1^2 \rangle+r_1^2 \chi({\cal O}_X)}{r_1}+
\frac{-(D_2^2)+\langle v_2^2 \rangle+r_2^2 \chi({\cal O}_X)}{r_2} \right).
\end{split}
\end{equation}

\begin{lem}
$$
|r_1^2 ((D^2)-\langle v^2 \rangle)-
r^2((D_1^2)-\langle v_1^2 \rangle)| \leq  r^2 (-(D^2)+\langle v^2 \rangle+r^2 \chi({\cal O}_X)).
$$
\end{lem}

\begin{proof}
We note that
\begin{equation}
r_1^2 ((D^2)-\langle v^2 \rangle)-
r^2((D_1^2)-\langle v_1^2 \rangle)=
r_1^2 ((D^2)-\langle v^2 \rangle-r^2 \chi({\cal O}_X) )-
r^2((D_1^2)-\langle v_1^2 \rangle-r_1^2 \chi({\cal O}_X)).
\end{equation}

By \eqref{eq:DD_1},
$$
0 \leq -(D_1^2)+\langle v_1^2 \rangle+r_1^2 \chi({\cal O}_X) \leq 
\frac{r_1}{r} (-(D^2)+\langle v^2 \rangle+r^2 \chi({\cal O}_X)).
$$
Then it is easy to see the inequality.

\begin{NB}
\begin{equation}
r_1^2 ((D^2)-\langle v^2 \rangle)-
r^2((D_1^2)-\langle v_1^2 \rangle) \geq r_1^2 ((D^2)-\langle v^2 \rangle-r^2 \chi({\cal O}_X))
\end{equation}
\end{NB}
\end{proof}

\begin{defn}\label{defn:G(v)}
For a Mukai vector 
$$
v=r+dH+kf+D+a \varrho_X, D \in N_H,
$$
we set
\begin{equation}
{\cal G}(v):=\{(s,\alpha) \mid s \in {\Bbb R}, rs-d \ne 0, \alpha \in (H_H) \otimes_{\Bbb Q} {\Bbb R} \}. 
\end{equation}
\end{defn}

\begin{defn}\label{defn:wall}
\begin{enumerate}
\item
[(1)]
Let ${\cal U}(v)$ be the set of Mukai vectors $u=(0,\eta,b)$ such that
\begin{enumerate}
\item
$\langle u^2 \rangle=0,-2$,
\item
$\eta$ is effective with $(\eta \cdot f)=0$, and
\item
there is a Mukai vector $u'=(0,\eta',b')$ such that
$\eta'$ is effective with 
$(\eta' \cdot f)=0$,
there is $(s,\alpha) \in {\cal G}(v)$ with
$\langle e^{sH+\alpha},u \rangle=\langle e^{sH+\alpha},u' \rangle=0$ 
and $\langle (v-u-u')^2 \rangle \geq -r^2 \chi({\cal O}_X)$. 
\end{enumerate}
\item
[(2)]
Let $U(v)$ be a subset of $H^*(X,{\Bbb Q})_{\alg}$ consisting of Mukai vectors 
$$
v_1=r_1+d_1 H+k_1 f+D_1+a_1 \varrho_X,\; D_1 \in N_H
$$
such that 
\begin{enumerate}
\item
$r v_1-r_1 v \ne 0$.
\item
$0 \leq r_1 < r$, $d_1=r_1 \frac{d}{r}$, 
\item
If $0<r_1<r$, then
$\langle v_1^2 \rangle \geq -r_1^2 \chi({\cal O}_X)$ and
$\langle (v-v_1)^2 \rangle \geq -(r-r_1)^2 \chi({\cal O}_X)$,
\item
If $r_1=0$, then $v_1 \in {\cal U}(v)$.
\end{enumerate}
\item[(3)]
For $v_1 \in U(v)$, 
we define a wall 
\begin{equation}
W_{v_1}:=\{(s,\alpha) \in {\cal G}(v) \mid \langle e^{sH+\alpha},rv_1-r_1 v \rangle=0 \}.
\end{equation}
\item[(4)]
A chamber ${\cal C}$ is a connected component of ${\cal G}(v) \setminus \cup_{v_1 \in U(v)} W_{v_1}$.
\end{enumerate}
\end{defn}

\begin{NB}
For $E \in \overline{\cal M}_{(f,H)}^G(v)$, we take the Harder-Narasimhan filtration
in Definition \ref{defn:HNF}.
Then $\langle v(F_s/F_0)^2 \rangle \geq -r^2 \chi({\cal O}_X)$ and
$\langle e^{sH+\alpha},v(F_0) \rangle=0$.
Hence if $F_0 \ne 0$, then the Mukai vector $u$ of a stable factor of $F_0$
satisfies $u \in U(v)$.
Assume that $F_0=0$.
If $\Hom(E,F) \ne 0$ with $\langle e^{sH+\alpha},v(F) \rangle=0$, then
$E \to F$ is surjective and
$\ker(E \to F)$ satisfies Bogomolov inequality.
So the Mukai vector $u$ of a stable factor of $F$ satisfies
$u \in U(v)$.
\end{NB}

\begin{rem}
We usually regard ${\cal G}(v)$ as a subset of
$\NS(X)_{\Bbb R}$ by the correspondence
$(s,\alpha) \mapsto s H+\alpha$. 
\end{rem}

\begin{rem}\label{rem:U(v)}
By the conditions (a), (b) in (1), 
$\{D(v_1) \mid v_1 \in {\cal U}(v) \}$ is a finite set. 
By the proof of Lemma \ref{lem:wall1},
$\{(\langle v_1^2 \rangle, D(v_1)) \mid v_1 \in U(v) \}$ is also a finite set.
\end{rem}

\begin{defn}
Let $\overline{\cal M}_{(f,H)}^G(v)^{ss}$ be the stack of coherent sheaves
$E$ with $v(E)=v$ such that
$E_{|\pi^{-1}(\eta)}$ is a semi-stable vector bundle
and
\begin{equation}\label{eq:mu-ss}
\chi(G,E_1) \leq \rk E_1 \frac{\chi(G,E)}{\rk E}
\end{equation}
for all subsheaf $E_1$ of $E$ with
$$
(c_1(E_1) \cdot f)=\rk E_1 \frac{(c_1(E) \cdot f)}{\rk E}.
$$
\end{defn}
 
\begin{rem}\label{rem:mu-ss}
Every object $E \in \overline{\cal M}_{(f,H)}^G(v)^{ss}$ satisfies the following properties.
\begin{enumerate}
\item
$\Hom(A,E)=0$ for any $G$-twisted semistable fiber sheaf $A$ with $\chi(G,A)>0$.
\item
$\Hom(E,A)=0$ for any $G$-twisted semistable fiber sheaf $A$ with $\chi(G,A)<0$.
\end{enumerate}
\end{rem}

\begin{lem}\label{lem:general}
If $(s,\alpha)$ belongs to a chamber, then
$\overline{\cal M}_{(f,H)}^G(v)^{ss}={\cal M}_{(f,H)}^G(v)^{ss}$,
where $c_1(G)=\rk G (sH+\alpha)$.
\end{lem}

\begin{proof}
For $E \in \overline{\cal M}_{(f,H)}^G(v)^{ss}$ and a subsheaf $E_1$ of $E$ such that
\begin{equation}\label{eq:general}
(c_1(E_1) \cdot f)=\rk E_1 \frac{(c_1(E) \cdot f)}{\rk E},\;
\chi(G,E_1) = \rk E_1 \frac{\chi(G,E)}{\rk E},
\end{equation}
we first assume that $\rk E_1=0$.
Then $E_1$ is a $G$-twisted semistable fiber sheaf with $\chi(G,E_1)=0$.
Let $T$ be the torsion subsheaf of $E$. Then
$E_1 \subset T$ and $\langle v(E/T)^2 \rangle \geq -r^2 \chi({\cal O}_X)$.
Replacing $E_1$ by its stable factor, we may assume that $\langle v(E_1)^2 \rangle=0,-2$.
Then $v_1:=v(E_1)$ belongs to $U(v)$ and $\langle e^{sH+\alpha},v_1 \rangle=0$.
Therefore this case does not occur.
By this argument, we also get $\Hom(A,E)=0$ for all
$G$-twisted semistable fiber sheaf $A$ with $\chi(G,A)=0$.
In particular $E \in {\cal C}_G$.
We next assume that $\rk E_1=\rk E$.
Then $E/E_1$ is a $G$-twisted semistable fiber sheaf with $\chi(G,E/E_1)=0$.
In particular $E_1 \in {\cal C}_G$ and $v(E/E_1) \in {\cal T}_G(v)$.
Replacing $E/E_1$ by its quotient we may assume that $E/E_1$ is $G$-twisted stable.
Then $v(E/E_1) \in U(v)$ and $\langle e^{sH+\alpha},v(E/E_1) \rangle=0$, which is a contradiction.
Therefore this case does not occur either.  
In particular $E \in {\cal C}_G$.

We finally assume that 
$0<\rk E_1<\rk E$.
By Lemma \ref{lem:C_G}, we have a subsheaf $F$ of $\rk F=\rk E_1$ such that
\begin{enumerate}
\item
$F \in {\cal C}_G$ and
$E/F \in {\cal C}_G$.
\item
${E_1}_{|\pi^{-1}(\eta)}=F_{|\pi^{-1}(\eta)}$.
\item
$\chi(G,E_1) \leq \chi(G,F)$.
\end{enumerate}
By \eqref{eq:general} and $E \in \overline{\cal M}_{(f,H)}^G(v)^{ss}$,
$F$ also satisfies \eqref{eq:general} (i.e., $\chi(G,E_1) = \chi(G,F)$).
We set 
$v_1:=v(F)$. Then $v_1 \in U(v)$ and $\langle e^{sH+\alpha},rv_1-r_1 v \rangle=0$.
Since $v_1$ does not define a wall, we get
$r v_1-r_1 v=0$. 
Therefore $E$ is $(f,G,H)$-semistable. 
\end{proof}

\begin{NB}
\begin{defn}\label{defn:weak-wall}

\begin{enumerate}
\item[(1)]
$s=s_0$ is a wall for $v$ if there is a Mukai vector 
$$
v_1:=r_1+d_1 H+k_1 f+D_1+a_1 \varrho_X,\; D_1 \in \NS(X)_{\Bbb Q} \cap f^\perp \cap H^\perp
$$
such that 
\begin{enumerate}
\item
$(r_1,k_1,a_1) \ne \frac{r_1}{r}(r,k,a)$.
\item
$0 \leq r_1 \leq r$, $d_1=r_1 \frac{d}{r}$, 
\item
$\langle v_1^2 \rangle \geq -r_1^2 \chi({\cal O}_X)$,
$\langle (v-v_1)^2 \rangle \geq -(r-r_1)^2 \chi({\cal O}_X)$,
\item
$\langle v_1,v-v_1 \rangle>0$ if $r_1=0,r$
and 
\item
$
s_0=\frac{r_1 a-r a_1}{r_1 k-r k_1}.
$
\end{enumerate}
\item[(2)]
A chamber ${\cal C}$ is a connected component of the compliment of walls:
Let $\{ s_\lambda \mid \lambda \in \Lambda \}$ be the set of walls.
then a connected component of
${\Bbb R} \setminus \{x_\lambda \mid \lambda \in \Lambda \}$
is a chamber.
\end{enumerate}
\end{defn}
\end{NB}

For $v_1 \in U(v)$ with $r_1>0$, we note that
$$
\frac{(D_1^2)-\langle v_1^2 \rangle}{2r_1}=a_1-\frac{d_1}{r_1}k_1=a_1-\frac{d}{r}k_1.
$$
Hence we see that

\begin{equation}
\begin{split}
& -r_1 \langle e^{s H+\alpha},v \rangle +r \langle e^{s H+\alpha},v_1 \rangle=
r_1 (a-k s-(\alpha \cdot D))-r (a_1- k_1 s-(\alpha \cdot D_1))=0 \\
\iff & s=\frac{(r_1 a-r a_1)-(\alpha \cdot(r_1 D-r D_1))}{r_1 k-r k_1}\\
\iff & \left(s-\frac{d}{r} \right)(r_1 k-r k_1)+(\alpha \cdot(r_1 D-r D_1))=
-r_1\frac{\langle v^2 \rangle-(D^2)}{2r}+ r \frac{\langle v_1^2 \rangle-(D_1^2)}{2r_1}. 
\end{split}
\end{equation}

\begin{NB}
\begin{equation}
\left(s-\frac{d}{r} \right)(r_1 k-r k_1)+(\alpha \cdot(r_1 D-r D_1))=
-r_1\frac{\langle v^2 \rangle-(D^2)}{2r}+ r \frac{\langle v_1^2 \rangle-(D_1^2)}{2r_1}.
\end{equation}
\end{NB}

\begin{NB}
$\lambda((r_1 k-r k_1),r_1 D-r D_1)=(r_2 k-r k_2,r_2 D-r D_2)$ implies
$(r_2,k_2,D_2)=(\frac{r_2}{r}-\lambda \frac{r_1}{r})(r,k,D)+\lambda (r_1,k_1,D_1)$.
\end{NB}

\begin{lem}
Let $B$ be a compact subset of ${\cal G}(v)$.
Then 
$$
\# \{ v_1 \in U(v) \mid W_{v_1} \cap B \ne \emptyset \}<\infty.
$$
\end{lem}

\begin{proof}
We note that the choice of $r_1,D_1$ and $\langle v_1^2 \rangle$ is finite
(see Remark \ref{rem:U(v)}).
Hence $\{r_1 k-r k_1 \}$ is a bounded set.
Therefore the choice of $v_1$ is finite.
\end{proof}

\begin{NB}
Let $B_1$ be a compact subset of $(-\infty, \frac{d}{r})$ and
$B_2$ be a compact subset of $\NS(X)_{\Bbb R} \cap H^\perp \cap f^\perp$.
Then $\{r_1 k-r k_1 \}$ is a bounded set.
Hence the choice of $v_1$ is finite.
Therefore the set of walls is locally finite and .
wall  
\end{NB}

\begin{rem}
Let $X$ be an abelian surface.
For a pair of Mukai vectors 
$v_i=r_i+\xi_i+a_i \varrho_X \in U(v)$ ($i=1,2$) with $v=v_1+v_2$,
if $v_1$ defines a wall, then 
$\langle v_1,v_2 \rangle>0$:

We set $d_i:=(\xi_i \cdot f)$ $(i=1,2)$.
We note that 
$2r_1 r_2 \langle v_1,v_2 \rangle=-((r_2\xi_1-r_1 \xi_2)^2)+
r_2^2 \langle v_1^2 \rangle+r_1^2 \langle v_2^2 \rangle \geq 0$.
If $r_1,r_2>0$, then
$\langle v_1,v_2 \rangle>0$ unless $\langle v_1^2 \rangle=\langle v_2^2 \rangle$.
If $\langle v_1^2 \rangle=\langle v_2^2 \rangle=\langle v_1,v_2 \rangle=0$ and
$\langle e^{sH+\alpha},r_2v_1-r_1 v_2 \rangle=0$
$(s<\frac{d}{r})$, then we see that
$rs=d$. Hence it is not a wall.
If $r_1=0$, then $v_1=r_1 f+d_1 \varrho_X$.
By $s=\frac{d_1}{r_1}<\frac{d}{r}$,
$\langle v_1,v_2 \rangle=\langle v_1,v \rangle=r_1(\xi \cdot f)-rd_1>0$.
\end{rem}

\begin{lem}\label{lem:Gieseker-chamber}
For $v_1 \in U(v)$ with $r_1>0$,
$$
\frac{d}{r}-\frac{ \mu(-(D^2)+\langle v^2 \rangle+r^2 \chi({\cal O}_X))}{2}
<\frac{r_1 a-r a_1}{r_1 k-r k_1},
$$
where $\mu H \in \NS(X)$ (see Remark \ref{rem:mu}).
For $v_1=k_1 f+D_1+a_1 \varrho_X$ with $D_1 \in N_H$ and
$\langle (v-v_1)^2 \rangle \geq -r^2 \chi({\cal O}_X)$,
we also have the same inequality, where $r_1=0$.
\end{lem}

\begin{NB}
It seems that 
$$
\frac{d}{r}-\frac{ \mu (-(D^2)+\langle v^2 \rangle+r^2 \chi({\cal O}_X))}{2}
<\frac{r_1 a-r a_1}{r_1 k-r k_1},
$$
\end{NB}

\begin{proof}
If $r_1>0$, then since $\langle v_1^2 \rangle \geq -r_1^2 \chi({\cal O}_X)$ and $(D_1^2) \leq 0$, 
we see that
\begin{equation}
\begin{split}
\frac{r_1 a-r a_1}{r_1 k-r k_1}=&
\frac{d}{r}+\frac{r_1^2 ((D^2)-\langle v^2 \rangle)-
r^2((D_1^2)-\langle v_1^2 \rangle)}{2rr_1 (r_1 k-r k_1)}\\
\geq & \frac{d}{r}+\frac{r_1^2 ((D^2)-\langle v^2 \rangle)-r^2 r_1^2 \chi({\cal O}_X)}{2rr_1 (r_1 k-r k_1)}.
\end{split}
\end{equation}
Since $\mu (r_1 k-rk_1) \geq 1$ and $r_1<r$, we get the inequality.

If $r_1=0$, then 
Since 
$\langle (v-v_1)^2 \rangle =\langle v^2 \rangle-2(d k_1+(D \cdot D_1)-r a_1)+(D_1^2) 
\geq -r^2 \chi({\cal O}_X)$,
we get 
\begin{equation}
\begin{split}
\frac{r_1 a-r a_1}{r_1 k-r k_1}=
\frac{a_1}{k_1} \geq & \frac{d}{r}-\frac{((D-D_1)^2)-(D^2)+\langle v^2 \rangle+r^2 \chi({\cal O}_X)}{2rk_1}\\
\geq & \frac{d}{r}-\frac{\mu (-(D^2)+\langle v^2 \rangle+r^2 \chi({\cal O}_X))}{2r}.
\end{split}
\end{equation}
\end{proof}

\begin{prop}\label{prop:Gieseker-chamber}
We take a compact subset $B^* \subset (N_H) \otimes_{\Bbb Q} {\Bbb R}$.
Under the notation in Lemma \ref{lem:Gieseker-chamber},
we set
\begin{equation}
A:=\mu \max_{v_1 \in U(v), \alpha \in B_1} |(\alpha \cdot (r_1 D-r D_1)) |.
\end{equation}
Assume that $\alpha \in B^*$ and
$$
s<\frac{d}{r}-\frac{\mu (-(D^2)+\langle v^2 \rangle+r^2 \chi({\cal O}_X))}{2}
-A.
$$
Then $(f,{\cal O}_X(s H+\alpha),H)$-semistability
is $(sH+\alpha)$-twisted semistability with respect to 
$H_f$.
\end{prop}

\begin{proof}
Let $E$ be a coherent sheaf with $v(E)=v$.
We first assume that $E$ is $(f,{\cal O}_X(s H+\alpha),H)$-semistable.
Let $E_1$ be a subsheaf of $E$ such that
$(c_1(E_1) \cdot f)=\rk E_1 \frac{d}{r}$.
We set
$$
v(E_1)=r_1+d_1 H+k_1 f+D_1+a_1 \varrho_X, D_1 \in N_H.
$$
We first assume that $r_1>0$.
By the $(f,{\cal O}_X(s H+\alpha),H)$-semistability of $E$,
$$
a_1-(\alpha \cdot D_1)-k_1 s \leq r_1 \frac{a-(\alpha \cdot D)-ks}{r}.
$$
If $k_1 > r_1 \frac{k}{r}$, then
$$
s \geq \frac{a_1 r-a r_1}{k_1 r-k r_1}+\frac{(\alpha \cdot (r_1 D-r D_1))}{k_1 r-k r_1 } \geq 
\frac{a_1 r-a r_1}{k_1 r-k r_1}-A.
$$
By Lemma \ref{lem:Gieseker-chamber},
$$
s>\frac{d}{r}-\frac{\mu (-(D^2)+\langle v^2 \rangle+r^2 \chi({\cal O}_X))}{2}-A.
$$
Hence $k_1 \leq  r_1 \frac{k}{r}$.
Then we get
$$
a_1-(\alpha \cdot D_1) \leq r_1 \frac{a-(\alpha \cdot D)}{r}.
$$
We next assume that $r_1=0$.
Then we may assume that $\langle v(E_1)^2 \rangle=0,-2$ and
$E/E_1 \in {\cal C}_G$, and hence
$v(E_1) \in U(v)$.
If $k_1>0$, then 
$$
s \geq \frac{a_1}{k_1}-\frac{(\alpha \cdot D_1)}{k_1}>
\frac{d}{r}-\frac{\mu (-(D^2)+\langle v^2 \rangle+r^2 \chi({\cal O}_X))}{2}-A.
$$
Hence we get $k_1=0$.
In this case, $E_1$ is 0-dimensional, which implies
$D_1=0$ and $a_1>0$. 
Then $E$ is not $(f,{\cal O}_X(s H+\alpha),H)$-semistable.
Therefore $E$ is $(sH+\alpha)$-twisted semistable with respect to
$H_f$.

Conversely assume that 
$E$ is $(sH+\alpha)$-twisted semistable with respect to $H_f$.
We first prove that $E \in {\cal C}_G$.
If $\Hom(E,A) \ne 0$ ($A \in {\cal F}_G^-$), then
there is a subsheaf $E_1$ of $E$ such that $E/E_1 \in {\cal F}_G^-$.
Since $E_1$ is torsion free, $\langle v(E_1)^2 \rangle \geq -r^2 \chi({\cal O}_X)$.
We set
$$
v(E_1)=r+dH+k_1 f+D_1+a_1 \varrho_X.
$$
Then $k_1<k$ and
$$
a_1-(\alpha \cdot D_1)-k_1 s > a-(\alpha \cdot D)-ks.
$$
Hence
$$
s > \frac{a-a_1}{k-k_1}+\frac{(\alpha \cdot (D_1-D))}{k-k_1 } \geq 
\frac{a-a_1}{k-k_1}-A,
$$
which implies 
$$
s>\frac{d}{r}-\frac{\mu (-(D^2)+\langle v^2 \rangle+r^2 \chi({\cal O}_X))}{2}-A.
$$
Therefore $E \in {\cal C}_G$.

Let $E_1$ be a subsheaf of $E$ such that
$(c_1(E_1) \cdot f)=\rk E_1 \frac{d}{r}$ and $0<r_1<r$.
Then $k_1 \leq  r_1 \frac{k}{r}$.
If $k_1 =  r_1 \frac{k}{r}$, then
the stability of $E$ implies 
$$
a_1-(\alpha \cdot D_1) \leq r_1 \frac{a-(\alpha \cdot D)}{r}.
$$
Assume that $k_1 < r_1 \frac{k}{r}$.
If 
$$
a_1-(\alpha \cdot D_1)-k_1 s > r_1 \frac{a-(\alpha \cdot D)-ks}{r},
$$
then
$$
s > \frac{a_1 r-a r_1}{k_1 r-k r_1}+\frac{(\alpha \cdot (r_1 D-r D_1))}{k_1 r-k r_1 } \geq 
\frac{a_1 r-a r_1}{k_1 r-k r_1}-A.
$$
In the same way as the first paragraph, we get a contradiction.
Hence 
$$
a_1-(\alpha \cdot D_1)-k_1 s \leq  r_1 \frac{a-(\alpha \cdot D)-ks}{r}.
$$
Therefore $E$ is $(f,{\cal O}_X(s H+\alpha),H)$-semistable.
\end{proof}

\begin{rem}\label{rem:dual-Gieseker-chamber}
Assume that
$$
s>-\frac{d}{r}+\frac{ \mu (-(D^2)+\langle v^2 \rangle+r^2 \chi({\cal O}_X))}{2}
+A.
$$
By Definition \ref{defn:fGH2},
the $(f,{\cal O}_X(s H+\alpha),H)$-semistability of $E$
is the same as the $-(sH+\alpha)$-twisted semistability of $E^{\vee}$ with respect to 
$H_f$.
\end{rem}

\begin{NB}
A general wall:

$\langle e^{sH+\alpha},r_1 v-rv_1 \rangle=0$ if and only if
$(rk_1-r_1 k)s+((rD_1-r_1 D) \cdot \alpha)-(ra_1-r_1 a)=0$.

$\langle e^{sH+\alpha},r_1 v-rv_1 \rangle=0$ and 
$\langle e^{sH+\alpha},r_2 v-rv_2 \rangle=0$ define the same wall if and only if
$r_1 v-rv_1=\lambda (r_2 v-rv_2)$.
Hence $v_1=\frac{(r_1-\lambda r_2)}{r}+\lambda v_2$.

Assume that $v_1$ and $v_2$ define the same wall.
Then $\langle e^{sH+\alpha},r_1 v_2-r_2 v_1 \rangle >0$ iff
$\langle e^{sH+\alpha},r_1 (xv_2+yv_1)-(xr_2+yr_1) v_1 \rangle >0$
$(x,y>0)$.

Assume that $X$ is a K3 surface.
Then $({\Bbb Q}v+{\Bbb Q}v_1) \cap H^*(X,{\Bbb Z})$ contains infinitely many $(-2)$-vectors
or contains at most one $(-2)$-vector

\end{NB}

\section{Some properties of walls}\label{sect:wall-properties}

\subsection{The case where $(K_X \cdot H) \leq 0$.}

Let $W$ be a wall in ${\cal G}(v)$.
We shall study wall crossing. 
We take $(sH,\alpha) \in W$ and $(s_\pm H,\alpha_\pm)$ from adjacent chambers.
Let $G$ and $G_\pm$ be locally free sheaves with
\begin{equation}\label{eq:G_pm}
c_1(G)=\rk G (sH+\alpha),\;
c_1(G_\pm)=\rk G_\pm (s_\pm H+\alpha_\pm).
\end{equation}

Let ${\cal F}^\pm (v_1,v_2,\dots,v_m)$ (cf. \cite[Prop. 2.4]{Y:twist2})
be the stack of 
Harder-Narasimhan filtrations
$$
0 \subset F_1 \subset F_2 \subset \cdots \subset F_m=E
$$
of $E \in \overline{\cal M}_{(f,H)}^G(v)^{ss}$ 
with respect to $(f, G_\pm,H)$, where $v(F_i/F_{i-1})=v_i$.
\begin{rem}
If $\rk v_i=0$, then $i=1$ or $m$.
\end{rem}
By Lemma \ref{lem:EF2},
$$
\Ext^2(F_i/F_{i-1},F_j/F_{j-1})=0
$$
for $i>j$.
Since $(c_1(F_i \otimes E^{\vee}) \cdot f)=0$,
$$
\chi(F_i/F_{i-1},F_j/F_{j-1})=-\langle v(F_i/F_{i-1}),v(F_j/F_{j-1}(\tfrac{1}{2}K_X)) \rangle
=-\langle v_i,v_j \rangle.
$$
Hence by \cite[Prop. 2.5]{Y:twist2}, we get
\begin{equation}\label{eq:HNF-dim}
\dim {\cal F}^{\pm}(v_1,v_2,\dots,v_m)=\sum_{i<j} \langle v_i,v_j \rangle+
\sum_i \dim {\cal M}_{(f,H)}^{G_\pm}(v_i)^{ss}.
\end{equation}


\begin{defn}
A wall $W$ is totally semistable if 
${\cal M}_{(f,H)}^{G_+}(v)^{ss} \cap {\cal M}_{(f,H)}^{G_-}(v)^{ss}=\emptyset$,
where $G_\pm$ satisfy \eqref{eq:G_pm}.
\end{defn}

By using \eqref{eq:HNF-dim},
 we can prove the following classification of totally semistable walls.
\begin{prop}\label{prop:sswall}
Let $X$ be an elliptic K3 surface or an elliptic abelian surface. Then
a wall $W$ is totally semistable if $W$ is defined by $u \in U(v)$ satisfying
\begin{enumerate}
\item
$\langle u^2 \rangle=-2$ and $\langle v,u \rangle<0$ or
\item
$\langle u^2 \rangle=0$ and $\langle v,u \rangle=1$.
\end{enumerate}
\end{prop}

\begin{rem}
For an Enriques surface,
(i) is replaced by $u \in U(v)$ with $\langle v,u \rangle<0$ satisfying  
$\langle u^2 \rangle=-1$ or $\langle u^2 \rangle=-2$ and $c_1(u) \equiv D \mod 2$, $D$ is nodal.
\end{rem}

\begin{proof}
For the proof, see \cite{BM:2} and \cite{movable}.
\end{proof}
\begin{lem}
Assume that $(K_X \cdot H)<0$.
If $\rk v>0$, then
 ${\cal M}_{(f,H)}^G(v)^{ss}$ is a smooth stack of  
$\dim {\cal M}_{(f,H)}^G(v)^{ss}=\langle v^2 \rangle$.
\end{lem}

\begin{proof}
By Lemma \ref{lem:EF2},
$\Ext^2(E,E)=0$ for all $E \in {\cal M}_{(f,H)}^G(v)^{ss}$.
Hence our claim holds.
\end{proof}

Assume that $(K_X \cdot H)<0$, that is, $X$ is a rational elliptic surface or an elliptic ruled surface.
If $\rk v_1,\rk v_m>0$, then
$$
\dim {\cal M}_{(f,H)}^{G_\pm}(v) -\dim {\cal F}^\pm (v_1,v_2,...,v_m)=
\sum_{i>j} \langle v_i,v_j \rangle.
$$
Moreover if $X$ is an elliptic ruled surface, then
$\langle v_i^2 \rangle \geq 0$.
Therefore totally semistable walls are defined by an isotropic vector $u \in U(v)$ with 
$\langle u,v \rangle=1$.

\begin{prop}\label{prop:indep}
Assume that $(K_X \cdot H) \leq 0$.
If $\rk v>0$, then the virtual Hodge number $e({\cal M}_{(f,H)}^G(v)^{ss})$
does not depend on the choice of chambers.
\end{prop}

\begin{proof}
We note that $e({\cal M}_{(f,H)}^G(v)^{ss})$ is independent of the choice of $G$ if $v$ is the Mukai vector of a fiber sheaf.
We have
\begin{equation}
e({\cal F}^\pm(v_1,v_2,...,v_m))=
(xy)^{\sum_{i<j} \langle v_i,v_j \rangle} \prod_{i=1}^m e({\cal M}_{(f,H)}^{G_\pm}(v_i)^{ss}).
\end{equation}
By the induction on the rank, we get our claim.
\end{proof}

\begin{cor}\label{cor:ruled}
Let $X$ be an elliptic ruled surface.
Then ${\cal M}_{(f,H)}^G(v)^{ss}$ is smooth and
$e({\cal M}_{(f,H)}^G(v)^{ss})$ does not depend on the choice of chambers.

\end{cor}

For an elliptic ruled surface, we studied the non-emptiness of ${\cal M}_{H_f}(v)$
in \cite[Cor. 0.2]{Y:EllipticRuled2}. By Corollary \ref{cor:ruled}, we get the following result.
\begin{prop}
Let $\pi:X \to {\Bbb P}^1$ be an elliptic ruled surface with an elliptic fibration
and $g$ a fiber of the ${\Bbb P}^1$-fibration $X \to C$, where $C$ is an elliptic curve.
For a Mukai vector $v=r+\xi+a \varrho_X$,
${\cal M}_{(f,H)}^G(v) \ne \emptyset$ if and only if $\langle v^2 \rangle > 0$
or $r \mid (\xi \cdot g)$ and $\langle v^2 \rangle=0$.
\end{prop}

\subsection{A rational elliptic surface with a section.}
Let $X$ be a rational elliptic surface.
Assume that there is a section $\sigma$.
For a Mukai vector $v$, we have an expression
\begin{equation}
v=r+d\sigma+kf+D+a \varrho,\;
 d,k \in {\Bbb Z}, D \in ({\Bbb Z}\sigma+{\Bbb Z}f)^\perp.
\end{equation}
We take $r_0 \in {\Bbb Z}_{>0}$ and $d_0 \in {\Bbb Z}$ such that
$\gcd(r_0,d_0)=1$ and $r_0 d-rd_0=0$.
We set
\begin{equation}
{\cal E}(r,d)=\{ u \mid u=r_0+d_0 \sigma+kf+D+a \varrho_X, \langle v^2 \rangle=-1 \}. 
\end{equation}

\begin{thm}[{\cite[Thm. 7.3]{Y2010}}]\label{thm:K3/2}
Let $v$ be a Mukai vector and $G$ is general with respect to $v$.
\begin{enumerate}
\item[(1)]
Assume that $\langle v^2 \rangle \geq 0$.
Then
${\cal M}_{(f,H)}^G(v) \ne \emptyset$ if and only if
$\langle v,w \rangle \geq 0$ for all $w \in {\cal E}(r,d)$. 
\item[(2)]
Assume that $\langle v^2 \rangle=-1$.
If ${\cal M}_{(f,H)}^G(v) \ne \emptyset$, then
$\gcd(r,d)=1$.
\end{enumerate}
\end{thm}

\begin{proof}
(1)
The claim holds if ${\cal M}_{(f,H)}^G(v)={\cal M}_{H_f}^G(v)$ by \cite[Thm. 7.3]{Y2010}.
By Proposition \ref{prop:indep}, our claim holds in general. 

(2)
By using a relative Fourier-Mukai transform, we can reduce to the case where 
$r=0$.
For a stable sheaf $E$, Lemma \ref{lem:(-1)} implies
$\gcd(r,d)=-(c_1(E) \cdot K_X)=1$.
\begin{NB}$\Div(E)$ contains a reduced and irreducible curve $C$ with $(C^2)<0$.
Since $(K_X \cdot C) \leq 0$, we get $(C^2)=(C \cdot K_X)=-1$ or
$(C^2)=-2$ and $(C \cdot K_X)=0$.
\end{NB}
\end{proof}

\begin{lem}\label{lem:(-1)}
For an effective divisor $D$ with $(D^2)=-1$, assume that 
${\cal M}_H^G(0,D,a) \ne \emptyset$ where $G$ is general with respect to $(0,D,a)$.
Then $(D \cdot K_X)=-1$. 
\end{lem}

\begin{proof}
We note that $\Ext^2(E,E)=0$ for $E \in {\cal M}_H^G(0,D,a)$
by \cite[Prop. 2.7]{Y:EllipticRuled}.
By deforming $X \to {\Bbb P}^1$, we may assume that
every fiber is irreducible.
We take $E \in {\cal M}_H^G(0,D,a)$.
Then $E$ is $G$-twisted stable.
Since $(D^2)<0$, there is an irreducible and reduced component $C$ such that
$(D \cdot C)<0$.
Then we get $(C^2)<0$.
Since $(C \cdot K_X) \leq 0$ and $(C^2)+(C \cdot K_X) \geq -2$,
$(C^2)=(C \cdot K_X)=-1$ or $(C^2)=-2$ and $(C \cdot K_X)=0$.
By our assumption, the second case does not hold.
We note that $\chi({\cal O}_X(k),E)=-(C \cdot D)>0$ for all $k$.
We set $n:=\max \{ k \mid \Hom({\cal O}_C(k),E) \ne 0 \}$.
Then $\Hom({\cal O}_X(n),E)) \ne 0$ and $\Hom({\cal O}_X(n+1),E)) =0$.
Hence $\Hom(E,{\cal O}_X(n))=\Ext^2({\cal O}_X(n),E)^{\vee} \ne 0$.
Since $E$ is $G$-twisted stable, we get $E \cong {\cal O}_C(n)$.
Therefore $(D \cdot K_X)=-1$.
\end{proof}

\begin{NB}

\begin{equation}
{\cal H}={\Bbb Z}u+{\Bbb Z}w,\; \langle u^2 \rangle=-1,\; \langle u,w \rangle=0,\; \langle w^2 \rangle>0.
\end{equation}

We set
\begin{equation}
R_r:=\{ w \in {\cal H}^+ \mid \langle w^2 \rangle=-1, \rk w=r \}.
\end{equation}
Then $\# R_r \leq 2$.

Assume that 
\begin{equation}
\phi^-(E_1)>\phi^-(E_2)>\cdots >\phi^-(E_s).
\end{equation}

$v_i:=v(E_i)$.
$v_1=l_1 u_1$, $v_s=l_s u_s$.

$v_i \in {\Bbb Q}_{>0} w_1+{\Bbb Q}_{>0} w_s$ for $1<i<s$.

\begin{lem}
Assume $R_r=\{w_1,w_2 \}$.
For $v=xw_1+yw_2$ such that $x,y \geq 0$, if
$\langle v,w_1 \rangle \leq 0$ and $\langle v,w_2 \rangle \leq 0$, then
$x=y=0$. 
\end{lem}

\begin{proof}
Since 
$-x+y \langle w_2,w_1 \rangle \leq 0$ and $x \langle w_1,w_2 \rangle-y \leq 0$, we get
$(x+y)(\langle w_1,w_2 \rangle-1) \leq 0$.
Since ${\cal H}$ is hyperbolic, $\langle w_1,w_2 \rangle \geq 2$.
Hence $x=y=0$.
\end{proof}

If $s=2$, then obviously $\langle v_1,v_2 \rangle \geq 2$.
Assume that $s \geq 3$.

We set $w:=\sum_{i=2}^{s-1} v_i$.
$\langle w,v_1 \rangle>0$ or $\langle w,v_s \rangle>0$.

\begin{equation}
\begin{split}
\langle v_1,v_s \rangle
+\langle v_1, w \rangle=& \langle v,v_1 \rangle-\langle v_1^2 \rangle\\
\langle v_1,v_s \rangle
+\langle v_s, w \rangle=& \langle v,v_s \rangle-\langle v_s^2 \rangle
\end{split}
\end{equation}

Hence
\begin{equation}
\sum_{i<j} \langle v_i,v_j \rangle=\sum_{1<i<j<s} \langle v_i,v_j \rangle+\langle v_1,v_s \rangle
+\langle v_1, w \rangle+\langle v_s,w \rangle \geq 2.
\end{equation}
\end{NB}

By the proof of \cite[Lem. 6.2]{NY}, we get the following proposition.
\begin{prop}
Let $W$ be  a wall. 
Assume that $\langle v,u \rangle \geq 0$ for all
$u=D+a\varrho_X \in {\cal H}$ with
$(D^2)=-2$ and $(D \cdot f)=0$. 
Then 
$$
\dim {\cal M}_{(f,H)}^{G_\pm}(v) -\dim {\cal F}^\pm (v_1,v_2,...,v_m)\geq 1.
$$
Moreover if the equality holds, then
there is $u \in W$ with  
$\langle v,u \rangle=0$, where $u \in {\cal E}(r,d) \cap {\cal H}$ or 
$u=D+a\varrho_X \in {\cal H}$ with
$(D^2)=-2$ and $(D \cdot f)=0$.
\end{prop}

\begin{proof}
We note that
$\langle v,u \rangle \geq 0$ for $u \in {\cal E}(r,d)$ (Theorem \ref{thm:K3/2})
and
$$
\dim {\cal M}_{(f,H)}^{G_\pm}(v) -\dim {\cal F}^\pm (v_1,v_2,...,v_m)=
\sum_{i>j} \langle v_i,v_j \rangle +
\sum_i \left(\langle v_i^2 \rangle-\dim {\cal M}_{(f,H)}^{G_\pm}(v_i)\right).
$$
If ${\cal H}$ does not contain an isotropic vector, then
the proof of \cite[Lem. 6.2]{NY} implies the claim.
If ${\cal H}$ contains an isotropic vector,
then we get
$$
\{ i \mid \langle v_i^2 \rangle < 0 \}=\emptyset, \{1 \}, \{m \}.
$$
If $\langle v_1^2 \rangle<0$, then $\langle v,v_1 \rangle \geq 0$ and
\begin{equation}
\begin{split}
\sum_{i>j} \langle v_i,v_j \rangle +
\sum_i \left(\langle v_i^2 \rangle-\dim {\cal M}_{(f,H)}^{G_\pm}(v_i)\right)=&
\sum_{i>j>1} \langle v_i,v_j \rangle +
\sum_{i>1} \langle v_i,v_1 \rangle+
\langle v_1^2 \rangle-\dim {\cal M}_{(f,H)}^{G_\pm}(v_1)\\
=& \sum_{i>j>1} \langle v_i,v_j \rangle +
\langle v,v_1 \rangle-\dim {\cal M}_{(f,H)}^{G_\pm}(v_1) \geq 1.
\end{split}
\end{equation}
Assume that the equality holds. Then $m=2$,
$\langle v,v_1 \rangle=0$ and $\langle v_1^2 \rangle=-1,-2$.
\begin{NB}
If $v_1=p u$ with $\langle u^2 \rangle=-2$, then
$\dim {\cal M}_H^G(pu)=-p^2=\langle v_1^2 \rangle+p^2$. 
\end{NB}
If $\langle v_m^2 \rangle<0$, then similarly we get the same estimate.
\end{proof}

\begin{NB}
$$
I:=\{ i \mid \langle v_i^2 \rangle>0 \}.
$$
If $I \ne \emptyset$, then the proof pf \cite[Lem. 6.2]{NY} implies 
there is a non-negative integer $k$ with 
\begin{equation}
\sum_{i<j,(i,j) \in (I^2)^c} \langle v_i,v_j \rangle=\frac{k-\sum_{i \in I^c}\langle v_i^2 \rangle}{2}.
\end{equation}
Moreover $k>0$ if $\#I<s$. Hence $\sum_{i>j} \langle v_i,v_j \rangle \geq 2$ unless $\# I=1$
and $\# I^c=0,1$.

If $I=\emptyset$, then 
$$
\sum_{i<j} \langle v_i,v_j \rangle=\frac{\langle v^2 \rangle-\sum_i \langle v_i^2 \rangle}{2} >1.
$$
\end{NB}

\begin{NB}
Assume that ${\cal H}$ contains an isotropic vector.
Then the intersection matrix is
$\begin{pmatrix}
0 & b \\
b & c
\end{pmatrix}$.
If $2bxy+cy^2=-1$ has an integral solution, then
$2 \nmid c$.
If $2bxy+cy^2=-2$ has an integral solution, then
$2 \mid cy^2$.
If $2 \mid y$, then $bx(y/2)+2c(y/2)^2=-2$. Hence $y=\pm 1$.
Therefore $2 \mid c$.

Hence ${\cal H}$ contains at most one $(-1)$ or $(-2)$-vectors. 

\end{NB}

Let $E_0$ be a stable sheaf with $\langle v(E_0)^2 \rangle=-1$.
\begin{prop}
Let $\pi:X \to {\Bbb P}^1$ be a rational elliptic surface with a multiple fiber $mf_0$.
If $\langle v, v(E_0) \rangle<0$ for a stable sheaf $E_0$ with $\mu_f(E)=\mu_f(E_0)$.
Then there is no stable sheaf $E$ with $v(E)=v$.   
\end{prop}

\begin{proof}
We note that $\chi(E_0 (kf_0),E)=\chi(E_0,E)>0$ and $K_X=-f_0$.
If $\Hom(E_0 (kf_0),E)=0$, then 
$$
\Hom(E,E_0 (-(k+1)f_0))=\Ext^2(E_0(k f_0),E)^{\vee} \ne 0.
$$
Since $\Hom(E,E_0(-(k+1)f_0))=0$ for $k \gg 0$,
there is $k$ such that $\Hom(E_0(kf_0),E) \ne 0$.
We take a minimal $k$. Then $\Hom(E_0(kf_0),E) \ne 0$
and $\Hom(E,E_0(k f_0)) \ne 0$.
Therefore $E$ is not stable unless $E \cong E_0$.
\end{proof}

\subsection{Totally semistable walls}
Let $X$ be an elliptic K3 surface or an Enriques surface.
The totally semistable walls on a K3 surfaces
are studied by Bayer and Macri \cite{BM:2} and is generalized to an Enriques surfaces
by Nuer and the author in \cite[sect. 5]{NY}.
We can apply the results to our situation.
In this subsection, we shall apply the most complicated case, that is,
the wall associated to stable rigid sheaves. 
For a totally semistable wall $W$, we take $(sH,\alpha) \in W$ and
$(s_\pm H,\alpha_\pm)$ from adjacent chambers.
Let $G$ and $G_\pm$ be locally free sheaves with \eqref{eq:G_pm}. 
Let ${\cal H}$ be a primitive hyperbolic sublattice of $H^*(X,{\Bbb Z})$ containing $v$ 
and assume that
${\cal H}$ contains at least two $(-2)$-vectors or $(-1)$-vectors.
Let ${\cal H}^+$ be the effective cone, that is, the cone spanned by Mukai vectors
$v(E) \in {\cal H}$ of $(f,G,H)$-semistable sheaves $E$.
Let $R^+$ be the set of effective $u$ with $\langle u^2 \rangle=-1,-2$.
Then there are $u_0,u_1 \in R^+$ such that $R^+=\{ u_n \mid n \in {\Bbb Z} \}$ with
\begin{equation}
u_{i+1}=
\begin{cases}
-R_{u_i}(u_{i-1}) & i \ne 0,1\\
R_{u_i}(u_{i-1}) & i=0,1
\end{cases}
\end{equation} 
and there are $(f,G,H)$-stable sheaves $T_i$ $(i=0,1)$
with $v(T_i)=u_i$.
We note that $\rk T_0>0$ or $\rk T_1>0$.
For simplicity, we assume that $\rk T_0>0$.

\begin{defn}
For $E \in \Coh(X)$ with $\rk E>0$, we set
$$
\phi^\pm(E)=\frac{\chi(G_\pm,E)}{\rk E}.
$$ 
If $\rk T_1=0$, then we set
$\phi^+(T_1)=\infty$ and $\phi^-(T_1)=-\infty$.
\end{defn}

Let $T^\pm_i$ be $(f,G_\pm,H)$-stable sheaves with $v(T_i^\pm)=u_i$.
$T_i^\pm$ are uniquely determined by $u_i$ up to $\otimes {\cal O}_X(K_X)$.
We may assume that

\begin{equation}
\phi^+(T_1)>\phi^+(T_2^+)>\cdots>\phi^+(E)>\cdots>\phi^+(T_{-1}^+)>\phi^+(T_0),
\end{equation} 
  
\begin{equation}
\phi^-(T_1)<\phi^-(T_2^-)<\cdots<\phi^-(E)<\cdots<\phi^-(T_{-1}^-)<\phi^-(T_0),
\end{equation} 
where $E$ is a $(f,G,H)$-semistable sheaf satisfying $\langle v(E)^2 \rangle \geq 0$.
By \cite[Lem. 6.18]{NY},
\begin{equation}
R_{T_i^+}(T_{i-1}^+)=
\begin{cases}
T_{i+1}^+[1] & i \ne 0,1\\
T_{i+1}^+ & i = 0,1
\end{cases}
\end{equation}
and

\begin{equation}
R_{T_i^-}(T_{i+1}^-)=
\begin{cases}
T_{i-1}^-[1] & i \ne 0,1\\
T_{i-1}^- & i = 0,1.
\end{cases}
\end{equation}

\begin{NB}
$\Hom(T_2^+,T_1^+[k])=\Hom(T_0,T_1[k+1])$.

For $i \geq 2$,
$\Hom(T_{i+1}^+,T_i^+[k])=\Hom(T_{i-1}^+[-1],T_i^+[k+1])=\Hom(T_i^+[k+1],T_{i-1}^+[1])^{\vee}
=\Hom(T_i^+[k],T_{i-1}^+)^{\vee}$.
\end{NB}

\begin{NB}
\begin{lem}
Let $E$ be a $(f,G_+,H)$-stable sheaf with $\phi(E)=\phi({\cal O}_X)=\chi(G_0,{\cal O}_X)$.
\begin{enumerate}
\item[(1)]
If $\langle v(E)^2 \rangle \geq 0$, then 
$$
\phi^+(T_i)>\phi^+(E)>\phi^+(T_j)
$$
for any $i>0$ and $j \leq 0$.
\item[(2)]
If $\phi^+(E)\geq \phi^+(T_i)$, $i>0$, then
$E=T_p^+$ $(0<p \leq i)$.
\item[(3)]
If $\phi^+(E) \leq \phi^+(T_j)$, $j \leq 0$, then
$E=T_p^+$ $(0 \geq p \geq j)$.
\end{enumerate}
\end{lem}

\begin{proof}
We note that $v(E) \in {\cal H}$.
\end{proof}

\begin{lem}
Assume that $\Hom(T_i^+,T_{i-1}^+[k])=0$ for $k \ne 0$. For the evaluation map 
\begin{equation}
\varphi:\Hom(T_i^+,T_{i-1}^+) \otimes T_i^+ \to T_{i-1}^+,
\end{equation}
we get 
$\varphi$ is surjective and $\ker \varphi=R_{T_i^+}(T_{i-1}^+)$ is $(f,G_+,H)$-stable.
In particular $T_{i+1}^+=R_{T_i^+}(T_{i-1}^+)$ and
$\Hom(T_{i+1}^+,T_i^+[k])=0$ for $k \ne 0$.
\end{lem}

\begin{NB2}
$v(\im \varphi) \in {\Bbb Z}u_0+{\Bbb Z}u_1$.
\end{NB2}

\begin{proof}
We get
$$
\phi^+(T_{i-1}^+) \geq \phi_{\max}^+(\im \varphi) \geq \phi_{\min}^+(\im \varphi) \geq \phi^+(T_i^+).
$$
If $\phi^+(T_{i-1}^+) = \phi_{\max}^+(\im \varphi)$, then $T_{i-1}^+=\im \varphi$.
If $\phi^+(T_{i-1}^+)> \phi_{\max}^+(\im \varphi)$, then
$$
\phi_{\max}^+(\im \varphi)=\phi_{\min}^+(\im \varphi)=\phi^+(T_i^+)
$$
and $\im \varphi=T_i^+ \otimes U$.
Since $\Hom(T_i^+,T_{i-1}^+) \to \Hom(T_i^+,\im \varphi)$ is isomorphic,
$\varphi$ is injective.
By $\rk \Hom(T_i^+,T_{i-1}^+) \otimes T_i^+ > \rk T_{i-1}^+$,
this case does not occur. Hence $\varphi$ is surjective.

We note that $v(\ker \varphi)=u_{i+1}$. Assume that $\ker \varphi$ is not stable.
Then
$$
\phi^+(T_i^+) \geq \phi_{\max}^+(\ker \varphi) \geq \phi^+(T_{i+1}^+).
$$
If $\phi^+(T_i^+) = \phi_{\max}^+(\ker \varphi)$, then
$\ker \varphi$ contains $T_i^+$, which is a contradiction.
Hence $\phi^+(T_i^+) > \phi_{\max}^+(\ker \varphi)$.
In this case, we get $\phi_{\max}^+(\ker \varphi)=  \phi^+(T_{i+1}^+)$.
Hence $\ker \varphi$ is semistable.
By $v(\ker \varphi)=u_{i+1}$, $\ker \varphi=T_{i+1}^+$.
\end{proof}

\begin{lem}
Assume that $\Hom(T_i^+,T_{i-1}^+[k])=0$ for $k \ne 0$. 
We set
$$
\psi:T_i^+ \to \Hom(T_i^+,T_{i-1}^+)^{\vee} \otimes T_{i-1}^+.
$$
Then
$\psi$ is injective and $\coker \psi=R_{T_{i-1}^+}^{-1}(T_i^+)[1]=
T_{i-2}^+$.
In particular $\Hom(T_{i-1}^+,T_{i-2}^+[k])=0$ for $k \ne 0$.
\end{lem}

\begin{proof}
$$
\phi^+(T_i^+) \leq \phi_{\min}^+(\im \psi) \leq \phi_{\max}^+(\im \psi) \leq \phi^+(T_{i-1}^+).
$$
 
If $\phi^+(T_i^+) = \phi_{\min}^+(\im \psi)$, then $\im \psi=T_i^+$ and $\psi$ is injective.
If $\phi^+(T_i^+) < \phi_{\min}^+(\im \psi)$, then 
$$
\phi_{\min}^+(\im \psi) = \phi_{\max}^+(\im \psi) = \phi^+(T_{i-1}^+).
$$
Hence $\im \psi=T_{i-1}^+ \otimes U$.
Since $U \to \Hom(T_i^+,T_{i-1}^+)^{\vee}$ is isomorphic,
$\psi$ is surjective.
Since $\rk \ker \psi <0$, $\psi$ is injective.  
We note that 
$$
\phi^+(T_{i-1}^+) \leq \phi_{\min}^+(\coker \psi) \leq \phi_{\max}^+(\coker \psi) \leq \phi^+(T_{i-2}^+).
$$
If $\phi^+(T_{i-1}^+) = \phi_{\min}^+(\coker \psi)$, then
$\coker \psi=T_{i-1}^+ \otimes U$, which is a contradiction.
Then 
$\phi_{\min}^+(\coker \psi) = \phi_{\max}^+(\coker \psi) = \phi^+(T_{i-2}^+)$.
Since $v(\coker \psi)=u_{i-2}$, $\coker \psi=T_{i-2}^+$.
\end{proof}

\begin{lem}
Let $E$ be the universal extension
\begin{equation}
0 \to T_0 \to T_2^+ \to \Ext^1(T_1,T_0) \otimes T_1 \to 0.
\end{equation}
Then $E=T_2^+$.
\end{lem}

\begin{proof}
Let $E$ be the universal extension. We note that
$\phi^+(T_2^+) \leq \phi_{\max}^+(E) \leq \phi^+(T_1)$.
If $\phi_{\max}^+(E) = \phi^+(T_1)$, then $T_1$ is a direct summand of $E$.
So $\phi^+(T_2^+) = \phi_{\max}^+(E)$. Hence $E=T_2^+$.
\end{proof}

\begin{lem}
Let $E$ be the universal extension
\begin{equation}
0 \to \Ext^1(T_1,T_0)^{\vee} \otimes T_0 \to T_{-1}^+ \to T_1 \to 0.
\end{equation}
Then $E=E_{-1}^+$. 
\end{lem}

\begin{proof}
Let $E$ be the universal extension.

$\phi^+(T_0) \leq \phi_{\min}^+(E) \leq \phi^+(T_{-1}^+)$.
If $\phi^+(T_0) = \phi_{\min}^+(E)$, then $T_0$ is a direct summand of $E$.
Hence $\phi^+(T_0) < \phi_{\min}^+(E)$. Then $\phi_{\min}^+(E) = \phi^+(T_{-1}^+)$ and.
$E=T_{-1}^+$. 
\end{proof}

\end{NB}
We set
\begin{equation}
\begin{split}
{\cal C}_0:=&\{ v \in {\cal H} \mid \langle v, u_0 \rangle>0, \langle v,u_1 \rangle>0 \}\\
{\cal C}_n:=& \left\{v \in {\cal H} \left| \langle v,u_n \rangle<0, \langle v,u_{n+1} \rangle>0 \right. \right\},\; n>0\\
{\cal C}_n:=& \left\{v \in {\cal H} \left| \langle v,u_n \rangle>0, \langle v,u_{n+1} \rangle<0 \right. \right\},\; n<0
\end{split}
\end{equation}
Then $\{{\cal C}_n \mid n \in {\Bbb Z} \}$ is the chamber decomposition
of ${\cal H}$ by $u \in R^+$ and 
$$
{\cal C}_0=\{ v \in {\cal H} \mid \langle v, u_n \rangle>0, n \in {\Bbb Z} \}.
$$
For $v_0 \in {\cal C}_0$, 
we set 
\begin{equation}
v_n:= 
\begin{cases}
R_{u_n} \circ R_{u_{n-1}} \circ \cdots \circ R_{u_1}(v_0), & n>0\\
R_{u_{n+1}}^{-1} \circ R_{u_{n+2}}^{-1} 
\circ \cdots \circ R_{u_0}^{-1}(v_0), & n \leq 0.
\end{cases}
\end{equation}
Then for $v \in {\cal C}_n$, there is $v_0 \in {\cal C}_0$ such that
$v_n=v$.

\begin{NB}
\subsubsection{}

\begin{equation}
0 \to \Ext^1(T_1,T_0)^{\vee} \otimes T_1 \to T_2^- \to T_0 \to 0.
\end{equation}

$T_3^-:=R_{T_2^-}^{-1}(T_1^-[1])$. Then  
\begin{equation}
0 \to T_1^- \to \Ext^2(T_2^-,T_1) \otimes T_2^-  \to T_3^- \to 0.
\end{equation}

$R_{T_1} \circ R_{T_2}:{\cal A}_2^* \to {\cal A}_1^* \to {\cal A}_0$.

$T_2^+=R_{T_1}(T_0)$.
$$
0 \to T_0 \to T_2^+ \to \Ext^1(T_1,T_0) \otimes T_1 \to 0. 
$$

$T_2^+ \ne {\cal O}_X(3C)$.

Indeed $\Hom({\cal O}_C(-5),T_1^{\oplus 3}) \to \Hom({\cal O}_C(-5),T_0[1])={\Bbb C}^4$ is not injective.
Hence $\Hom({\cal O}_C(-5),T_2^+) \ne 0$.

$T_3^+[1]=R_{T_2^+}(T_1)$.

\begin{equation}
0 \to T_3^+ \to \Hom(T_2^+,T_1) \otimes T_2^+ \to T_1 \to 0.
\end{equation}

$R_{T_0}^{-1}(T_1)=T_{-1}^+$.

$$
0 \to \Ext^1(T_1,T_0) \otimes T_0 \to T_{-1}^+ \to T_1 \to 0.
$$

$T_{-1}^+$ is locally free.

$R_{T_0}(T_1)=T_{-1}^-$.

\begin{equation}
0 \to T_1 \to T_{-1}^- \to \Ext^1(T_0,T_1) \otimes T_0 \to 0.
\end{equation}

$R_{T_{-1}^+}(T_{-2}^+)=T_0[1]$.

\begin{equation}
0 \to T_0 \to \Hom(T_0,T_{-1}^+)^{\vee} \otimes T_{-1}^+ \to T_{-2}^+ \to 0.  
\end{equation}

\begin{equation}
0 \to T_{-1}^+ \to \Hom(T_{-1}^+,T_{-2}^+)^{\vee} \otimes T_{-2}^+ \to T_{-3}^+ \to 0.  
\end{equation}

$R_{T_{-1}^-}(T_0)=T_{-2}^- [1]$.

\begin{equation}
0 \to T_{-2}^- \to \Hom(T_{-1}^-,T_0) \otimes T_{-1}^- \to T_0 \to 0.
\end{equation}

$T_i^+$ $(i<0)$ are locally free.
$T_i^-$ $(i \geq 2, i<0)$ are not torsion free.

For $i<0$, 
\begin{equation}
0 \to T_{i-1}^- \to \Hom(T_{i}^-,T_{i+1}^-) \otimes T_{i}^- \to T_{i+1}^- \to 0.
\end{equation}
\end{NB}

\begin{NB}

$$
\varphi:\Hom(T_i^-,T_{i+1}^-) \otimes T_i^- \to T_{i+1}^-
$$

$$
\phi^-(T_{i+1}^-) \geq \phi_{\max}^-(\im \varphi) \geq \phi_{\min}^-(\im \varphi) \geq \phi^-(T_i^-).
$$

If $\phi^-(T_{i+1}^-) = \phi_{\max}^-(\im \varphi)$, then $T_{i+1}^-=\im \varphi$.

If $\phi^-(T_{i+1}^-)> \phi_{\max}^-(\im \varphi)$, then
$\phi_{\max}^-(\im \varphi)=\phi_{\min}^-(\im \varphi)=\phi^-(T_i^-)$
and $\im \varphi=T_i^- \otimes U$.
Since $\Hom(T_i^-,T_{i+1}^-) \to \Hom(T_i^-,\im \varphi)$ is isomorphic,
$\varphi$ is injective.

$v(\ker \varphi)=u_{i-1}$. Assume that $\ker \varphi$ is not stable.
Then
$$
\phi^-(T_i^-) \geq \phi_{\max}^-(\ker \varphi) \geq \phi^-(T_{i-1}^-).
$$
If $\phi^-(T_i^-) = \phi_{\max}^-(\ker \varphi)$, then
$\ker \varphi$ contains $T_i^-$, which is a contradiction.
Hence $\phi^-(T_i^-) > \phi_{\max}^-(\ker \varphi)$.
In this case, we get $\phi_{\max}^-(\ker \varphi)=  \phi^-(T_{i-1}^-)$.
Hence $\ker \varphi$ is semistable.
By $v(\ker \varphi)=u_{i-1}$, $\ker \varphi$ is stable.

\end{NB}

We set
\begin{equation}
R_-:=R_{T_0} \circ R_{T_1},\; R_+:=R_{T_1} \circ R_{T_0}.
\end{equation}
Then
\begin{equation}
R_-=R_{T_i^-} \circ R_{T_{i+1}^-},\; R_+=R_{T_i^+} \circ R_{T_{i-1}^+}
\end{equation}
(cf. \cite[Lem. 6.18]{NY}).

\begin{prop}[{\cite[Prop. 6.20]{NY}}]\label{Prop:NonMinimalIsomorphism}
\begin{enumerate}
\item[(1)]
Assume that $n$ is even and $v_n \in {\cal C}_n$.
Then $R_+^{\frac{n}{2}} \circ R_-^{\frac{n}{2}}$ induces a birational map
\begin{equation}
{\cal M}_{(f,H)}^{G_-}(v_n) \cong {\cal M}_{(f,H)}^{G_-}(v_0) \cdots \to {\cal M}_{(f,H)}^{G_+}(v_0)
\cong {\cal M}_{(f,H)}^{G_+}(v_n).
\end{equation}
\item[(2)]
Assume that $n$ is odd and $v_n \in {\cal C}_n$.
Then $R_+^{\frac{n-1}{2}} \circ R_{T_1} \circ R_{T_1} \circ R_-^{\frac{n-1}{2}}$
induces a birational map
\begin{equation}
{\cal M}_{(f,H)}^{G_-}(v_n) \cong {\cal M}_{(f,H)}^{G_-}(v_1) \cong {\cal M}_{(f,H)}^{G_+}(v_0)  
\cdots \to {\cal M}_{(f,H)}^{G_-}(v_0) \cong 
{\cal M}_{(f,H)}^{G_+}(v_1)
\cong {\cal M}_{(f,H)}^{G_+}(v_n).
\end{equation}
\end{enumerate}
\end{prop}

\begin{ex}
Let $X$ be a K3 surface and
$C$ a $(-2)$-curve in a fiber.
$u_0=v({\cal O}_X)$ and $u_1=(0,C,-3)$.
$$
{\cal H}:=({\Bbb Q}v_0+{\Bbb Q}v_1) \cap H^*(X,{\Bbb Z})={\Bbb Z}u_0+{\Bbb Z}u_1
$$ 
is a hyperbolic lattice.
$u=x u_0+y u_1$ is a $(-2)$-vector if and only if $x^2-3xy+y^2=1$.
$u_2=u_0+\langle u_0,u_1 \rangle u_1=u_0+3u_1$.
$u_3=-(u_1+\langle u_1,u_2 \rangle u_2)=3u_0+8u_1$.
$u_4=-(u_2+\langle u_2,u_3 \rangle u_3)=7u_0+17 u_1$.
$u_{-1}=u_1+\langle u_1,u_0 \rangle u_0=3u_0+u_1$.
$u_{-2}=-(u_0+\langle u_0,u_{-1} \rangle u_{-1})=8u_0+3u_1$.
$u_{-3}=-(u_{-1}+\langle u_{-1},u_{-2} \rangle u_{-2})=21u_0+8u_1$.

We take $G_0$ with $\langle v(G_0),u_1 \rangle=0$.
We set $T_0:={\cal O}_X$.
Let $T_1$ be a $G_0$-twisted stable fiber sheaf with $v(T_1)=u_1$.
Let ${\cal A}$ be the full subcategory of $\Coh(X)$ generated by $E \in \overline{\cal M}_{(f,H)}^{G_0}(v')$ 
with $v' \in {\cal H}$.
Thus $\phi(E)=\frac{\chi(G_0,E)}{\rk E}=\chi(G_0,{\cal O}_X)$.
\end{ex}

\begin{NB}
Let $f=C_1+C_2$ be a singular fiber.
Assume that there is a section $\sigma$ with $(\sigma \cdot C_1)=1$.
We set $H=\sigma+\frac{1}{4}C_1+kf$, where $(H^2)=0$.
We set $\alpha=\frac{1}{2}(C_1-C_2)$.
Then $(\alpha \cdot C_1)=-2$, $(\alpha \cdot C_2)=2$,
$(\alpha^2)=-2$ and $(H \cdot \alpha)=0$.
$(H \cdot C_i)=\frac{1}{2}$.

$\chi({\cal O}_{C_1}(-sH-\alpha))=-\frac{1}{2}s+3$ and
$\chi({\cal O}_{C_2}(-sH-\alpha))=-\frac{1}{2}s-1$.
$\chi(T_1(-sH-\alpha))=\chi({\cal O}_{C_2}(-sH-\alpha-4))=-\frac{1}{2}s-5$

If $s=-10$, then ${\cal O}_{C_2}(-4)$ is $G_0$-twisted stable with
$\chi(G_0,{\cal O}_{C_2}(-4))=0$.
${\cal O}_X$ is $(f,G_0,H)$-stable sheaf with $v({\cal O}_X)=v_0$.
So $T_0={\cal O}_X$ and $T_1={\cal O}_{C_2}(-4)$.

Assume that $s=-2$.
Then $\chi(G_0,{\cal O}_X)=\chi(G_0,{\cal O}_{C_2})=0$.
In this case ${\cal O}_X(-C_2)$ is $(f,G_0,H)$-stable.
Indeed $\chi(G_0,{\cal O}_X(-C_2)_{|C_1})=2$ and
$\chi(G_0,{\cal O}_X(-C_2)_{|C_2})=6$.

\end{NB}

We set $v:=7u_0+17u_1$. Then $\langle v^2 \rangle=38$ and $v \in {\cal C}_2$.
\begin{NB}
$\langle v,u_2 \rangle=-2$ and $\langle v,u_3 \rangle=7$.
\end{NB}
Thus 
$v=v_2=7u_0+17u_1$,
$v_1=R_{u_2}(v_2)=5u_0+11u_1$ and
$v_0=R_{u_1}(v_1)=5u_0+4u_1 \in {\cal C}_0$.
\begin{NB}
$v_3=28 u_0+73 u_1$.
$v_{-1}=7u_0+4u_1$.
\end{NB}

Let $F$ be a $(f,G_0,H)$-stable sheaf with $v(F)=v_0$.

\begin{equation}
0 \to \Ext^1(F,T_1)^{\vee} \otimes T_1 \to F_1^- \to F \to 0.
\end{equation}

\begin{equation}
0 \to \Ext^1(F_1^-,T_2^-)^{\vee} \otimes T_2^- \to F_2^- \to F_1^- \to 0.
\end{equation}

\begin{NB}
$R_{T_1} \circ R_{T_2^-}:{\cal A}_2^* \to {\cal A}_1^* \to {\cal A}$.
$T_2^- \in {\cal A}_1^*$ is irreducible. Hence $\Hom(T_2^-,F_1^-)=\Hom(F_1^-,T_2^-)=0$.

Or since $R_{T_1}(T_2^-)=T_0$ and $R_{T_1}(F_1^-)=F$, $\Hom(F_1^-,T_2^-[k])=\Hom(F,T_0[k])$.

$F_i^- \in {\cal M}_{(f,H)}^{G^-}(v_i)$.
\end{NB}

\begin{equation}
0 \to F \to F_1^+ \to \Ext^1(T_1,F) \otimes T_1  \to 0.
\end{equation}

\begin{equation}
0 \to F_1^+ \to F_2^+ \to \Ext^1(T_2^+,F_1^+) \otimes T_2^+  \to 0.
\end{equation}

\begin{NB}
$R_{T_2^+} \circ R_{T_1}:{\cal A} \to {\cal A}_1 \to {\cal A}_2$.
$T_2^+ \in {\cal A}_1$ is irreducible. Hence $\Hom(T_2^+,F_1^+)=\Hom(F_1^+,T_2^+)=0$.
\end{NB}

\begin{equation}
{\cal M}_{(f,H)}^{G_-}(v) \overset{R_-}{\to} {\cal M}_{(f,H)}^{G_-}(v_0) \cdots \to {\cal M}_{(f,H)}^{G_+}(v_0) 
\overset{R_+}{\to}  {\cal M}_{(f,H)}^{G_+}(v)
\end{equation}

\begin{NB}
Other example.

$v=5u_0+13u_1$, $v_1=2u_0+5u_1$, $v_2=u_0+2u_1$ and $v_3=u_0+u_1$.
$\langle v^2 \rangle=2$.
\end{NB}

\begin{ex}
Let $X$ be an Enriques surface and $C$ a $(-2)$-curve.
We set $u_0:=v({\cal O}_X)$ and $u_1:=(0,C,-2)$.
Then
$$
{\cal H}:=({\Bbb Q}v_0+{\Bbb Q}v_1) \cap H^*(X,{\Bbb Z})={\Bbb Z}u_0+{\Bbb Z}u_1
$$ 
is a hyperbolic lattice with
$\langle u_0^2 \rangle=-1, \langle u_0,u_1 \rangle=2, \langle u_1^2 \rangle=-2$.
We see that
$u_2:=u_0+\langle u_0,u_1 \rangle u_1=u_0+2u_1$,
$u_3:=-(u_1+2\langle u_1,u_2 \rangle u_2)=4u_0+7u_1$,
$u_{-1}:=u_1+2\langle u_1,u_0 \rangle u_0=4u_0+u_1$,
$u_{-2}:=-(u_0+\langle u_0,u_{-1} \rangle u_{-1})=7u_0+2u_1$.

\begin{NB}
$(-1+\sqrt{2})^n=a_n+b_n \sqrt{2}$.

If $2 \mid u$, then
$u_n=-(a_n+2b_n)u_0-b_n u_1=-a_n u_0-b_n (u_1+2u_0)$.
If $2 \nmid n$, then
$u_n=-(2(a_n+b_n)u_0+a_n u_1)=-a_n (2u_0+u_1)-2b_n u_0$.
\end{NB}
For $v:=5u_0+8u_1$, we see that
$v \in {\cal C}_2$ with $\langle v^2 \rangle=7$.
Hence we have
\begin{equation}
{\cal M}_{(f,H)}^{G_-}(v) \overset{R_-}{\to} {\cal M}_{(f,H)}^{G_-}(v_0) \cdots \to {\cal M}_{(f,H)}^{G_+}(v_0) 
\overset{R_+}{\to}  {\cal M}_{(f,H)}^{G_+}(v),
\end{equation}
where 
$v_1=R_{u_2}(v)=3u_0+4u_1$ and
$v_0=R_{u_1}(v_1)=3u_0+2u_1$.
\begin{NB}
$v_0=3u_0+2u_1$. $\langle v_0^2 \rangle=7$.
$v_{-1}=5u_0+2u_1$.
$v_1=3u_0+4u_1$.
$v_2=5u_0+8u_1$.

$\langle v,u_2 \rangle=-1$ and $\langle v,u_3 \rangle=2$.

$xu_0+yu_1 \in {\cal C}_0 \iff y<x<2y$.
 \end{NB}
\end{ex}

\subsection{The dependence on $H$.}
Assume that all fibers of $\pi$ are irreducible.
For $\eta \in N_H$, 
we set $H_\eta:=H+\eta-\frac{(\eta^2)}{2}f$. 
It is a $\pi$-ample ${\Bbb Q}$-divisor. 
Hence we have a decomposition 
$\NS(X)_{\Bbb Q}=({\Bbb Q}H_\eta+{\Bbb Q}f)+N_{H_\eta}$.
\begin{NB}
$sH+\alpha=sH_\eta+((\alpha \cdot \eta)-\tfrac{(\eta^2)}{2})f+(((s(\eta^2)-(\alpha \cdot \eta))f-s\eta+\alpha)$.
\end{NB}
Let us consider the projective space
${\Bbb P}:={\Bbb P}({\Bbb R} \oplus \NS(X)_{\Bbb R}/{\Bbb R}f)$.
It is a compactification of $\NS(X)_{\Bbb R}/{\Bbb R}f$.
Let $L:={\Bbb P} \setminus \NS(X)_{\Bbb R}/{\Bbb R}f$ be the hyperplane at infinity.
We have an embedding ${\cal G}(v) \hookrightarrow {\Bbb P}$
by $(s,\alpha) \to {\Bbb R}^{\times}(1+sH+\alpha)$.

\begin{prop}\label{prop:unbounded}
Let ${\cal C}$ be a chamber which is adjacent to $L$, that is,
the closure $\overline{\cal C}$ intersects $L$.
Assume that we can take a general point ${\Bbb R}^{\times }(H+\eta)$ of $\overline{\cal C} \cap L$, where
$\eta \in N_H$. Then
${\cal M}_{(f,H)}^G(v)={\cal M}_{(H_\eta)_f}^G(v)$
or
${\cal M}_{(f,H)}^G(v)=\{E \mid E^{\vee} \in {\cal M}_{(H_\eta)_f}^{G^{\vee}}(v^{\vee}) \}$.
\end{prop}
\begin{NB}
Since ${\cal C}$ defines an unbounded chamber in $\{s H_\eta \mid rs \ne d \}$,
${\cal M}_{(f,H)}^G(v)$ parameterizing Gieseker semistable sheaves $E$ or its dual, where the polarization is 
$H_\eta+nf$ $(n \gg 0)$
\end{NB} 

\begin{proof}
Since ${\cal C}$ defines an unbounded chamber in $\{s H_\eta \mid rs \ne d \}$,
the claim follows from Proposition \ref{prop:Gieseker-chamber} and Definition \ref{defn:fGH2}.
\end{proof}

\section{Relative Fourier-Mukai transforms}\label{sect:FM}

\subsection{Fourier-Mukai transforms by the moduli spaces of stable 1-dimensional sheaves.}\label{subsect:FM}

Let $\pi:X \to C$ be an elliptic surface over a curve $C$
and $f$ a fiber of $\pi$.
Let $H$ be a relatively ample ${\Bbb Q}$-divisor such that
$(H \cdot f)=1$ and $(H^2)=0$.
Let 
$$
v_0=r_0 f+b \varrho_X, \; r_0, b \in {\Bbb Z}
$$
be a primitive and isotropic Mukai vector 
such that $r_0>0$.
We shall consider $\beta$-twisted stability with respect to
$H_f$.
\begin{lem}
${\cal M}_{H_f}^\beta(v_0)^{ss}$ depends only on
$H$ and
$\beta \mod {\Bbb Q}f+{\Bbb Q}H$.
In particular ${\cal M}_{H_f}^{\beta-\frac{1}{2}K_X}(v_0)^{ss}={\cal M}_{H_f}^\beta(v_0)^{ss}$.
\end{lem}

\begin{proof}
For a coherent sheaf $E$ supported on fibers,
$$
\langle e^{\beta+xH+yf},v(E) \rangle=\langle e^\beta, v(E) \rangle+(c_1(E) \cdot (xH+yf))=
\langle e^\beta, v(E) \rangle+x(c_1(E) \cdot (H+nf)).
$$
Hence the claim holds.  
\end{proof}
Replacing $\beta$ by $\beta-\frac{\langle e^\beta,v_0 \rangle}{r_0} H$,
we assume that
\begin{equation}\label{eq:beta}
\langle e^\beta,v_0 \rangle=0 \iff \beta \equiv \frac{b}{r_0}H+\beta_0 \mod {\Bbb Q}f\,\; (\beta_0 \in N_H).
\end{equation}
\begin{NB}
We can take $e^\beta$ such that $\langle e^\beta,v_0 \rangle=0$.
Indeed we can choose $y$ 
satisfying
$\langle e^{\beta+yH_n},v_0 \rangle=\langle e^\beta,v_0 \rangle+y(H \cdot r_0 f)=0$. 
Then $v_0=r_0 f e^\beta$.
\end{NB}
Assume that $X':=\Mod_{H_n}^\beta(v_0)$ is a fine moduli
scheme consisting of $\beta$-twisted stable sheaves. Then
$X'$ is an elliptic surface with a fibration 
$\pi':X' \to C$. We denote a fiber of $X' \to C$ by $f'$.

\begin{rem}
$X'$ is fine if and only if there is a divisor $\xi$ such that
$\gcd(r_0 (\xi \cdot f),b)=1$.
\end{rem}

For a universal family ${\cal P}$,
we have a Fourier-Mukai transform
\begin{equation}\label{eq:relativeFM}
\Phi_{X \to X'}^{{\cal P}^{\vee}[1]}:{\bf D}(X) \to {\bf D}(X').
\end{equation}
For simplicity, we set $\overline{\Phi}:=\overline{\Phi}_{X \to X'}^{{\cal P}^{\vee}[1]}$.  
We set
\begin{equation}
\begin{split}
v_0':=&v({\cal P}_{|\{ x \} \times X'}^{\vee}[1])=
\overline{\Phi}(\varrho_X)\\
H':=&-c_1(\overline{\Phi}(r_0 e^\beta)).
\end{split}
\end{equation}
Then $c_1(v_0')=r_0 f'$, $({H'}^2)=0$ and $(H' \cdot f')=1$.
\begin{NB}
$r_0 (H',f')=(H',c_1(v_0'))=\langle r_0 e^\beta,-\varrho_X \rangle=r_0$.
\end{NB}
Since $\langle \overline{\Phi}(He^\beta)^2 \rangle=0$ and
$\langle \overline{\Phi}(He^\beta),\varrho_{X'} \rangle=-(H \cdot r_0 f)=-r_0$,
there is $\beta' \in \NS(X')_{\Bbb Q}$ with
$\overline{\Phi}(He^\beta)=r_0 e^{\beta'}$.
Then we have 
\begin{equation}\label{eq:H'}
\begin{split}
\overline{\Phi}(r_0 e^\beta)&=-H' e^{\beta'}\\
\overline{\Phi}(He^\beta)&=r_0 e^{\beta'}\\
\overline{\Phi}(v_0)&=-\varrho_{X'}\\
\overline{\Phi}(\varrho_X)&=v_0'.
\end{split}
\end{equation}
By $\langle \overline{\Phi}(He^\beta),\overline{\Phi}(\varrho_X) \rangle=0$, we get
$v_0'=e^{\beta'}(r_0 f')$.
\begin{NB}
$\Phi({\cal P}_{|X \times \{ x' \}})={\Bbb C}_{x'}$.
$\Phi({\Bbb C}_x)=({\cal P}_{|\{x \} \times X'}^{\vee}[1])[-1]$.
For a torsion free sheaf $E$ on an irreducible curve $C \in |nH|$,
$\Phi(E)[-1] \in \Coh(X')$.
\end{NB}
Hence we obtain the following lemma.
\begin{lem}\label{lem:coh-FM}
For 
$$
v=e^\beta(r+\xi+a \varrho_X)=e^\beta(r+pH+qf+D+a \varrho_X), 
D \in N_H,
$$
we have 
$$
\overline{\Phi}(v)=e^{\beta'}(pr_0-\tfrac{r}{r_0}H'+ar_0 f'-D'-\frac{q}{r_0}\varrho_{X'}),\; 
D' \in N_{H'}.
$$
\end{lem}
We set ${\cal Q}:={\cal P}^{\vee}[1]$.
The coherent sheaves ${\cal P}$ and ${\cal Q}$ on $X \times X'$ are 
the universal families in the following sense.
\begin{equation}
M_{H'}^{\beta'}(e^{\beta'}(r_0 f'))=\{ {\cal Q}_{|\{ x \} \times X'} \mid x \in X\},\;
M_{H'}^{-\beta'}(e^{-\beta'}(r_0 f'))= \{{\cal P}_{|\{ x \} \times X'} \mid x \in X\}.
\end{equation}

\begin{equation}
M_{H}^{\beta}(e^{\beta}(r_0 f))=\{ {\cal P}_{|X \times \{ x' \}} \mid x' \in X' \},\;
M_{H}^{-\beta}(e^{-\beta}(r_0 f))= \{{\cal Q}_{|X \times \{ x' \}} \mid x' \in X' \}.
\end{equation}
$\Phi_{X' \to X}^{{\cal Q}^{\vee}[1]}:{\bf D}(X') \to {\bf D}(X)$
satisfies
$$
\Phi_{X' \to X}^{{\cal Q}^{\vee}[1]} \circ \Phi_{X \to X'}^{{\cal P}^{\vee}[1]}=\otimes {\cal O}_X(-K_X)[-1].
$$

We take $l,l' \in {\Bbb Z}$ such that 
\begin{equation}\label{eq:l}
lH \in \NS(X), l' H' \in \NS(X').
\end{equation}

\begin{NB}
\begin{lem}
For 
$$
v=e^\beta(r+\xi+a \varrho_X)=e^\beta(r+pH+qf+D+a \varrho_X), 
D \in N_H,
$$
we have 
$$
-\overline{\Phi}(v)^{\vee}=e^{-\beta'}(\tfrac{r}{r_0}H'+pr_0-\frac{q}{r_0}\varrho_{X'}+D'-ar_0 f'),\; 
D' \in N_{H'}.
$$
\end{lem}
\end{NB}

\subsection{Stability and Fourier-Mukai transforms.}

\begin{NB}
\begin{prop}\label{prop:relation}
Assume that $p<0$.
$\Phi_{X \to X'}^{{\cal P}^{\vee}[2]}$ induces an isomorphism
$$
{\cal M}_{H_f}^\beta (v)^{ss} \to {\cal M}_{(f',H')}^{\beta'+\epsilon H'}(v')^{ss},
$$
where $v'=\Phi_{X \to X'}^{{\cal P}^{\vee}[2]}(v)$ and $\epsilon>0$ is sufficiently small.
\end{prop}

We take $F \in  {\cal M}_{(f',H')}^{\beta'+\epsilon H'}(v')^{ss}$.
Since 
$\chi({\cal Q}_{|\{ x \} \times X'}(-\beta'-\epsilon H'))<0$,
$\Hom(F,{\cal Q}_{|\{ x \} \times X'})=0$.
Since $\epsilon$ is sufficiently small,
$F \in \overline{\cal M}_{(f',H')}^{\beta'}(v')$.
Hence $\Hom({\cal Q}_{|\{ x \} \times X'},F)=0$ except finitely many points $x \in X$.
Therefore $\Phi_{X' \to X}^{{\cal Q}^{\vee}[1]}(F)$ is torsion free. 
\end{NB}

\begin{lem}\label{lem:f-slope}
For $E_1, E_2 \in {\bf D}(X)$,
we set $F_1:=\Phi_{X \to X'}^{{\cal P}^{\vee}[1]}(E_1),
F_2:=\Phi_{X \to X'}^{{\cal P}^{\vee}[1]}(E_2)$.
Then
\begin{equation}
\begin{split}
\frac{1}{r_0}(\rk E_2 \rk F_1-\rk E_1 \rk F_2)=&
\rk E_2 (c_1(E_1) \cdot f)-\rk E_1 (c_1(E_2) \cdot f)\\
=&
\rk F_2 (c_1(F_1) \cdot f')-\rk F_1  (c_1(F_2) \cdot f').
\end{split}
\end{equation}
\end{lem}

\begin{NB}
\begin{equation}
\begin{split}
& e^{-(xH+yf+\delta)}(r+pH+qf+D+a^\beta \varrho_X)\\
=&
r+(p-rx)H+(q-ry)f+(D-r \delta)+
\left(a^\beta-qx-py-(\delta \cdot D)+rxy+r\frac{(\delta^2)}{2} \right)\varrho_X.
\end{split}
\end{equation}

\begin{equation}
\begin{split}
& e^{\frac{1}{xr_0^2}(H'-xyr_0^2 f'+r_0 \widehat{\delta}-\frac{(\delta^2)r_0^2 }{2} f')}
(pr_0-\frac{r}{r_0}H'+a^\beta r_0 f'-\widehat{D}-\frac{q}{r_0} \varrho_{X'})\\
=&
pr_0+\frac{p-rx}{xr_0} H'+r_0 \left(a^\beta-p y-\frac{(\delta^2)p}{2x}\right)f'
+(-\widehat{D}+
\frac{p}{x}\widehat{\delta})+
\frac{1}{xr_0} \left(a^\beta-qx-py-(\delta \cdot D)+rxy+r\frac{(\delta^2)}{2} \right)\varrho_{X'}.
\end{split}
\end{equation}

\begin{NB2}
Assume that $x<0$. 
If $\frac{p}{r}=\frac{p_1}{r_1}$ and $\frac{a^\beta-qx}{r}=\frac{a_1^\beta-q_1 x}{r_1}$,
then
\begin{equation}
\frac{q}{r}-\frac{q_1}{r_1}>0 \iff
\frac{-a^\beta}{r}-\frac{-a_1^\beta}{r_1} >0
\end{equation}
\end{NB2}
\end{NB}

Let $G \in K(X)_{\Bbb Q}$ be an element with
\begin{equation}\label{eq:G}
v(G)=\rk G e^{\beta+sH+yf+\alpha},\; (\alpha \in N_H).
\end{equation}
We set
$G':=\Phi_{X \to X'}^{{\cal P}^{\vee}[1]}(G)$.
Then
$$
v(G')=sr_0 \rk G e^{\beta'-\frac{1}{sr_0^2}(H'-syr_0^2 f'
+r_0 \alpha'-\frac{(\alpha^2)r_0^2 }{2} f')},\; \alpha' \in N_{H'}.
$$
\begin{NB}
$\rk G'=r_0 s \rk G$. 
\end{NB}

$$
\chi(G',F)=\chi(G,E)=\rk G (a^\beta-qs-py+rsy).
$$

\begin{NB}
We set ${\cal Q}:={\cal P}^{\vee}[1]$. Then 
$\chi(G',{\cal Q}_{|\{ x \} \times X'})=\chi(G,{\Bbb C}_x)=\rk G>0$.
\end{NB}

\begin{NB}
\begin{lem}\label{lem:spectral}
Let $E$ be a coherent sheaf on $X$.
\begin{enumerate}
\item[(1)]
We have an exact sequence
\begin{equation}
0 \to E_1 \to E \to E_2 \to 0
\end{equation}
such that 
\begin{equation}
\begin{split}
\Phi_{X \to X'}^{{\cal P}^{\vee}[1]}(E_1)=& H^0(\Phi_{X \to X'}^{{\cal P}^{\vee}[1]}(E)) \in \Coh(X'),\\
\Phi_{X \to X'}^{{\cal P}^{\vee}[2]}(E_2)=& H^1(\Phi_{X \to X'}^{{\cal P}^{\vee}[1]}(E)) \in \Coh(X').
\end{split}
\end{equation}
\item[(2)]
If $E_{|f}$ is a semistable vector bundle with
$\chi(G,E_{|f})>0$ for a general fiber $f$, then
$E_2$ is a fiber sheaf with $\chi(G,E_2) \leq 0$.
\item[(3)]
If $E_{|f}$ is a semistable vector bundle with
$\chi(G,E_{|f})<0$ for a general fiber $f$, then
$E_1$ is a fiber sheaf with $\chi(G,E_1) \geq 0$.
\end{enumerate}
\end{lem}

\begin{proof}
We set 
\begin{equation}
\begin{split}
E_1:=& H^1(\Phi_{X' \to X}^{{\cal P}} (H^0(\Phi_{X \to X'}^{{\cal P}^{\vee}}(E)))),\\
E_2:=& H^0(\Phi_{X' \to X}^{{\cal P}}( H^1(\Phi_{X \to X'}^{{\cal P}^{\vee}}(E)))).
\end{split}
\end{equation} 
Then (1) follows from \cite{Br:1}.

(2)
If $E_{|f}$ is a semistable vector bundle with
$\chi(G,E_{|f})>0$ for a general fiber $f$, then
$E_2$
is a fiber sheaf.
For a $G$-twisted stable fiber sheaf $A$ with $\chi(G,A)>0$,
$$
\Hom(A,E_2)=\Hom(\Phi_{X \to X'}^{{\cal P}^{\vee}[1]}(A),\Phi_{X \to X'}^{{\cal P}^{\vee}[2]}(E_2)[-1])=0.
$$
\begin{NB2}
A 0-dimensional sheaf is generated by ${\Bbb C}_x$ and 
$\Phi_{X \to X'}^{{\cal P}^{\vee}[1]}({\Bbb C}_x)={\cal P}_{|\{ x \} \times X'}^{\vee}[1] \in \Coh(X')$.
\end{NB2}
Hence $\chi(G,E_2) \leq 0$.

(3)
If $E_{|f}$ is a semistable vector bundle with
$\chi(G,E_{|f})<0$ for a general fiber $f$, then
$E_1$
is a fiber sheaf.
For a $G$-twisted stable fiber sheaf $A$ with $\chi(G,A)<0$,
$$
\Hom(E_1,A)=\Hom(\Phi_{X \to X'}^{{\cal P}^{\vee}[1]}(E_1), 
\Phi_{X \to X'}^{{\cal P}^{\vee}[2]}(A)[-1])=0.
$$
\begin{NB2}
A 0-dimensional sheaf is generated by ${\Bbb C}_x$ and 
$\Phi_{X \to X'}^{{\cal P}^{\vee}[1]}({\Bbb C}_x)={\cal P}_{|\{ x \} \times X'}^{\vee}[1] \in \Coh(X')$.
\end{NB2}
Hence $\chi(G,E_1) \geq 0$.

\end{proof}
\end{NB}

\begin{NB}
\begin{proof}
We set 
$v(E_i):=e^\beta (r_i+p_i H+q_i f+D_i+a_i^\beta \varrho_X)$.
Then
\begin{equation}
(c_1(F_i) \cdot f')=-\frac{r_i}{r_0}+p_i r_0 (\beta' \cdot f'),\;
\rk F_i= p_i r_0= r_0 ((c_1(E_i)-(\rk E_i)\beta)\cdot f)
\end{equation}
Hence we get the claim.
\end{proof}
\end{NB}

\begin{prop}\label{prop:FMstability1}
Assume that 
$s>0$ and $p-rs>0$ and $G$ is general with respect to 
$v$.
Let $E \in {\bf D}(X)$ be an object with $v(E)=v$.
Then the following conditions are equivalent.
\begin{enumerate}
\item
$E$ is $(f,G,H)$-semistable.
\item
$F:=\Phi_{X \to X'}^{{\cal P}^{\vee}[1]}(E)$ is $(f',G',H')$-semistable.
\end{enumerate}
\end{prop}

\begin{proof}
We note that
\begin{equation}\label{eq:condition1}
\chi({\cal P}_{|X \times \{ x' \}},E[1])=r_0 p>0,\; \chi({\cal Q}_{|\{ x \} \times X'},F)=-r<0
\end{equation}
where $x \in X$ and $x' \in X'$.
Assume that $E$ is $(f,G,H)$-semistable.
By our assumption on $s$ and $p$,
we see that 
\begin{equation}\label{eq:condition11}
\mu_{\beta',f'} (G')=-\frac{1}{sr_0^2}<-\frac{r}{r_0^2 p}=\mu_{\beta',f'}(F)<0.
\end{equation}
We shall check the conditions in Definition \ref{defn:fGH} to prove 
the $(f',G',H')$-semistability of $F$.
Since 
\begin{equation}\label{eq:condition12}
0<\mu_{\beta,f} (G)=s<\frac{p}{r}=\mu_{\beta,f} (E),
\end{equation}
$E \in \Coh(X)$ and $E_{|\pi^{-1}(\eta)}$ is a semistable vector bundle.
Since the Fourier-Mukai transforms on elliptic curves preserve semistability 
(up to shift) and $\rk F=pr_0>0$,
$F_{|{\pi'}^{-1}(\eta)}$ is a semistable vector bundle.
In particular $H^i(F)_{|{\pi'}^{-1}(\eta)}=0$ for $i \ne 0$.
Since $\chi(G,{\cal P}_{|X \times \{ x' \}})=-r_0 \rk G s<0$,
\begin{NB}
$\langle v(G),r_0 f e^\beta \rangle=r_0 \rk G s$.
\end{NB}
$(f,G,H)$-semistability of $E$ implies
$$
\Ext^2({\cal P}_{|X \times \{ x' \}},E)=\Hom(E,{\cal P}_{|X \times \{ x' \}})^{\vee}=0.
$$
Hence $F$ is a two-term complex $V_{-1} \to V_0$ of locally free sheaves.
In particular $H^{-1}(F)$ is torsion free.
Since $H^{-1}(F)_{|{\pi'}^{-1}(\eta)}=0$, we get $H^{-1}(F)=0$. Therefore
$F \in \Coh(X')$.

Assume that there is an exact sequence
\begin{equation}
0 \to F_1 \to F \to F_2 \to 0
\end{equation}
such that
\begin{equation}\label{eq:chi(F)}
(c_1(F_2) \cdot f') =\frac{\rk F_2}{\rk F}(c_1(F) \cdot f'),\; \chi(G',F_2) \leq \frac{\rk F_2}{\rk F} \chi(G',F).
\end{equation}
Replacing $F_2$ by its quotient, we may assume that
$\Hom(A,F_2)=0$ for all $G'$-semistable fiber sheaves $A$ with $\chi(G',A)>0$
(use Lemma \ref{lem:F^+}).
Then 
$\Hom({\cal Q}_{|\{ x \} \times X'},F_2)=0$ for all $x \in X$.
\begin{NB}
$\chi(G',{\cal Q}_{|\{ x \} \times X'})=\rk G' (\chi(e^{\beta'},v_0')+\frac{1}{sr_0^2}(H' \cdot f'))>0$.
\end{NB}
Since $(F_2)_{|{\pi'}^{-1}(\eta)}$ is a stable vector bundle with $\mu_{\beta',f'}(F_2)<0$
or $(F_2)_{|{\pi'}^{-1}(\eta)}=0$, we see that
$\Phi_{X' \to X}^{{\cal P}[1]}(F_2) \in \Coh(X)$. 
Therefore we have an exact sequence of coherent sheaves on $X$
\begin{equation}
0 \to E_1 \to E \to E_2 \to 0
\end{equation}
where $E_i:= \Phi_{X' \to X}^{{\cal P}[1]}(F_i)(K_X)$.
By \eqref{eq:chi(F)} and Lemma \ref{lem:f-slope},
we have
$$
\frac{\rk E_2}{\rk E}=\frac{\rk F_2}{\rk F},\;
\chi(G,E_2) \leq \frac{\rk E_2}{\rk E} \chi(G,E).
$$
By the $(f,G,H)$-semistability of $E$,
$\chi(G,E_2) = \frac{\rk E_2}{\rk E} \chi(G,E)$.
Since $G$ is general with respect to $v$, we get $v(E_2) \in {\Bbb Q}v$, and hence
$v(F_2) \in {\Bbb Q}v(F)$.
Therefore $F$ is $(f',G',H')$-semistable.

Conversely assume that $F$ is $(f',G',H')$-semistable.
By our assumption $F \in \Coh(X')$ and satisfies the conditions in Definition \ref{defn:fGH}. 
We shall prove that $E \in \Coh(X)$ and satisfies the conditions in Definition \ref{defn:fGH}. 
By the semistability of $F_{|{\pi'}^{-1}(\eta)}$,
$E_{|{\pi}^{-1}(\eta)}$ is a semistable vector bundle. 
In particular $H^i(E)_{|{\pi}^{-1}(\eta)}=0$ for $i \ne 0$.
Since ${\cal Q}_{|\{ x \} \times X'}=\Phi_{X \to X'}^{{\cal P}^{\vee}[1]}({\Bbb C}_x)$, 
$\chi(G', {\cal Q}_{|\{ x \} \times X'})=\chi(G,{\Bbb C}_x)>0$.
By the $(f',G',H')$-semistability of $F$,
$\Hom({\cal Q}_{|\{ x \} \times X'},F)=0$.
Hence $E=\Phi_{X' \to X}^{{\cal Q}^{\vee}[2]}(F)(K_X)$ 
is a two term complex $V_{-1} \to V_0$ of locally free sheaves.
\begin{NB}
Since $F_{|{\pi'}^{-1}(\eta)}$ is a semistable vector bundle,
there is a non-empty open subset $U$ of $C$ such that
$F_{|{\pi'}^{-1}(t)}$ are semistable for all $t \in U$. In particular 
$F_{|{\pi'}^{-1}(U)}$ is locally free.
Then we have $\Hom({\cal Q}_{| \{ x \} \times X'},F[1])=H^0(X',{\cal P}_{|\{ x \} \times X'} \otimes F)=0$
if $\pi(x) \in U$.
\end{NB}
Hence we get $E \in \Coh(X)$.
Assume that there is an exact sequence
\begin{equation}
0 \to E_1 \to E \to E_2 \to 0
\end{equation}
such that 
\begin{equation}\label{eq:chi(E)}
(c_1(E_1) \cdot f) =\frac{\rk E_1}{\rk E}(c_1(E) \cdot f), \;
\chi(G,E_1) \geq \frac{\rk E_1}{\rk E} \chi(G,E).
\end{equation}
Replacing $E_1$ by its subsheaf, we may assume that
$\Hom(E_1,A)=0$ for all $G$-semistable fiber sheaves $A$ with $\chi(G,A)<0$
(use Lemma \ref{lem:F^-} and Lemma \ref{lem:F^+}).
Then 
$\Hom(E_1,{\cal P}_{|X \times \{ x' \}})=0$ for all $x' \in X'$.
\begin{NB}
$\chi(G,{\cal P}_{|X \times \{ x' \}})=\rk G (\chi(e^{\beta},v_0)-s(H \cdot f))<0$.
\end{NB}
Since $(E_1)_{|\pi^{-1}(\eta)}$ is a stable vector bundle with $\mu_{\beta,f}(E_1)>0$
or $(E_1)_{|\pi^{-1}(\eta)}=0$, we see that
$\Phi_{X \to X'}^{{\cal P}^{\vee}[1]}(E_1) \in \Coh(X')$. 

Hence we have an exact sequence

\begin{equation}
0 \to F_1 \to F \to F_2 \to 0.
\end{equation}
where $F_i:=\Phi_{X \to X'}^{{\cal P}^{\vee}[1]}(E_i)$.
By \eqref{eq:chi(E)} and Lemma \ref{lem:f-slope},
we have
$$
\frac{\rk E_1}{\rk E}=\frac{\rk F_1}{\rk F},\;
\chi(G',F_1) \geq \frac{\rk F_1}{\rk F} \chi(G',F).
$$
By the $(f',G',H')$-semistability of $F$,
$\chi(G',F_1) = \frac{\rk F_1}{\rk F} \chi(G',F)$, which implies
$\chi(G,E_1) = \frac{\rk E_1}{\rk E} \chi(G,E)$.
Since $G$ is general with respect to $v$, $v(E_1) \in {\Bbb Q}v$.
Therefore $E$ is $(f,G,H)$-semistable.
\end{proof}

\begin{rem}\label{rem:Gieseker}
Assume that $0<s \ll 1$.
Since $-\frac{1}{r_0^2 s}$ is sufficiently small,
Proposition \ref{prop:FMstability1} and Proposition \ref{prop:Gieseker-chamber}
imply that
$E$ is $(f,G,H)$-semistable if and only if 
$\Phi_{X \to X'}^{{\cal P}^{\vee}[1]}(E)$ is $G'$-twisted semistable with respect to
$H'+mf'$ $(m \gg 0)$.
\end{rem}

\begin{rem}
If $\alpha=0$, then the equivalence still holds even if $G$ is not general (see Lemma \ref{lem:special}).
\end{rem}

\begin{lem}\label{lem:special}
Let $E_i$ $(i=1,2)$ be coherent sheaves on $X$ such that $\rk E_i >0$,
$\mu_{f}(E_1)=\mu_{f}(E_2)$ and $\frac{\chi(G,E_1)}{\rk E_1}=\frac{\chi(G,E_2)}{\rk E_2}$.
We set $F_i:=\Phi_{X \to X'}^{{\cal P}^{\vee}[1]}(E_i)$.
Assume that $s>0$ and $\alpha=0$. Then
$$
\frac{(c_1(F_1) \cdot H')}{\rk F_1}>\frac{(c_1(F_2) \cdot H')}{\rk F_2} \iff
\frac{(c_1(E_1) \cdot H')}{\rk E_1}>\frac{(c_1(E_2) \cdot H')}{\rk E_2}. 
$$
\end{lem}

\begin{proof}
We set 
$$
v(E_i)=e^\beta (r_i+p_i H+q_i f+D_i+a_i \varrho_X),\, D_i \in N_H.
$$ 
Then
\begin{equation}
\begin{split}
0=\mu_{f}(E_1)-\mu_{f}(E_2)=&\frac{p_1}{r_1}-\frac{p_2}{r_2},\\
0=\frac{\chi(G,E_1)}{\rk E_1}-\frac{\chi(G,E_2)}{\rk E_2}=& \rk G \left(\frac{a_1-q_1 s}{r_1}-\frac{a_2-q_2 s}{r_2} \right),\\
\frac{(c_1(E_1) \cdot H)}{\rk E_1}-\frac{(c_1(E_2) \cdot H)}{\rk E_2}=&\frac{q_1}{r_1}-\frac{q_2}{r_2},\\
 \frac{(c_1(F_1) \cdot H')}{\rk F_1}-\frac{(c_1(F_2) \cdot H')}{\rk F_2}=&\frac{a_1}{p_1}-\frac{a_2}{p_2}.
\end{split}
\end{equation}
Since
$$
s \left(\frac{q_1}{r_1}-\frac{q_2}{r_2} \right)=\frac{p_1}{r_1}\left(\frac{a_1}{p_1}-\frac{a_2}{p_2} \right)-
\left(\frac{a_1-q_1 s}{r_1}-\frac{a_2-q_2 s}{r_2} \right),
$$
we get our claim.
\end{proof}

\begin{prop}\label{prop:FMstability2}
Assume that 
$p<0$ and $p-rs>0$ and $G$ is general with respect to 
$v$.
Let $E \in {\bf D}(X)$ be an object with $v(E)=v$.
Then the following cnditions are equivalent.
\begin{enumerate}
\item
$E$ is $(f,G,H)$-semistable.
\item
$F:=\Phi_{X \to X'}^{{\cal P}^{\vee}[2]}(E)$ is $(f',G'',H')$-semistable, where
$G'':=\Phi_{X \to X'}^{{\cal P}^{\vee}[2]}(G)$.
\end{enumerate}
\end{prop}

\begin{NB}
Relation with Prop \ref{prop:relation}:

$x \ll 0$ corresponds to $0<\epsilon -\frac{1}{xr_0^2}$.

\end{NB}

\begin{proof}
We note that $\rk G''=-r_0 \rk G s>0$.
We shall apply Proposition \ref{prop:FMstability1} after replacing
the quadruplet $(X,G,\Phi_{X \to X'}^{{\cal P}^{\vee}[1]},E)$ by 
$(X',G'',\otimes {\cal O}_X(K_X) \circ \Phi_{X' \to X}^{{\cal Q}^{\vee}[1]},F)$.
We first note that 
$\Phi_{X' \to X}^{{\cal Q}^{\vee}[1]}=\Phi_{X' \to X}^{{\cal P}}$, and hence
$G=\Phi_{X' \to X}^{{\cal Q}^{\vee}[1]}(G'')(K_X)$ and
$E=\Phi_{X' \to X}^{{\cal Q}^{\vee}[1]}(F)(K_X)$.
By the conditions $p<0$ and $p-rs>0$, we have
\begin{equation}
\begin{split}
0<\mu_{\beta',f'}(G'')=-\frac{1}{r_0^2 s}<& -\frac{r}{r_0^2 p}=\mu_{\beta',f'}(F)\\
\mu_{\beta,f}(G)=s<& \frac{p}{r}=\mu_{\beta,f}(E)<0,
\end{split}
\end{equation}
which correspond to conditions \eqref{eq:condition11}, \eqref{eq:condition12}.
Therefore our claim follows from Proposition \ref{prop:FMstability1}.

\begin{NB}
We note that $\chi(G,{\cal P}_{|X \times \{ x' \}})>0$.
By the $(f,G,H)$-semistability of $E$,
we get $\Hom({\cal P}_{|X \times \{ x' \}},E)=0$.
Hence $F=\Phi_{X \to X'}^{{\cal P}^{\vee}[2]}(E)$ is a two-term complex of locally free sheaves.
Since $\Ext^1({\cal P}_{|X \times \{ x' \}},E)=0$ for a general $x' \in X'$, we get
$F \in \Coh(X')$.
Assume that there is an exact sequence
\begin{equation}
0 \to F_1 \to F \to F_2 \to 0
\end{equation}
such that
$$
(c_1(F_1) \cdot f') =\frac{\rk F_1}{\rk F}(c_1(F) \cdot f'), \;
\chi(G'',F_1) \leq \frac{\rk F_1}{\rk F} \chi(G'',F).
$$
Applying Lemma \ref{lem:spectral} to $\Phi_{X' \to X}^{{\cal Q}^{\vee}[1]}$,
we may assume that
$\Phi_{X' \to X}^{{\cal Q}^{\vee}[1]}(F_1) \in \Coh(X)$.
Therefore we have an exact sequence of coherent sheaves on $X$
\begin{equation}
0 \to E_1 \to E \to E_2 \to 0
\end{equation}
where 
$E_i=\Phi_{X' \to X}^{{\cal P}}(F_i)(K_X)$.
\begin{NB2}
$\chi(G'',F_i)=\chi(G,E_i)$.
\end{NB2}
Hence 
$$
\chi(G'',F_1) = \frac{\rk F_1}{\rk F} \chi(G'',F),\;
(c_1(F_1) \cdot H') \leq  \frac{\rk F_1}{\rk F} (c_1(F) \cdot H').
$$
Therefore $F$ is $(f',G'',H')$-semistable.

Conversely for a $(f',G'',H')$-semistable sheaf $F$,
$\chi(G'', {\cal Q}_{|\{ x \} \times X'})=-\chi(G,{\Bbb C}_x)<0$ implies 
$$
\Ext^2({\cal Q}_{|\{ x \} \times X'},F)=
\Hom(F,{\cal Q}_{|\{ x \} \times X'})^{\vee}=0.
$$
Hence $E:=\Phi_{X' \to X}^{{\cal P}}(F)(K_X) \in \Coh(X)$.
Assume that there is an exact sequence
\begin{equation}
0 \to E_1 \to E \to E_2 \to 0
\end{equation}
such that 
$$
(c_1(E_2) \cdot f) =\frac{\rk E_1}{\rk E}(c_1(E) \cdot f),\; \chi(G,E_2) \leq \frac{\rk E_2}{\rk E} \chi(G,E).
$$
By Lemma \ref{lem:spectral},
we may assume that $\Phi_{X \to X'}^{{\cal P}^{\vee}[2]}(E_2) \in \Coh(X')$.
Hence we have an exact sequence

\begin{equation}
0 \to \Phi_{X \to X'}^{{\cal P}^{\vee}[2]}(E_1)(K_X) \to F \to 
\Phi_{X \to X'}^{{\cal P}^{\vee}[2]}(E_2)(K_X) \to 0.
\end{equation}
\end{NB}
\end{proof}

\begin{NB}
$G^{\#}={\bf R}\Hom_{p_{X'}}(p_X^*(G(-K_X)),{\cal P})(-K_{X'})=
{\bf R}\Hom_{p_{X'}}(p_X^*(G),{\cal P})$.
\end{NB}
\begin{prop}\label{prop:FMstability3}
Assume that 
$s<0$ and $p>0$ and $G$ is general with respect to 
$v$.
Let $E \in {\bf D}(X)$ be an object with $v(E)=v$.
Then the following conditions are equivalent.
\begin{enumerate}
\item
$E$ is $(f,G,H)$-semistable.
\item
$F:={\bf R}\Hom_{p_{X'}}(p_X^*(E),{\cal P}[1])=
\Phi_{X \to X'}^{{\cal P}^{\vee}[1]}(E)^{\vee}(-K_{X'})$ is $(f',-{G'}^{\vee}(-2K_{X'}),H')$-semistable.
\item
$\Phi_{X \to X'}^{{\cal P}^{\vee}[1]}(E)$ is $(f',-G',H')$-semistable.
\end{enumerate}
\end{prop}

\begin{NB}
$p_{X' !}(p_X^*(G^{\vee}) \otimes {\cal P})=-{G'}^{\vee}(-K_{X'})$.
\end{NB}

\begin{NB}
By the FM, the condition $x<0,p>0$ is preserved:
$(x,p) \mapsto (1/xr_0^2, r/r_0)$.
\end{NB}

\begin{proof}
(i) $\iff$ (ii).
Let $E$ be a $(f,G,H)$-semistable object with $v(E)=v$.
Our assumptions imply that
\begin{equation}
\mu_{\beta,f}(G)=s<0<\frac{p}{r}=\mu_{\beta,f}(E),
\end{equation}
In particular $E \in \Coh(X)$.
We shall prove that $F$ is $(f',-{G'}^{\vee}(-K_{X'}),H')$-semistable.
Since
\begin{equation}
\mu_{-\beta',f'}(-{G'}^{\vee}(-2K_{X'}))=\frac{1}{r_0^2 s}<0<
\frac{r}{r_0^2 p}=\mu_{-\beta',f'}(F),
\end{equation}
we need to prove that $F \in \Coh(X')$ and satisfies the conditions in Definition \ref{defn:fGH}.

Since $E_{|\pi^{-1}(\eta)}$ is a semistable vector bundle and $\rk F=pr_0>0$,
$F_{|{\pi'}^{-1}(\eta)}$ is a semistable vector bundle.
In particular 
\begin{equation}\label{eq:gen-fiber}
\Hom_{p_{X'}}(p_X^*(E),{\cal P}[i])_{|{\pi'}^{-1}(\eta)}=0\;\; (i \ne 1).
\end{equation}
Since $\chi(G,{\cal P}_{|X \times \{ x' \}})=-r_0 \rk G s>0$,
\begin{NB}
$\langle v(G),r_0 f e^\beta \rangle=r_0 \rk G s$.
\end{NB}
the $(f,G,H)$-semistability of $E$ implies
$$
\Ext^2(E,{\cal P}_{|X \times \{ x' \}})=\Hom({\cal P}_{|X \times \{ x' \}},E)^{\vee}=0.
$$
Hence $F$ is a two-term complex of locally free sheaves.
Since $\Hom_{p_{X'}}(p_X^*(E),{\cal P})$ is torsion free,  
\eqref{eq:gen-fiber} implies
$\Hom_{p_{X'}}(p_X^*(E),{\cal P})=0$. Therefore
$F \in \Coh(X')$.

Assume that there is an exact sequence
\begin{equation}
0 \to F_1 \to F \to F_2 \to 0
\end{equation}
such that
\begin{equation}\label{eq:chi(F)2}
(c_1(F_2) \cdot f') =\frac{\rk F_2}{\rk F}(c_1(F) \cdot f'),\; 
-\chi({G'}^{\vee}(-2K_{X'}),F_2) \leq -\frac{\rk F_2}{\rk F} \chi({G'}^{\vee}(-2K_{X'}),F).
\end{equation}

\begin{NB}
We note that $\chi(-{G'}^{\vee}(-2K_{X'}),{\cal P}_{| \{x \} \times X'})=\rk G>0$.
\end{NB}
Replacing $F_2$ by its quotient, we may assume that
$\Hom(A,F_2)=0$ for all $-{G'}^{\vee}$-semistable fiber sheaves $A$ with $\chi(-{G'}^{\vee},A)>0$.
Then $\Hom({\cal P}_{|\{ x \} \times X'},F_2)=0$ for all $x \in X$. Hence 
$E_2:={\bf R}\Hom_{p_{X}}(p_{X'}^*(F_2),{\cal P}[1]) \in \Coh(X)$.
Therefore we have an exact sequence of coherent sheaves on $X$
\begin{equation}
0 \to E_2 \to E \to E_1 \to 0
\end{equation}
where $E_1:={\bf R}\Hom_{p_{X}}(p_{X'}^*(F_1),{\cal P}[1])$.
We note that
${\bf R}\Hom_{p_{X'}}(p_X^*(G),{\cal P}[1])={G'}^{\vee}(-K_{X'})$.
Hence
\begin{equation}
\begin{split}
\chi(G,E)=& \chi(F,{G'}^{\vee}(-K_{X'}))=
\chi({G'}^{\vee}(-2K_{X'}),F),\\
\chi(G,E_2)=& \chi(F_2,{G'}^{\vee}(-K_{X'}))=
\chi({G'}^{\vee}(-2K_{X'}),F_2).
\end{split}
\end{equation}
By \eqref{eq:chi(F)2} and Lemma \ref{lem:f-slope},
we have
$$
\frac{\rk E_2}{\rk E}=\frac{\rk F_2}{\rk F},\;
-\chi(G,E_2) \leq -\frac{\rk E_2}{\rk E} \chi(G,E).
$$
Hence $\chi({G'}^{\vee},F_2) = \frac{\rk F_2}{\rk F} \chi({G'}^{\vee},F)$.
Since $G$ is general, 
$v(E_2) \in {\Bbb Q}v$, and hence $v(F_2) \in {\Bbb Q}v(F)$.
\begin{NB}
According to Definition \ref{defn:wall}, we have 
$v(F_2)= \frac{\rk F_2}{\rk F} v(F)$.
It seems to be sufficient to use Defn. \ref{defn:weak-wall}. 
\end{NB}
Therefore $F$ is $(f',-{G'}^{\vee}(-2K_{X'}),H')$-semistable.

Conversely for a $(f',-{G'}^{\vee}(-2K_{X'}),H')$-semistable sheaf $F$,
$-\chi({G'}^{\vee}(-2K_{X'}),{\cal P}_{|\{ x \} \times X'})=\chi(G,{\Bbb C}_x)>0$ implies
$\Hom({\cal P}_{|\{ x \} \times X'},F)=0$.
Hence 
$E:={\bf R}\Hom_{p_X}(p_{X'}^*(F),{\cal P}) \in \Coh(X)$.
Assume that there is an exact sequence
\begin{equation}
0 \to E_1 \to E \to E_2 \to 0
\end{equation}
such that 
\begin{equation}\label{eq:chi(E)2}
(c_1(E_2) \cdot f) =\frac{\rk E_2}{\rk E}(c_1(E) \cdot f), \;
\chi(G,E_2) \leq \frac{\rk E_2}{\rk E} \chi(G,E).
\end{equation}
Replacing $E_2$ by its quotient,
we may assume that $\Hom(A,E_2)=0$ for all $G$-twisted semistable fiber sheaves $A$
with $\chi(G,A)>0$.
Hence 
we may assume that 
${\bf R}\Hom_{p_{X'}}(p_X^*(E_2),{\cal P}) \in \Coh(X')$.
\begin{NB}
$\chi(G,{\cal P}_{|X \times \{ x' \}})=-\rk G sr_0>0$.
\end{NB}
Hence we have an exact sequence

\begin{equation}
0 \to F_2 \to F \to F_1 \to 0.
\end{equation}
where $F_i:={\bf R}\Hom_{p_{X'}}(p_X^*(E_i),{\cal P})$. 
By \eqref{eq:chi(E)2} and Lemma \ref{lem:f-slope},
we have
$$
\frac{\rk E_2}{\rk E}=\frac{\rk F_2}{\rk F},\;
\chi({G'}^{\vee}(-2K_{X'}),F_2) \leq 
\frac{\rk F_2}{\rk F} \chi({G'}^{\vee}(-2K_{X'}),F).
$$
In the same way as above, we see that $E$ is $(f,G,H)$-semistable.

(ii) $\iff$ (iii).
It easily follows from Definition \ref{defn:fGH2}.
\end{proof}

\begin{NB}
By the Grothendieck duality, we have a commutative diagram
\begin{equation}
\begin{CD}
{\bf D}(X) @>{\Phi_{X \to X'}^{{\cal P}^{\vee}[1]}}>> {\bf D}(X')\\
@V{(\otimes {\cal O}_X(K_X)) \circ D_X }VV @VV{D_{X'}}V\\
{\bf D}(X) @>>{\Phi_{X \to X'}^{{\cal Q}^{\vee}[2]}}> {\bf D}(X')
\end{CD}
\end{equation}

The following conditions are equivalent
\begin{enumerate}
\item
$E$ is $(f,G,H)$-semistable if and only if $\Phi_{X \to X'}^{{\cal P}^{\vee}[1]}(E)$ is
$(f',G',H')$-semistable
\item
$E^{\vee} (K_X)$ is $(f,G^{\vee},H)$-semistable if and only if 
$\Phi_{X \to X'}^{{\cal Q}^{\vee}[2]}(E^{\vee}(K_X))$ is
$(f',G^{\#},H')$-semistable, where 
$G^{\#}:=\Phi_{X \to X'}^{{\cal Q}^{\vee}[2]}(G^{\vee})=(G')^{\vee}(-K_{X'})$.
\end{enumerate}
\end{NB}

\begin{prop}\label{prop:FMstability4}
Assume that $p=0$ and $s<0$ and
$G$ is general with respect to $v$.
Let $E \in {\bf D}(X)$ be an object with $v(E)=v$.
Then the following conditions are equivalent.
\begin{enumerate}
\item
$E$ is $(f,G,H)$-semistable.
\item
$F:=\Phi_{X \to X'}^{{\cal P}^{\vee}[2]}(E)$ is $-G'$-twisted semistable
with respect to $H'+nf'$.
\end{enumerate}
\end{prop}

\begin{proof}
We note that $\Hom({\cal P}_{|X \times \{ x' \}},E)=0$ for all $x' \in X'$.
\begin{NB}
$\chi(G,{\cal P}_{|X \times \{ x' \}})=-sr_0 \rk G>0$.
\end{NB}
Hence $F$ is a two-term complex of locally free sheaves.
Since $\rk F=0$,
$F$ is a purely 1-dimensional sheaf.
For an exact sequence
$$
0 \to F_1 \to F \to F_2 \to 0,
$$
we assume that
$$
\chi(-G',F_1) \geq (c_1(F_1) \cdot f') \frac{\chi(-G',F)}{(c_1(F) \cdot f')}.
$$
We may assume that $E_1:=\Phi_{X' \to X}^{{\cal Q}^{\vee}[1]}(F_1) \in \Coh(X)$.
\begin{NB}
Since $\chi(-G',{\cal Q}_{|\{ x \} \times X'})=-\chi(G,{\Bbb C}_x)=-\rk G<0$,
$\Ext^2({\cal Q}_{|\{ x \} \times X'},E_1)=\Hom(E_1,{\cal Q}_{|\{ x \} \times X'})^{\vee}=0$
for all $x \in X$.
\end{NB}
Then $E_2:=\Phi_{X' \to X}^{{\cal Q}^{\vee}[1]}(F_2) \in \Coh(X)$ and 
we have an exact sequence
$$
0 \to E_1 \to E \to E_2 \to 0.
$$
We set 
$$
v(E_i):=e^\beta (r_i+q_i f+D_i+a_i \varrho_X),\; D_i \in N_H,\;\; (i=1,2).
$$
Then 
$$
\chi(G,E_1) \geq \frac{r_1}{r}\chi(G,E).
$$
By the $(f,G,H)$-semistability of $E$,
we get $\chi(G,E_1)= \frac{r_1}{r}\chi(G,E)$.
Since $G$ is general,
$v(E_1) \in {\Bbb Q}v$, which implies 
$F$ is $G'$-twisted semistable.

Conversely we assume that $F$ is $G'$-twisted semistable.
Then
$$
\frac{\chi(-G',F)}{(c_1(F) \cdot f')}(c_1(F_2) \cdot f') \leq \chi(-G',F_2)
$$
for all quotient $F \to F_2$.
In particular $\chi(-G',F_2) \geq 0$ if $F_2$ is a fiber sheaf.
Since $\chi(-G',{\cal Q}_{|\{ x \} \times X'})=-\rk G<0$,
$\Hom({\cal Q}_{|\{ x \} \times X'},F[2])=\Hom(F, {\cal Q}_{|\{ x \} \times X'})^{\vee}=0$.
Hence $E:=\Phi_{X' \to X}^{{\cal Q}^{\vee}[1]}(F)$ is a two-term complex of locally free sheaves.
Since $E_{\pi^{-1}(\eta)}$ is a semistable vector bundle,
$E \in \Coh(X)$.
Assume that there is an exact sequence
$$
0 \to E_1 \to E \to E_2 \to 0
$$
such that
$$
(c_1(E_2) \cdot f)=\rk E_2 \frac{(c_1(E) \cdot f)}{\rk E},\;
\chi(G,E_2) \leq \rk E_2 \frac{\chi(G,E)}{\rk E}.
$$
\begin{NB}
We note that $\chi(G,{\cal P}_{|X \times \{ x' \}})=-s\rk G>0$.
\end{NB}
Replacing $E_2$ by its quotient, we may assume that
$F_2:=\Phi_{X \to X'}^{{\cal P}^{\vee}[2]}(E_2) \in \Coh(X')$.
Then $F_1:=\Phi_{X \to X'}^{{\cal P}^{\vee}[2]}(E_1) \in \Coh(X')$ and
we have an exact sequence
$$
0 \to F_1 \to F \to F_2 \to 0
$$
with
$$
\chi(-G',F_2) \leq (c_1(F_2) \cdot f') \frac{\chi(-G',F)}{(c_1(F) \cdot f')}.
$$
Hence $\chi(G,E_2) = \rk E_2 \frac{\chi(G,E)}{\rk E}$.
Since $G$ is general, we get $v(E_2) \in {\Bbb Q}v(E)$. Therefore
$E$ is $(f,G,H)$-semistable.
\end{proof}

\begin{thm}\label{thm:FMpreserve}
Let $\Phi_{X \to X'}^{{\cal P}^{\vee}[1]}$ be the Fourier-Mukai transform in \eqref{eq:relativeFM}.
Let $v=r+\xi+a \varrho_X$ $(r>0, \xi \in \NS(X))$ be a Mukai vector 
such that $\langle v,v_0 \rangle \ne 0$.
We take $G \in K(X)_{\Bbb Q}$ such that
$\rk G>0$ and $(f,G,H)$ is general with respect to $v$.
For an object $E \in {\bf D}(X)$ with $v(E)=v$, the following conditions are equivalent.
\begin{enumerate}
\item
$E$ is $(f,G,H)$-semistable.
\item
$F:=\Phi_{X \to X'}^{{\cal P}^{\vee}[k]}(E)$ is $(f',\epsilon G',H')$-semistable,
 where $k=1$ or $2$ according as $\langle v,v_0 \rangle>0$ or $\langle v,v_0 \rangle<0$
and 
$\epsilon=1$ or $-1$ according as
$\rk G'>0$ or $\rk G'<0$.
\end{enumerate}
\end{thm}

\begin{proof}
As in subsection \ref{subsect:FM},
we may assume that 
$$
\beta=\frac{b}{r_0}H+\beta_0 \,\, (\beta_0 \in N_H).
$$
We set 
$$
c_1(G)=\rk G (\beta+sH+yf+\alpha),\; \alpha \in N_H
$$
and
$$
v=e^\beta (r+pH+qf+D+a \varrho_X),\; D \in N_H. 
$$
We note that 
\begin{equation}
k=1 \iff p>0,\quad
\epsilon=1 \iff s>0.   
\end{equation}

(I) Assume that $p>0$.
There are three cases.
\begin{equation*}
(a)\; s<0<\frac{p}{r}, \quad
(b)\; 0<s<\frac{p}{r},\quad 
(c)\; 0<\frac{p}{r}<s.
\end{equation*}

(a) We first assume that $s<0<\frac{p}{r}$.
In this case, $k=1$ and $\epsilon=-1$.
Then the claim follows from Proposition \ref{prop:FMstability3}.

(b) 
We next assume that $0<s<\frac{p}{r}$.
Then the claim follows from Proposition \ref{prop:FMstability1}.

(c) 
Finally we assume that $0<\frac{p}{r}<s$. 
In this case, $k=1$ and $\epsilon=1$.
We note that $E$ is $(f,G,H)$-semistable if and only if 
$E^{\vee}(K_X)$ is a $(f,G^{\vee},H)$-semistable (Definition \ref{defn:fGH2}).
Since $\frac{p}{r}<s$, we shall consider $E^{\vee}(K_X)$.
\begin{NB}
$0>\frac{-p}{r}>-s$.
\end{NB}
We have
\begin{equation}
\begin{split}
v(E^{\vee})=& e^{-\beta}(r-pH-qf-D+a \varrho_X)\\
c_1(G^{\vee})=& \rk G (-\beta-sH-yf-\alpha),
\end{split}
\end{equation}
and $-s<-\frac{p}{r}<0$.
We first apply $\Phi_{X \to X'}^{{\cal Q}^{\vee}[2]}:{\bf D}(X) \to {\bf D}(X')$
to $E^{\vee}(K_X)$. 
Since
\begin{equation}\label{eq:QG2}
\Phi_{X \to X'}^{{\cal Q}^{\vee}[2]}(G^\vee)
={\bf R}\Hom_{p_{X'}}(p_X^*(G),{\cal P}[1])={G'}^{\vee}(-K_{X'}),
\end{equation}
the following conditions are equivalent by
Proposition \ref{prop:FMstability2}.
\begin{enumerate}
\item
$E^{\vee}(K_X)$ is $(f,G^{\vee},H)$-semistable.
\item
$(\Phi_{X \to X'}^{{\cal P}^{\vee}[1]}(E))^{\vee}=
\Phi_{X \to X'}^{{\cal Q}^{\vee}[2]}(E^{\vee}(K_X))$
is $(f',{G'}^{\vee}(-K_{X'}),H')$-semistable.
\end{enumerate}
By Definition \ref{defn:fGH2}, the condition (ii) is equivalent to the $(f',G',H')$-semistability of
$F=\Phi_{X \to X'}^{{\cal P}^{\vee}[1]}(E)$.
Therefore our claim holds.


(II) Assume that $p<0$.
There are three cases.
\begin{equation*}
(d)\; s<\frac{p}{r}<0, \quad
(e)\; \frac{p}{r}<s<0,\quad 
(f)\; \frac{p}{r}<0<s.
\end{equation*}

\begin{NB}
Case (i) $\varphi(-\infty)=0<\varphi(s)<\varphi(p/r)<\infty=\varphi(0)$.

Case (ii) $\varphi(-\infty)=0<\varphi(p/r)<\varphi(s)<\infty=\varphi(0)$.

Case (iii) $\varphi(s)<\varphi(-\infty)=0<\varphi(p/r)<\infty=\varphi(0)$.
\end{NB}

(d) We first assume that $s<\frac{p}{r}<0$.
In this case $k=2$ and $\epsilon=-1$.
By Proposition \ref{prop:FMstability2},
$E$ is a $(f,G,H)$-semistable sheaf if and only if 
$F=\Phi_{X \to X'}^{{\cal P}^{\vee}[2]}(E)$ 
is a $(f',-G',H')$-semistable sheaf. Thus the claim holds.

(e) We next assume that $\frac{p}{r}<s<0$.
In this case $k=2$ and $\epsilon=-1$.
We shall apply $\Phi_{X \to X'}^{{\cal Q}^{\vee}[1]}$ to $E^{\vee}(K_X)$.
We note that 
\begin{equation}\label{eq:QG}
\Phi_{X \to X'}^{{\cal Q}^{\vee}[1]}(G^{\vee})={G'}^{\vee}(-K_{X'})[-1]
\end{equation}
by \eqref{eq:QG2}.
Then the following conditions are equivalent by Proppsition \ref{prop:FMstability1}.
\begin{enumerate}
\item
$E^{\vee}(K_X)$ is $(f,G^{\vee},H)$-semistable.
\item
$\Phi_{X \to X'}^{{\cal P}^{\vee}[2]}(E)^{\vee}=\Phi_{X \to X'}^{{\cal Q}^{\vee}[1]}(E^{\vee}(K_X))$
is $(f',-{G'}^{\vee}(-K_{X'}),H')$-semistable.
\end{enumerate}
By Definition \ref{defn:fGH2}, 
our claim holds.

\begin{NB}
$0<-s<\frac{-p}{r}$.
\end{NB}

(f) Assume that $\frac{p}{r}<0<s$.
\begin{NB}
$-s<0<\frac{-p}{r}$.
\end{NB}
In this case $k=2$ and $\epsilon=1$.
We shall apply $\Phi_{X \to X'}^{{\cal Q}^{\vee}[1]}$ to $E^{\vee}(K_X)$.
Then following conditions are equivalent by Proposition \ref{prop:FMstability3} and \eqref{eq:QG}.
\begin{enumerate}
\item
$E^{\vee}(K_X)$ is $(f,G^{\vee},H)$-semistable.
\item
$\Phi_{X \to X'}^{{\cal Q}^{\vee}[1]}(E^{\vee}(K_X))$
is $(f',{G'}^{\vee}(-K_{X'}),H')$-semistable.
\end{enumerate}
Since
$F=\Phi_{X \to X'}^{{\cal P}^{\vee}[2]}(E)=\Phi_{X \to X'}^{{\cal Q}^{\vee}[1]}(E^{\vee}(K_X))^{\vee}$,
the condition (ii) is equivalent to 
the $(f',G',H')$-semistability of $F$.
Therefore our claim holds in this case. 
\begin{NB}
$\Phi_{X \to X'}^{{\cal Q}^{\vee}[1]}(G^{\vee})={G'}^{\vee}(-K_{X'})[-1]$.

$(G')^{\vee}=(G^{\vee})'(K_{X'})[1]$.

$-(((G^{\vee})')^{\vee}(-K_{X'})=G'$.
\end{NB}
\end{proof}

If $\langle v,v_0 \rangle=0$, then we also have the following result.

\begin{thm}\label{thm:FMpreserve0}
Let $\Phi_{X \to X'}^{{\cal P}^{\vee}[1]}$ be the Fourier-Mukai transform in \eqref{eq:relativeFM}.
Let $v=r+\xi+a \varrho_X$ $(r>0, \xi \in \NS(X))$ be a Mukai vector 
such that $\langle v,v_0 \rangle = 0$.
We take $G \in K(X)_{\Bbb Q}$ such that
$\rk G>0$ and $(f,G,H)$ is general with respect to $v$.
For an object $E \in {\bf D}(X)$ with $v(E)=v$, the following conditions are equivalent.
\begin{enumerate}
\item
$E$ is $(f,G,H)$-semistable.
\item
$F:=\Phi_{X \to X'}^{{\cal P}^{\vee}[2]}(E)$ is a $-G'$-twisted semistable 1-dimensional
sheaf with respect to $H'_{f'}$.
\end{enumerate}
\end{thm}

\begin{proof}
In the notation of the proof of Theorem \ref{thm:FMpreserve},
we have two cases.
\begin{equation*}
(a)\; s<\frac{p}{r}=0, \quad
(b)\; \frac{p}{r}=0<s.
\end{equation*}
For case (a), the claim follows from Proposition \ref{prop:FMstability4}.
For case (b),
we note that $-s<-\frac{p}{r}=0$
and
 $E$ is $(f,G,H)$-semistable if and only if 
$E^{\vee}(K_X)$ is a $(f,G^{\vee},H)$-semistable (Definition \ref{defn:fGH2}).
We first apply $\Phi_{X \to X'}^{{\cal Q}^{\vee}[2]}:{\bf D}(X) \to {\bf D}(X')$
to $E^{\vee}(K_X)$. 
By \eqref{eq:QG2}, 
$\Phi_{X \to X'}^{{\cal Q}^{\vee}[2]}(G^\vee)={G'}^{\vee}(-K_{X'})$.
Hence the following conditions are equivalent by
Proposition \ref{prop:FMstability4}.
\begin{enumerate}
\item
$E^{\vee}(K_X)$ is $(f,G^{\vee},H)$-semistable.
\item
$(\Phi_{X \to X'}^{{\cal P}^{\vee}[2]}(E))^{\vee}[1]=
\Phi_{X \to X'}^{{\cal Q}^{\vee}[2]}(E^{\vee}(K_X))$
is $-{G'}^{\vee}(-K_{X'})$-twisted semistable.
\end{enumerate}
For a sheaf $F$ of dimension 1,
$F$ is $G$-twisted semistable if and only if
$F^{\vee}[1]$ is $G^{\vee}(-K_X)$-twisted semistable.
Therefore our claim holds.
\end{proof}

\begin{rem}\label{rem:reduce}
Assume that there is no multiple fiber.
Then for a primitive Mukai vector $v_0=rf+d\varrho_X$,
$X':=M_H^G(v_0)$ is a smooth projective surface and there is a universal family
${\cal P}$ as a $(1_X \times {\alpha'}^{-1})$-twisted sheaf on $X \times X'$, where
$\alpha'$ is a 2-cocycle of ${\cal O}_{X'}^{\times}$.
Then we have a Fourier-Mukai transform
$\Phi_{X \to X'}^{{\cal P}^{\vee}}:
{\bf D}(X) \to {\bf D}^{\alpha'}(X')$.
Hence we can find a suitable $v_0$
such that $\Phi_{X \to X'}^{{\cal P}^{\vee}}(E)$ is a $G'$-twisted semistable sheaf 
for any $E \in M_{(f,H)}^G(v)$. 
Thus our wall-crossing for $(f,G,H)$-semistability is reduced to a wall-crossing 
for the moduli of $G'$-twisted semistable 1-dimensional twisted sheaves.
On the other hand, if there is a multiple fiber, then
we do not have such a reduction in general.
For example, $r$ and the multiplicity of a multiple fiber is not relatively prime, then
we can not reduce to a wall-crossing 
for the moduli of $G'$-twisted semistable 1-dimensional twisted sheaves.
\end{rem}

\begin{cor}\label{cor:Gieseker}
Assume that $(f,G,H)$ is general with respect to $v=r+\xi+a\varrho_X$ $(r>0,\xi \in \NS(X))$.
Then there is a relative Fourier-Mukai transform
$\Phi_{X \to X'}^{{\cal P}^{\vee}[1]}:{\bf D}(X) \to {\bf D}(X')$ which induces an isomorphism
$$
{\cal M}_{(f,H)}^G(v)^{ss} \cong {\cal M}_{H'_{f'}}^{G'}(v')^{ss},
$$
 where $v'=\overline{\Phi}_{X \to X'}^{{\cal P}^{\vee}[1]}(v)$.
Moreover ${\cal M}_{H'_{f'}}^{G'}(v')^{ss}$ consists of $\mu$-stable vector bundles provided that
$v$ is primitive.
\end{cor}

\begin{proof}
For a chamber ${\cal C}$,
we take a Mukai vector $v_0:=(0,r_0 f,d_0)$ such that
\begin{enumerate}
\item
there is $\xi \in \NS(X)$ with $\gcd(r_0 (\xi \cdot f),d_0)=1$ and
\item
$\beta:=\frac{d_0}{r_0 (f \cdot H)}H+\alpha$
$(\alpha \in N_H)$ 
satisfies $\beta+sH \in {\cal C}$ for all $0 \leq s \ll 1$.
\end{enumerate}
We may assume that $\alpha$ is general with respect to $v_0$.
Then $X':=M_{H_f}^\beta(v_0)$ is smooth and has a universal family ${\cal P}$.
By Theorem \ref{thm:FMpreserve} and Remark \ref{rem:Gieseker},
$\Phi_{X \to X'}^{{\cal P}^{\vee}[1]}:{\bf D}(X) \to {\bf D}(X')$ induces an isomorphism
${\cal M}_{(f,H)}^G(v)^{ss} \cong {\cal M}_{H'_{f'}}^{G'}(v')^{ss}$, where $v'=\overline{\Phi}_{X \to X'}^{{\cal P}^{\vee}[1]}(v)$.

Assume that $v$ is primitive. 
Since $\beta \in {\cal C}$, $\langle e^\beta,(r_1 v-r v_1) \rangle \ne 0$
for any Mukai vector $v_1 \in U(v)$.
Assume that there is a subsheaf $F_1$ of $F:=\Phi_{X \to X'}^{{\cal P}^{\vee}[1]}(E)$ such that
\begin{equation}\label{eq:pss}
\frac{(c_1(F_1) \cdot (H'+nf'))}{\rk F_1}=\frac{(c_1(F) \cdot (H'+nf'))}{\rk F}
\end{equation}
for all $n \gg 0$.
Then $F_1=\Phi_{X \to X'}^{{\cal P}^{\vee}[1]}(E_1)$
with $v(E_1) \in U(v)$.
By Lemma \ref{lem:coh-FM} and \eqref{eq:pss}, we see that
$$
-\frac{\langle e^\beta,v(E_1) \rangle}{\rk E_1}=\frac{\rk F_1}{\rk E_1} 
\frac{(c_1(F_1(-\beta')) \cdot H')}{\rk F_1}=
\frac{\rk F}{\rk E} 
\frac{(c_1(F(-\beta')) \cdot H')}{\rk F}=
-\frac{\langle e^\beta,v(E) \rangle}{\rk E}.
$$
Hence there is no subsheaf $F_1$ with \eqref{eq:pss}, which implies that 
$F$ is a $\mu$-stable vector bundle for any $E \in {\cal M}_{(f,H)}^G(v)^{ss}$.
\begin{NB}
We set $F:=\Phi_{X \to X'}^{{\cal P}^{\vee}[1]}(E)$.
Then $(c_1(F(-\beta')) \cdot H')/\rk F=\frac{r}{p} \langle e^\beta,v \rangle/r$.
\end{NB}
\end{proof}

\begin{thm}\label{thm:FM2}
Assume that $X':=M_{H\f}^\alpha(0,r_0 f,d_0)$ is a fine moduli space.
We take a compact neighborhood $B^*$ of $\alpha$.
For a Mukai vector $v=r+\xi+a \varrho_X$ ($\xi \in\NS(X)$),
we set $d:=(\xi \cdot f)$. 
In the notation of Proposition \ref{prop:Gieseker-chamber},
if
$$
\frac{d_0}{r_0}<\frac{d}{r}-\frac{\mu (-(D^2)+\langle v^2 \rangle+r^2 \chi({\cal O}_X))}{2}-A,
$$
then $\Phi_{X \to X'}^{{\cal P}^{\vee}[1]}$ preserves the Gieseker semistability.
In particular $\Phi_{X \to X'}^{{\cal P}^{\vee}[1]}(E(mH))$ is Gieseker semistable
for a semistable sheaf $E$ if $m$ is sufficiently large.  
\end{thm}

\begin{proof}
We take a rational number $s$ with 
$$
\frac{d_0}{r_0}<s \ll \frac{d}{r}-\frac{\mu (-(D^2)+\langle v^2 \rangle+r^2 \chi({\cal O}_X))}{2}-A
$$
and set
$G:=e^{sH+\alpha}$.
Then $(f,G,H)$-semistability is equivalent to the $\alpha$-twisted Gieseker semistability for $v$.
Since $s-\tfrac{d_0}{r_0}>0$ is sufficiently small, 
Remark \ref{rem:Gieseker} implies the claim. 
\end{proof}

\begin{NB}
\begin{equation}
\frac{\chi(G,E)}{\rk E}-\frac{\chi(G,E_1)}{\rk E_1}
\sim
\frac{\chi(G',F)}{(c_1(F(-\beta')) \cdot (\frac{1}{n'}H'+f'))}-
\frac{\chi(G',F_1)}{(c_1(F_1 (-\beta')) \cdot (\frac{1}{n'}H'+f'))}
\end{equation} 
\end{NB}

\subsection{Line bundles}

In this subsection, we shall study some algebro-geometric properties of line bundles on
$M_{(f,H)}^G(v)$, which are minor modifications of \cite[sect. 5.1]{Y:elliptic}.
For a Mukai vector $v=r+\xi+a \varrho_X$ $(r>0, \xi \in \NS(X))$,
we set 
$$
K(X)_v:=\{\alpha \in K(X) \mid \langle v(\alpha),ve^{-\frac{K_X}{2}} \rangle=0 \}.
$$
\begin{NB}
$\langle v(\alpha),v(E) e^{-\frac{K_X}{2}} \rangle=-\chi(\alpha,E)$.
\end{NB}
Let ${\cal E}$ be a universal family on $M_{(f,H)}^G(v) \times X$.
We have a homomorphism
\begin{equation}
\begin{matrix}
\theta_v:& K(X)_v & \to & \Pic(M_{(f,H)}^G(v))\\ 
& \alpha & \mapsto & \det p_{M_{(f,H)}^G(v) !}({\cal E} \otimes p_X^*(\alpha^{\vee}))
\end{matrix}
\end{equation}
%
which is independent of the choice of ${\cal E}$.
We define a homomorphism $D^*$ by 
\begin{equation}
\begin{matrix}
D^*:& K(X) & \to & K(X)\\
& \alpha & \mapsto & \alpha^{\vee}(-K_X).
\end{matrix}
\end{equation} 
$D^*$ induces an isomorphism
$K(X)_v \to K(X)_{D_X(v)}$, where
$D_X(v):=r-\xi+a \varrho_X$. 
\begin{NB}
$\langle v(\alpha), ve^{-K_X/2} \rangle=\langle v(\alpha^{\vee}(-K_X)),v^{\vee}e^{-K_X/2} \rangle$.
\end{NB}

For the Fourier-Mukai transform 
$\Phi_{X \to X'}^{{\cal P}^{\vee}[k]}$ in Theorem \ref{thm:FMpreserve},
we have a commutative diagram
\begin{equation}\label{eq:comm}
\begin{CD}
K(X)_v @>{\Phi_{X \to X'}^{{\cal P}^{\vee}[k]}}>> K(Y)_{v'}\\
@V{\theta_v}VV @VV{\theta_{v'}}V\\
\Pic(M_{(f,H)}^G(v)) @= \Pic(M_{(f',H')}^{G'}(v'))
\end{CD}
\end{equation}
where $v'=\overline{\Phi}_{X \to X'}^{{\cal P}^{\vee}[k]}(v)$ and we identify $\Pic(M_{(f,H)}^G(v))$ with 
$\Pic(M_{(f',H')}^{G'}(v'))$ by $(\Phi_{X \to X'}^{{\cal P}^{\vee}[k]})_*$. 
We also have
\begin{equation}\label{eq:D}
\theta_{v}(\alpha)=\det p_{M_{(f,H)}^G(v) !} ({\cal E}^{\vee} \otimes p_X^*(\alpha(K_X)))^{\vee}
=\theta_{D_X(v)}(D^*(\alpha))^{\vee}.
\end{equation}

\begin{NB}
Let $q:S \times X \times Y \to S$ be the projection, 
$q_X$ and $q_Y$ are projections from $S \times X \to Y$ to $X$ and $Y$ respectively. 
For ${\cal E}' \in {\bf D}(S \times Y)$,
\begin{equation}
\begin{split}
\det p_{!}(\Phi_{Y \to X}^{{\bf P}^{\vee}}({\cal E}') \otimes p_X^*(\alpha^{\vee}))=&
\det p_{!}({\bf R}q_{S \times X*}({\bf P}^{\vee} \otimes {\cal E}') \otimes p_X^*(\alpha^{\vee}))\\
=& \det q_{!}({\bf P}^{\vee} \otimes {\cal E}' \otimes q_X^*(\alpha^{\vee}))\\
=& \det q_{!}({\cal E}' \otimes q_{X \times Y}^*(({\bf P} \otimes p_X^*(\alpha))^{\vee}))\\
=& \det p'_{!}({\cal E}' \otimes {\bf R}p_{Y*}({\bf P} \otimes p_X^*(\alpha(K_X)))^{\vee})\\
=& \det p'_{!}({\cal E}' \otimes \Phi_{X \to Y}^{{\bf P}}(\alpha(K_X))^{\vee})
\end{split}
\end{equation}
\end{NB}

\begin{prop}[]
Assume that $G$ is general with respect to $v$.
\begin{enumerate}
\item[(1)]
For $\alpha \in K(X)$ with 
$v(\alpha)=r f+(\xi \cdot f) \varrho_X$, 
$\theta_v(m \alpha)$ $(m \gg 0)$ is base point free.
\item[(2)]
If the Kodaira dimension of $X$ is 1, then
$K_{M_{H_f}({\bf e})}^{\otimes m}$
defines a morphism $M_{(f,H)}^G(v) \to S^l C$.
\end{enumerate}
\end{prop}

\begin{proof}
The proof is almost the same as in \cite[Prop. 5.6]{Y:elliptic}. So we only show (1).
By \eqref{eq:D} and Definition \ref{defn:fGH2}, we may assume that 
$\mu_f(G)<\frac{(\xi \cdot f)}{r}$.
Let $\delta$ be a very ample divisor on $C$.
For $E \in M_{(f,H)}^G(v)$, 
we take a divisor $\sum_{i=1}^N p_i \in |\delta|$ such that $\pi^{-1}(p_i)$
are smooth and $E_{|\pi^{-1}(p_i)}$ are semistable vector bundles.
We take semistable vector bundles $F_i$ of rank $r$ on $\pi^{-1}(p_i)$
such that $\chi(F_i)=(\xi \cdot f)$
and $\Hom(E,F_i(K_X))=0$ for all $i$.
\begin{NB}
Since $\Pic^0(C) \to \Pic^0(X)$ is an isomorphism,
$\Pic^0(X) \to \Pic^0(\pi^{-1}(p_i))$ are trivial.
If we choose $F_i$ to satisfy $\det F_i \not \cong \det E_{|\pi^{-1}(p_i)}$,
then $\det F_i \not \cong \det E'_{|\pi^{-1}(p_i)}$
for all $E' \in M_{H_f}({\bf e})$.
If $E_{|\pi^{-1}(p_i)}$ is not stable, then we have a quotient
$E \to G$ such that $G$ is a stable vector bundle of rank $r_1$ and $\chi(G)=d_1$
with $rd_1-r_1 (\xi \cdot f)<0$.
Then we get $\Hom(G,F_i) \ne 0$. 
\end{NB}
Since $\chi(G,F_i)=\rk G (\xi \cdot f)-r(c_1(G) \cdot f)>0$,
$\Hom(F_i,{\cal E}_{|\{y \} \times X})=0$ for all $y \in M_{(f,H)}^G(v)$.
Hence 
${\bf R}\Hom_{p_{M_{(f,H)}^G(v)}}(p_X^*(F_i),{\cal E}[2])$ is a two-term complex of locally free sheaves.  
We set
\begin{equation}
D:=\{ y \in M_{(f,H)}^G(v) \mid \Ext^2(\oplus_i F_i, {\cal E}_{|\{y \} \times X}) \ne 0 \}.
\end{equation}
Then 
${\cal O}_{M_{(f,H)}^G(v)}(D) \cong \theta_v(\sum_{i=1}^N F_i)$.
Since $E \not \in D$, $\theta_v(\sum_{i=1}^N F_i)$ is base point free. 
Therefore we get (1).

\begin{NB}
(2)
We note that $mrK_X=\pi^{-1}(\delta)$ $(\delta \in \Pic(C))$ if $m_i \mid m$
for all $i$.
Then $K_{M_{H_f}({\bf e})}^{\otimes m}={\cal L}_{\bf e}(\delta)$ up to torsion.
Hence the second claim holds.
\end{NB}
\end{proof}

\begin{rem}
Assume that there is a smooth moduli space $X':=M_{H_f}(0,r_0 f,d_0)$ such that
$r_0 (\xi \cdot f)-r d_0=0$.
\begin{NB}
For example, all fibers are irreducible.
\end{NB}
Let ${\cal P}$ be a (twisted) universal family on $X \times X'$.
By $\Phi_{X \to X'}^{{\cal P}^{\vee}}$,
$M_{(f,H)}^G(v)$
is isomorphic to a moduli space of stable twisted sheaves of dimensioon 1 on $X'$.
Then we have a morphism $M_{(f,H)}^G(v) \to \Hilb_{X'}^{\xi'}$,
where $\Hilb_{X'}^{\xi'}$ is the Hilbert scheme of divisors $D$ numerically equivalent to $\xi'$.
$\theta_v(u) =(\Phi_{X \to X'}^{{\cal P}^{\vee}[2]})^* (\theta_{v'}(\varrho_{X'}))$
 is the pull-back of a line bundle on $\Hilb_{X'}^{\xi'}$.
\end{rem}

Let $\psi:M_{(f,H)}^G(v) \to {\Bbb P}^N$ be the morphism defined by $\theta_v(m \alpha)$ $m \gg 0$.
The following proposition shows the relation between our chamber structure
and $\psi$-ample line bundles.

\begin{prop}\label{prop:rel-ample}
Let ${\cal C}$ be a chamber and 
$\beta:=s_0 H+\alpha \in {\cal C}$, where $s_0  \in {\Bbb Q}$ and
$\alpha \in N_H$. Then
$\theta_v (F)$ is $\psi$-ample, where
$F \in K(X)_v$ with
$v(F)=-n(1+\beta+b \varrho_X)$.
\end{prop}

\begin{proof}
We take $\epsilon>0$ such that 
$\beta \pm \epsilon H \in {\cal C}$. 
We shall prove the claim by modifying the proof of Corollary \ref{cor:Gieseker}.
We set $s_0=\frac{d_0}{r_0 (f \cdot H)}$ ($r_0 \in {\Bbb Z}_{>0}$, $d_0 \in {\Bbb Z}$).
We take an integer $k \in \frac{1}{(f \cdot H)}{\Bbb Z}$ such that
$$
s_0-\epsilon <\frac{d_0 k-1}{r_0 k (f \cdot H)}
<\frac{d_0 }{r_0 (f \cdot H)}<\frac{d_0 k+1}{r_0 k (f \cdot H)}<
s_0+\epsilon.
$$
If the claim holds for $\lambda=\frac{d_0 k-1}{r_0 k (f \cdot H)}, \frac{d_0 k+1}{r_0 k (f \cdot H)}$,
then the claim also holds for $\lambda=\frac{d_0}{r_0 (f \cdot H)}$.
\begin{NB}
$s(1,\eta_1,b_1)+t(1,\eta_2,b_2) \sim
(1,\frac{s \eta_1+t \eta_2}{s+t},\frac{s b_1+t b_2}{s+t})$.
\end{NB}
Hence we shall prove the claim under the assumption $\gcd(r_0 (f \cdot H),d_0)=1$.
We may also replace $\alpha$ by any general $\alpha' \in B$ with respect to
$r_0 f+d_0 \varrho_X$, where $B$ is a small neighborhood of $\alpha$.
\begin{NB}
We take $\alpha+\pm \eta$ from opposite chambers.
If the claim holds for $\alpha+\pm \eta$, then the claim holds for $\alpha$. 
\end{NB}
We set $X'=M_{H_f}^{\alpha'}(r_0 f+d_0 \varrho_X)$.
Then 
$X'$ is a fine moduli space and we have a Fourier-Mukai transform
$\Phi_{X \to X'}^{{\cal P}^{\vee}[1]}:{\bf D}(X) \to {\bf D}(X')$ which induces an isomorphism
${\cal M}_{(f,H)}^{s_0 H+\alpha'}(r+\xi+a\varrho_X) \cong 
{\cal M}_{H'_{f'}}(r'+\xi'+a' \varrho_{X'})^{ss}$.

We set $F':=\Phi_{X \to X'}^{{\cal P}^{\vee}[1]}(F)$.
By \cite[Lem. 3.2.1]{PerverseII}, 
$\rk (F')=0$ and
$c_1(F')$
is $\pi$-ample. 
By Corollary \ref{cor:Gieseker} and Lemma \ref{lem:ample}, 
$\theta_{v'}(F')$ is $\psi$-ample, where
$v'=r'+\xi'+a' \varrho_{X'}$. Hence the claim holds.
\end{proof}

\begin{NB}
For the moduli space $M_{H_f}(v)$, 
$-(1,\eta,b)+k(0,H,*)$ is ample if $k \gg 0$.
If all fibers of $\pi$ are irreducible, then
$\theta_{\bf e}(F)$ is ample if $(\eta \cdot f) \ll 0$.
\end{NB}

\begin{lem}[{\cite[Lem. 5.8]{Y:elliptic}}]\label{lem:ample}
Assume that $M_{H_f}(v)$ consists of $\mu$-stable vector bundles.
Then $\theta_v (F)$ is ample, where $F \in K(X)_v$ satisfies
$\rk F=0$ and
$c_1(F) \in {\Bbb Z}_{>0} H_f$. 
\end{lem}

\begin{proof}
For the proof of \cite[Lem. 5.8]{Y:elliptic}, we do not use 
the assumption $\gcd(r,(\xi \cdot f))=1$.
So the claim holds.
\end{proof}

\section{Examples}\label{sect:example}

\subsection{Group action on the parameter space of stability.}

Let $\pi:X \to C$ be an elliptic surface with a section $\sigma$.
Assume that $e:=\chi({\cal O}_X)$ is even.
Then $(\sigma^2)=-e$ is even and $H:=\sigma-\frac{(\sigma^2)}{2}f \in \NS(X)$
satisfies $(H^2)=0$.
$\NS(X)$ is an even lattice with an orthogonal decomposition
$$
\NS(X)=({\Bbb Z}H+{\Bbb Z}f)+({\Bbb Z}H+{\Bbb Z}f)^\perp.
$$
In this case 
$$
v(E) \in H^*(X,{\Bbb Z}),\; E \in {\bf D}(X).
$$
We set $X_1:=M_H(0,f,0)$. Then $X_1 \cong X$ is a compactification of the relative Jacobian. 
Let ${\cal P}$ be the universal family on $X \times X$.
For the equivalence $\Phi_{X \to X}^{{\cal P}^{\vee}[1]}$, we may assume that
$\beta=\beta'=0$ and $H'=H$.
Then we see that 
\begin{equation}
\overline{\Phi}_{X \to X}^{{\cal P}^{\vee}[1]}(r+dH+kf+D+a \varrho_X)=
d-rH+af-D-k \varrho_X 
\end{equation}
(cf. \cite[(6.21)]{BBH}).
\begin{rem}
$\Phi$ in \cite[6.2.3]{BBH} corresponds to $\otimes {\cal L} \circ \Phi_{X \to X}^{{\cal P}^{\vee}[1]}$,
where $c_1({\cal L})=-\frac{e}{2}f$.
\end{rem}
Hence
$\overline{\Phi}_{X \to X}^{{\cal P}^{\vee}[1]}(e^{sH})=se^{-\frac{1}{s}H}$.
For ${\cal L} \in \pi^*(\Pic(C))$, we have an autoequivalence $\otimes {\cal L} \in {\bf Auteq}({\bf D}(X))$.
Thus $\Pic(C)$ is regarded as a subgroup of ${\bf Auteq}({\bf D}(X))$. 
We set
$$
G:=\langle \Phi_{X \to X}^{{\cal P}^{\vee}[1]},\;\otimes {\cal O}_X(H),\;[1],\;
\Pic(C) \rangle \subset {\bf Auteq}({\bf D}(X)).
$$ 
Then we have a homomorphism
\begin{equation}
\varphi:G \to \SL(2,{\Bbb Z}) \times \SL(2,{\Bbb Z})
\end{equation}
such that
\begin{equation}
\begin{matrix}
\varphi(\Phi_{X \to X}^{{\cal P}^{\vee}[1]})=
\begin{pmatrix}
0 & 1 \\
-1 & 0
\end{pmatrix} \times 
\begin{pmatrix}
1 & 0 \\
0 & 1
\end{pmatrix},\;
&
\varphi(\otimes {\cal O}_X(H))=
\begin{pmatrix}
1 & 0 \\
1 & 1
\end{pmatrix} \times 
\begin{pmatrix}
1 & 0 \\
0 & 1
\end{pmatrix},\;\\
\varphi([1])=
\begin{pmatrix}
-1 & 0 \\
0 & -1
\end{pmatrix} \times 
\begin{pmatrix}
1 & 0 \\
0 & 1
\end{pmatrix},\;&
\varphi(\otimes {\cal L})=
\begin{pmatrix}
1 & 0 \\
0 & 1
\end{pmatrix} \times 
\begin{pmatrix}
1 & k \\
0 & 1
\end{pmatrix}
\end{matrix}
\end{equation}
where $c_1({\cal L})=kf$ and the group action $\cdot$ on
$\SL(2,{\Bbb Z}) \times \SL(2,{\Bbb Z})$ is $(A_1,B_1) \cdot (A_2,B_2):=(A_1 A_2,B_2 B_1)$.
\begin{NB}
$c_1(K_X)=(2g(C)-2+e)f$.
\end{NB}
We set
$$
\Phi_{X \to X}^{{\cal P}_1^{\vee}[2]}
:=(\Phi_{X \to X}^{{\cal P}^{\vee}[1]})^{-1} \circ \otimes {\cal O}_X(H) \circ 
\Phi_{X \to X}^{{\cal P}^{\vee}[1]}.
$$
Then 
${\cal P}_1 \in \Coh(X \times X)$ such that
${\cal P}_{1| \{ x \} \times X}^{\vee}[1] \in M_{H_f}(f-\varrho_X)$ and
${\cal P}_{1| X \times \{ x \}} \in M_{H_f}(f+\varrho_X)$.
\begin{NB}
$\Phi_{X \to X}^{{\cal P}_1^{\vee}[2]}({\Bbb C}_x)=(\Phi_{X \to X}^{{\cal P}^{\vee}[1]})^{-1}
(L)$,
where $v(L)=f+\varrho_X$.
Since $(\Phi_{X \to X}^{{\cal P}^{\vee}[1]})^{-1} (L)=\Phi_{X \to X}^{{\cal P}[1]}(L)(-K_X)$
is $L'[1]$ with $v(L')=f-\varrho_X$,
${\cal P}_{1|\{x \} \times X}^{\vee}[2]=L'[1]$, which implies
${\cal P}_1^{\vee}[1] \in \Coh(X \times X)$.
\end{NB}
We see that
\begin{equation}
\overline{\Phi}_{X \to X}^{{\cal P}_1^{\vee}[2]}(r+dH+kf+D+a \varrho_X)=r-d+dH+(k-a)f+D+a \varrho_X
\end{equation}
and hence
$$
\varphi(\Phi_{X \to X}^{{\cal P}_1^{\vee}[2]})=
\begin{pmatrix}
1 & -1 \\
0 & 1
\end{pmatrix} \times 
\begin{pmatrix}
1 & 0 \\
0 & 1
\end{pmatrix}.
$$
We set
$$
H^*(X,{\Bbb Z})_H:={\Bbb Z}+{\Bbb Z}H+{\Bbb Z}f+{\Bbb Z}\varrho_X.
$$
Then we have an orthogonal decomposition
$$
H^*(X,{\Bbb Z})= H^*(X,{\Bbb Z})_H \oplus H^*(X,{\Bbb Z})_H^\perp.
$$
Let $M_2({\Bbb Z})$ be the set of 2 by 2 matrices of integer coefficient.
Then we have an identification
\begin{equation}
\begin{matrix}
\nu:& H^*(X,{\Bbb Z})_H & \to & M_2({\Bbb Z})\\
& r+dH+kf+a \varrho_X & 
\mapsto &
\begin{pmatrix}
r & k\\
d & a
\end{pmatrix}.
\end{matrix}
\end{equation}
We define the action of $(A,B) \in \SL(2,{\Bbb Z}) \times \SL(2,{\Bbb Z})$ on $X \in M_2({\Bbb Z})$ by
$(A,B) \cdot X=AXB$.
Then
\begin{equation}
\varphi(\Phi) \cdot \nu(v)=\nu(\overline{\Phi}(v)),\; \Phi \in G.
\end{equation}
\begin{NB}
$\Phi_{X \to X}^{{\cal P}^{\vee}[1]}$ corresponds to 
$\begin{pmatrix}
0 & 1 \\
-1 & 0
\end{pmatrix}
\in SL(2,{\Bbb Z})$ and
$\begin{pmatrix}
0 & 1 \\
-1 & 0
\end{pmatrix}
\begin{pmatrix}
r & k\\
d & a
\end{pmatrix}
=
\begin{pmatrix}
d & a\\
-r & -k
\end{pmatrix}
$.

$\otimes {\cal O}_X(H)$ corresponds to
$\begin{pmatrix}
1 & 0 \\
1 & 1
\end{pmatrix} \in \SL(2,{\Bbb Z})$ and
$\begin{pmatrix}
1 & 0 \\
1 & 1
\end{pmatrix}
\begin{pmatrix}
r & k\\
d & a
\end{pmatrix}
=
\begin{pmatrix}
r & k \\
r+d & k+a
\end{pmatrix}
$.
For an equivalence whose matrix is
$
\begin{pmatrix}
-d_1 & r_1\\
p_2 & q_2
\end{pmatrix}
$,
$(0,r_1 f,d_1) \mapsto -\varrho_X$.

\begin{rem}
Assume that $(r_1 a-d_1 k,l)=1$.
Then $(r_1+nl p)a-(d_1+nl q)k,l)=1$.
Moreover $\frac{d_1+nl q}{r_1+nl p} \to \frac{q}{p}$ ($n \to \infty$). 
By a suitable FM, $(f,G,H)$-semistability is equivalent to $\mu$-stability, if
$\gcd(r,d,k,a)=1$.
\end{rem}

For a line bundle $L$ with $c_1(L)=f$,
$\otimes L$ corresponds to right multiplication 
$\begin{pmatrix}
r & k\\
d & a
\end{pmatrix}
\mapsto
\begin{pmatrix}
r & k\\
d & a
\end{pmatrix}
\begin{pmatrix}
1 & 1 \\
0 & 1
\end{pmatrix}
$.



\end{NB}

For any Mukai vector $v$, there is a relative Fourier-Mukai transform $\Phi$ such that
$$
\overline{\Phi}(v)=r+kf+D-a \varrho_X,\; D \in N_H.
$$
So we assume that $v=r+kf+D-a \varrho_X$.
For an open interval $I (\subset {\Bbb R})$,
we set
$$
{\cal G}(v)_I:=\{(s,\alpha) \in {\cal G}(v) \mid s \in I \}.
$$
We define a family of intervals $\{I_n \}_{n \in {\Bbb Z}}$
by
\begin{equation}
\begin{split}
I_n=& (-\tfrac{1}{rn},-\tfrac{1}{r(n+1)}),\quad n \ne -1,0,\\ 
I_0=& (-\infty,-\tfrac{1}{r}),\\
I_{-1}=& (\tfrac{1}{r},\infty).
\end{split}
\end{equation} 
We set
\begin{equation}\label{eq:auto}
\Phi_{X \to X}^{{\cal R}^{\vee}[2]}:=(\Phi_{X \to X}^{{\cal P}_1^{\vee}[2]})^r \circ {\cal L}_a \in G
\end{equation}
where $c_1({\cal L}_a)=-af$.
Then we see that 
${\cal R} \in \Coh(X \times X)$ and
$$
\varphi(\Phi_{X \to X}^{{\cal R}^{\vee}[2]})=
\begin{pmatrix}
1 & -r \\
0 & 1
\end{pmatrix} \times 
\begin{pmatrix}
1 & -a \\
0 & 1
\end{pmatrix}.
$$ 
Hence
$\overline{\Phi}_{X \to X}^{{\cal R}^{\vee}[2]}(r+kf-a \varrho_X)=r+kf-a \varrho_X$ and
$\overline\Phi_{X \to X}^{{\cal R}^{\vee}[2]}(e^{sH})=(1-rs)e^{\frac{s}{1-rs}H}$.
\begin{NB}
$v({\cal R}_{|\{ x \} \times X}^{\vee}[2])=(0,-rf,1)$ and
$v({\cal R}_{|X \times \{ x \}})=(0,rf,1)$.
\end{NB}
Thus we have an action of 
$\Phi_{X \to X}^{{\cal R}^{\vee}[2]}$ on ${\cal G}(v)$ by
$(s,\alpha) \mapsto (\tfrac{s}{1-rs},\alpha)$. 
\begin{NB}
$\Phi_{X \to X}^{{\cal R}^{\vee}[2]}(-\infty)=-\frac{1}{r},\;\Phi_{X \to X}^{{\cal R}^{\vee}[2]}(-\frac{1}{r})=-\frac{1}{2r}, 
\Phi_{X \to X}^{{\cal R}^{\vee}[2]}(-\frac{1}{2r})=-\frac{1}{3r}$,...
\end{NB}
In particular
\begin{equation}
\Phi_{X \to X}^{{\cal R}^{\vee}[2]}({\cal G}(v)_{I_n})={\cal G}(v)_{I_{n+1}}.
\end{equation}
In order to study walls for $v$, it is sufficient to study walls in ${\cal G}(v)_{I_0}$.

\begin{NB}
The stabilizer for the case of rank 0:
For $v=rH+kf+D+a \varrho_X$,
an autoequivalence $\otimes {\cal O}_X(-rH+kf)$ corresponding to 
$
\begin{pmatrix}
1 & 0 \\
-r & 1
\end{pmatrix} \times 
\begin{pmatrix}
1 & k \\
0 & 1
\end{pmatrix}
$
fixes $v$. Under this action, $s$ is replaced by $s-r$.
\end{NB}

\subsection{Examples for an elliptic K3 surface. }

Let $\pi:X \to {\Bbb P}^1$ be an elliptic K3 surface such that
$\NS(X)={\Bbb Z}\sigma+{\Bbb Z}f$, where 
$\sigma$ is a section and $f$ a fiber of $\pi$.
We set $H:=\sigma+f$.

\begin{lem}
Let $W$ be a wall defined by $v_1$.
We set $L:=({\Bbb Q}v+{\Bbb Q}v_1) \cap H^*(X,{\Bbb Z})$.
Then $L$ contains an isotropic Mukai vector.
\end{lem}

\begin{proof}
$u:=(\rk v_1) v-(\rk v) v_1 \in L$ is an isotropic Mukai vector.
\end{proof}

We assume that $v=r+kf+a \varrho_X$.
By using Proposition \ref{prop:sswall},
we shall classify totally semistable walls. 

\begin{lem}
We set $u:=x+yf+z \varrho_X \in L$ $(x,y,z \in {\Bbb Z})$. 
\begin{enumerate}
\item[(1)]
Assume that $\langle u^2 \rangle=-2$.
Then 
$u=1+yf+\varrho_X, -1+yf- \varrho_X$.
\item[(2)]
Assume that $u$ is isotropic.
Then $x=0$ or $z=0$.
\end{enumerate}
\end{lem}

If a $(-2)$-vector $u=1+yf+\varrho_X$ defines a totally semistable wall, then
$\langle u,v \rangle<0$. Hence $r+a \geq 1$.

If an isotropic vector $u=x+yf+z \varrho_X$ defines a totally semistable wall, then 
$\langle u,v \rangle=1$, and hence
$-rz=1$ or $-xa=1$.
In particular $r=1$ or $a=-1$.

\begin{prop}\label{prop:bir-FM}
For a Mukai vector $v=r+\xi+a \varrho_X$, we set
$p:=\gcd(r,(\xi \cdot f))$.
If $p \geq 2$ and $\langle v^2 \rangle \geq 2p^2$, then
there is no totally semistable walls.
In particular
$\Phi_{X \to X}^{{\cal P}^{\vee}}(E)$ is a stable sheaf up to shift for a general
$E \in M_{(f,H)}^G (v)$. 
\end{prop}

\begin{proof}
We may assume that $v=p+kf+b\varrho_X$.
Then our assumption implies $p \geq 2$ and $b \leq -p$.
Hence there is no totally semistable wall.
\end{proof}

\begin{cor}
$\Phi_{X \to X}^{{\cal R}^{\vee}[2]}$ induces a infinite order birational automorphism
of $M_{H_f}(r+kf-a\varrho_X)$ if $r,a \geq 3$ and $a>r$.
\end{cor}

\begin{NB}
\begin{rem}
Assume that $a>r \geq 3$. Then 
for a general $E \in M_{H_f}(r+kf-a \varrho_X)$,
$E_{|\pi^{-1}(c)}$ is a semistable sheaf for any $c \in C$.
\end{rem}
\end{NB}

\begin{NB}
For a description of the wall crossing behavior under totally semistable walls,
we introduce a spherical twist appearing in our setting.
Let $E_0$ be a rigid stable sheaf with $v(E_0) \in L$.
We set 
\begin{equation}
{\cal E}:=\ker(p_1^*(E_0^{\vee}) \otimes p_2^*(E_0) \to {\cal O}_\Delta).
\end{equation}
Then the spherical twist is an auotequivalence defined by
$$
R_{E_0}(E):={\bf R}p_{2*}(p_1^*(E) \otimes {\cal E})[1],\; E \in {\bf D}(X).
$$
If $E \in \Coh(X)$ satisfies $\Ext^2(E_0,E)=0$,
then we have an exact sequence 
$$
0 \to \Hom(E_0,E) \otimes E_0 \to E \to R_{E_0}(E) \to \Ext^1(E_0,E) \otimes E_0 \to 0.
$$
\end{NB}
   
\subsubsection{Walls for $v=2+f-\varrho_X$.}
We set $v=2+f-\varrho_X$.
Then $\langle v^2 \rangle=4$.
For a wall $W$, we have a decomposition
$$
v=v_1+v_2, v_1, v_2 \in U(v).
$$
Then one of the following holds:
\begin{enumerate}
\item
$v_1=1+k_1 f-2\varrho_X, v_2=1+(1-k_1)f+\varrho_X$. In this case $s=-\frac{3}{2k_1-1}$.
\item
$v_1=1+k_1 f-\varrho_X, v_2=1+(1-k_1)f$. In this case
$s=-\frac{1}{2k_1-1}$.
\item
$v_1=k_1 f,-\varrho_X, v_2=2+(1-k_1)f$. In this case
$s=-\frac{1}{k_1}$.
\item
$v_1=(2m+1) f-3\varrho_X, v_2=2(1-m f+\varrho_X)$. In this case $s=-\frac{3}{2m+1}$.
\end{enumerate}

We note that $s=-\frac{1}{2}$ is a wall.
Indeed we have a decomposition $v=(2f-\varrho_X)+(2-f)$.
We shall study walls and chambers in $s \leq -\frac{1}{2}$.
There are threee walls $s=-3, -1,-\frac{3}{5}$ and hence 4 chambers
\begin{equation}
J_1:=(-\infty, -3),\;
J_2:=(-3,-1),\;
J_3:=(-1,-\tfrac{3}{5}),\;
J_4:=(-\tfrac{3}{5},-\tfrac{1}{2}).
\end{equation}
All walls are totally semistable.

For the analysis of wall- crossing, we also need to 
walls for $v=1+lf-2 \varrho_X$.
It is easy to see that walls are defined by 
$v_1=k_1 f+a_1\varrho_X$, where $a_1=-3,-2,-1$ and $k_1 \geq 1$.
$v_1$ defines a wall $W_{\tfrac{a_1}{k_1}}$: $s=\tfrac{a_1}{k_1}$.

For $s<-\tfrac{1}{2}$, we have the following walls: 

\begin{center}
\begin{tabular}{ccc}
(i) $s=-3$ with $v_1=f-3\varrho_X$. & (ii) $s=-2$ with $v_1=f-2\varrho_X$.  
& (iii) $s=-\frac{3}{2}$ with $v_1=2f-3\varrho_X$. \\
(iv) $s=-1$ with $v_1=f-\varrho_X$. &
(v) $s=-\frac{3}{4}$ with $v_1=4f-3\varrho_X$. &  (vi) $s=-\frac{2}{3}$ with $v_1=3f-2\varrho_X$.\\
 (vii) $s=-\frac{3}{5}$ with $v_1=5f-3\varrho_X$ &  &  \\ 
\end{tabular}
\end{center}
and
$s=-1$ is the unique totally semistable wall in $(-\infty,-\tfrac{1}{2})$.
Hence we have the following chambers for $v=1+lf-2\varrho_X$.
\begin{equation}
J_1':=(-\infty,-3), J_2':=(-3,-2), J_3':=(-2,-\tfrac{3}{2}), J_4':=(-\tfrac{3}{2},-1),
J_5':=(-1,-\tfrac{3}{4}),...
\end{equation}

For simplicity, we set
$$
M_J(v)=M_{(f,H)}^{{\cal O}_X(sH)}(v),\; s \in J 
$$
where $J$ is a chamber for $v$.

(1) 
The wall $W_{-3}$: $s=-3$.
We note that
$M_{J_1}(v)=M_{H_f}(v)$.
$W_{-3}$ is a totally semistable wall defined by $v({\cal O}_X)$. 
By \cite{BM:2}, $R_{{\cal O}_X}$ induces isomorphisms
\begin{equation}\label{eq:reflection1}
\begin{split}
R_{{\cal O}_X}:& M_{J_1}(2+f-\varrho_X) \stackrel{\sim}{\longrightarrow} M_{J_2'}(1+f-2\varrho_X),\\
R_{{\cal O}_X}:& M_{J_1'}(1+f-2\varrho_X) \stackrel{\sim}{\longrightarrow} M_{J_2}(2+f-\varrho_X).
\end{split}
\end{equation}
We note that
$M_{J_1}(1+f-2\varrho_X)=M_{H_f}(1+f-2\varrho_X)$.
In order to describe the wall crossing for $v':=1+f-2\varrho_X$,
we set 
$$
M_{H_f}(1+f-2\varrho_X)_i:=\{ I_Z(f) \in M_{H_f}(1+f-2 \varrho_X) \mid h^0(I_Z(f))=i \}.
$$
Then 
\begin{enumerate}
\item
$M_{H_f}(1+f-2\varrho_X)_0=M_{J_1}(1+f-2 \varrho_X) \cap M_{J_2'}(1+f-2 \varrho_X)$,
\item
$M_{H_f}(1+f-2\varrho_X)_1$ is a ${\Bbb P}^2$-bundle over $M_{H_f}(f-3 \varrho_X)$ and
\item
$M_{H_f}(1+f-2\varrho_X)_i=\emptyset$ for $i>1$.
\end{enumerate}
Hence we have a birational map
$$
M_{J_2'}(1+f-2 \varrho_X) \cdots \to  M_{J_1}(1+f-2 \varrho_X)
$$
which is a Mukai flop along $M_H(1+f-2\varrho_X)_1$. 
Combining \eqref{eq:reflection1}, we see that 
$R_{{\cal O}_X} \circ R_{{\cal O}_X}$ induces a birational map
$$
M_{J_1}(2+f-\varrho_X) \cdots \to M_{J_2}(2+f-\varrho_X),
$$
which is a Mukai flop.

By \eqref{eq:reflection1},
a general $E \in M_{J_1}(2+f-\varrho_X)$ fits in an exact sequence
$$
0 \to {\cal O}_X \to E \to I_Z(f) \to 0
$$
and
a general $E' \in M_{J_2}(2+f-\varrho_X)$ fits in
an exact sequence
$$
0 \to I_Z(f) \to E' \to {\cal O}_X \to 0
$$
where $I_Z(f)=R_{{\cal O}_X}(E) \in M_{H_f}(1+f-2\varrho_X)_0$.

\begin{NB}
For $E \in M_{J_1}(2+f-\varrho_X)$,
$\Hom(E,{\cal O}_X)=0$.
Hence we have an exact sequence
$$
0 \to H^0(X,E) \otimes {\cal O}_X \to E \overset{\phi}{\to} R_{{\cal O}_X}(E) \to
H^1(X,E) \otimes {\cal O}_X \to 0
$$
and $\phi(E)=I_Z(f) \in M_{H_f}(1+f-2\varrho_X)$ if $h^0(E)=1$ and
$\phi(E)=L \in M_{H_f}(f-3\varrho_X)$ if $h^0(E)=2$.

For $I_Z(f) \in M_{H_f}(1+f-2\varrho_X)$,
$\Hom(I_Z(f),{\cal O}_X)=0$.
Hence we have an exact sequence
$$
0 \to H^0(X,I_Z(f)) \otimes {\cal O}_X \to I_Z(f) \overset{\phi}{\to} R_{{\cal O}_X}(I_Z(f)) \to
H^1(X,I_Z(f)) \otimes {\cal O}_X \to 0.
$$
$\phi(I_Z(f))=I_Z(f)$ if $h^0(I_Z(f))=0$ and $\phi(I_Z(f))=
L \in M_{H_f}(0,f-3\varrho_X)$ if $h^0(I_Z)=1$.
\end{NB}

\begin{NB}
If $h^0(I_Z(f))=0$, then
\begin{equation}
0 \to I_Z(f) \to \Phi_{X \to X}^{I_\Delta [1]}(I_Z(f)) \to H^1(X,I_Z(f)) \otimes {\cal O}_X \to 0.
\end{equation}

If $h^0(I_Z(f))=1$, then
\begin{equation}
0 \to L \to \Phi_{X \to X}^{I_\Delta [1]}(I_Z(f)) \to H^1(X,I_Z(f)) \otimes {\cal O}_X \to 0.
\end{equation}
\end{NB}

(2) The wall $W_{-1}$:$s=-1$.  
$W_{-1}$ is a totally semistable wall defined by $v({\cal O}_X(-f))$.
Hence we have isomorphisms
\begin{equation}\label{eq:reflection2}
\begin{split}
R_{{\cal O}_X(-f)}:& M_{J_2}(2+f-\varrho_X) \stackrel{\sim}{\longrightarrow} M_{J_5'}(1+2f-2\varrho_X),\\
R_{{\cal O}_X(-f)}:& M_{J_4'}(1+2f-2\varrho_X) \stackrel{\sim}{\longrightarrow} M_{J_3}(2+f-\varrho_X).
\end{split}
\end{equation}

We note that $R_{{\cal O}_X}({\cal O}_X(f))={\cal O}_X(-f)[1]$
and $E:=R_{{\cal O}_X}(I_Z(f)) \in M_{J_2}(2+f-\varrho_X)$ for $I_Z \in M_{J_3}(1+f-2\varrho_X)$
by \eqref{eq:reflection1}.
Hence
$$
\Hom({\cal O}_X(-f),E)=\Hom({\cal O}_X(f),I_Z(f)[1])=H^1(X,I_Z)\cong {\Bbb C}^2
$$
and we have an exact sequence
$$
0 \to {\cal O}_X(-f)^{\oplus 2} \to E \to F \to {\cal O}_X(-f) \to 0,
$$
where $F \in M_{J'_5}(1+2f-2\varrho_X)$.
A general $F$ fits in an exact sequence
$$
0 \to \oplus_{i=1}^3 L_i \to F \to {\cal O}_X(-f) \to 0,
$$
where 
$L_i \in X':=M_{H_f}(f-\varrho_X)$.

Since $s=-1$ is a totally semistable wall for $v=1+2f-2\varrho_X$,
$M_{J_4'}(1+2f-2\varrho_X)$ and $M_{J_5'}(1+2f-2\varrho_X)$ have Hilbert-Chow type contractions, and 
we have an isomorphism
\begin{equation}\label{eq:reflection:r=1}
M_{J_4'}(1+2f-2\varrho_X)  \stackrel{\sim}{\longrightarrow} M_{J_5'}(1+2f-2\varrho_X) 
\end{equation} 
by a suitable contaraviant Fourier-Mukai transform.
Combining \eqref{eq:reflection2} with \eqref{eq:reflection:r=1},
we have an isomorphism
\begin{equation}
M_{J_2}(2+f-\varrho_X) \stackrel{\sim}{\longrightarrow} M_{J_3}(2+f-\varrho_X).
\end{equation}
By \eqref{eq:reflection2},
we also have an exact sequence
$$
0 \to \Hom({\cal O}_X(-f),I_Z(2f)) \otimes {\cal O}_X(-f) \to I_Z(2f) \to
E' \to \Ext^1({\cal O}_X(-f),I_Z(2f)) \otimes 
{\cal O}_X(-f) \to 0
$$
where $I_Z(2f) \in M_{J_4'}(1+2f-2\varrho_X) \cap M_{J_1}(1+2f-2\varrho_X)$ and
$E'=R_{{\cal O}_X(-f)}(I_Z(2f)) \in M_{J_3}(2+f-\varrho_X)$.

\begin{NB}
$R_{{\cal O}_X(mf)}({\cal O}_X((m+1)f))={\cal O}_X((m-1)f)[1]$.
$$
\Hom({\cal O}_X(-2f),R_{{\cal O}_X(-f)}(I_Z(2f))) \cong
\Hom({\cal O}_X[-1],I_Z(2f))=H^1(I_Z(2f)) \cong {\Bbb C}.
$$
\end{NB}

\begin{NB}
$\Hom({\cal O}_X(-f),I_Z(2f)) \ne 0$.
For a general $Z=\{ x_1,x_2,x_3 \}$,
$f_1+f_2+f_3$ ($x_i \in f_i$) is the corresponding section.

For $E$ with $H^0(E)=H^2(E)=0$,
$E_1:=\Phi_{X \to X}^{I_\Delta^{\vee}[1]}(E)$ fits in an exact sequence
$$
0 \to H^1(E) \otimes {\cal O}_X \to E_1
\to E \to 0.
$$
Applying $\Phi_{X \to X}^{I_\Delta^{\vee}}$, we get 
$$
0 \to H^1 (E) \otimes {\cal O}_X \to
\Phi_{X \to X}^{I_\Delta^{\vee}}(E_1) \to E_1 [-1] \to 0
$$

\end{NB}

\begin{NB}
Walls for $v=(1,2f,-2)$.

$v_1=(0,k_1 f,a_1)$, $a_1=-3,-2,-1$ and $k_1 \geq 1$.

$s=-3$. $v=(0,f,-3)+(1,f,1)$.

$s=-2$. $v=(0,f,-2)+(1,f,0)$.

$s=-\frac{3}{2}$. $v=(0,2f,-3)+(1,0,1)$.

$s=-1$. $v=3(0,f,-1)+(1,-f,1)$. totally semistable wall.

$s=-\frac{3}{4}$. $v=(0,4f,-3)+(1,-2f,1)$.

$s=-\frac{2}{3}$. $v=(0,3f,-2)+(1,-f,0)$.

$s=-\frac{3}{5}$. $v=(0,5f,-3)+(1,-3f,1)$.
Hence we have chambers for $v=1+2f-2\varrho_X$.
\begin{equation}
J_1':=(-\infty,-3), J_2':=(-3,-2), J_3':=(-2,-\tfrac{3}{2}), J_4':=(-\tfrac{3}{2},-1),
J_5':=(-1,-\tfrac{3}{4}),...
\end{equation}

\end{NB}

\begin{NB}
$s=-\infty$ corresponds to the Donaldson-Uhlenbeck contraction.
\end{NB}


\begin{rem}
For $s>-1$, 
$\Hom(E,{\cal P}_{1|X \times \{ x \}}^{\vee}[1])=0$ for all $x \in X$.
Since $\Phi_{X \to X}^{{\cal P}_1^{\vee}[1]}(E)$ is not a line bundle,
$\Hom({\cal P}_{1|X \times \{ x \}}^{\vee}[1],E) \ne 0$ for a point $x \in X$.
Thus $E$ is not torsion free.
\end{rem}

\begin{rem}
Let $\Phi_{X \to X}^{{\cal P}_2^{\vee}[1]}$ be a relative Fourier-Mukai
transform such that
${\cal P}_2 \in \Coh(X \times X)$ and
$$
\varphi(\Phi_{X \to X}^{{\cal P}_2^{\vee}[1]})=
\begin{pmatrix}
2 & 1\\
-1 & 0
\end{pmatrix}
\times 
\begin{pmatrix}
1 & 0\\
0 & 1
\end{pmatrix}.
$$
Then we have
${\cal P}_{2| X \times \{ x \}} \in M_{H_f}(f-2 \varrho_X)$
and ${\cal P}_{2| \{x \} \times X} \in M_{H_f}(f)$.
$\overline{\Phi}_{X \to X}^{{\cal P}_2^{\vee}[1]}(2+f-\varrho_X)=4-2H+f-\varrho_X$ and
we have an isomorphism
$$
M_{J_2}(2+f-\varrho_X) \stackrel{\sim}{\longrightarrow} M_{H_f}(4-2H+f-\varrho_X).
$$
In order to construct a birational map
$M_{H_f}(2+f-\varrho_X) \to M_{H_f}(4-2H+f-\varrho_X)$, we need to use 
Fourier-Mukai transform
$\Phi_{X \to X}^{{\cal P}_2^{\vee}[1]} \circ R_{{\cal O}_X} \circ R_{{\cal O}_X}$ whose kernel is
not a sheaf:
We set
$$
\Phi_{X \to X}^{{\cal P}'[1]}:=(\Phi_{X \to X}^{{\cal P}_2^{\vee}[1]} \circ R_{{\cal O}_X} \circ R_{{\cal O}_X})^{-1}
=\Phi_{X \to X}^{I_\Delta^{\vee}[1]}\circ \Phi_{X \to X}^{I_\Delta^{\vee}[1]} 
\circ \Phi_{X \to X}^{{\cal P}_2[1]}.
$$
Then ${\cal P}'$ is a two-term complex with the triangle 
$$
{\cal O}_X^{\oplus 2}[1] \to {\cal P}'_{|X \times \{ x \}} \to 
\Phi_{X \to X}^{I_\Delta^{\vee}[1]}({\cal P}_{2|X \times \{ x \}})
\to  {\cal O}_X^{\oplus 2}[2]. 
$$ 
\begin{NB}
$E:=\Phi_{X \to X}^{I_\Delta^{\vee}[1]}( {\cal P}_{2| X \times \{ x \}} ) \in {\cal M}_{H_f}(2+f)$
fits in an exact sequence
$$
0 \to {\cal O}_X^{\oplus 2} \to E \to {\cal P}_{2| X \times \{ x \}} \to 0.
$$
\end{NB}
Thus we can not find a relative Fourier-Mukai transform
inducing a birational map.
\end{rem}

(3) The wall $W_{-\frac{3}{5}}$: $s=-\frac{3}{5}$.

It is also a totally semistable wall defined by $v({\cal O}_X(-2f))$.
\begin{NB} 
$v=(1,3f,-2)+(1,-2f,1)=(0,5f,-3)+2(1,-2f,1)$.
\end{NB}

We set
$$
\Lambda:=\Phi_{X \to X}^{{\cal P}_1^{\vee}[2]} \circ D_X \circ
 (\Phi_{X \to X}^{{\cal P}_1^{\vee}[2]})^{-1}.
$$
Then we see that
$$
\Lambda(r+dH+kf+a \varrho_X)=(r+2d)-dH-(k+2a)f+a \varrho_X.
$$
$\Lambda$ acts on the parameter space by
$s \mapsto \frac{-s}{1+2s}$.
So we get
\begin{equation}\label{eq:Lambda}
\begin{matrix}
\Lambda(J_1)=J_4, & \Lambda(J_2)=J_3,\\
\Lambda(W_{-3})=W_{-\frac{3}{5}},& \Lambda(W_{-1})=W_{-1}.
\end{matrix}
\end{equation}

\begin{prop}\label{prop:Lambda}
$\Lambda$ induces an isomorphism
$$
M_{(f,H)}^{{\cal O}_X(sH)}(2+f-\varrho_X) \to
M_{(f,H)}^{{\cal O}_X(s' H)}(2+f-\varrho_X) 
$$
where $s$ is in a chamber and $s'=\frac{-s}{1+2s}$.
In particular
\begin{equation}
M_{J_1}(2+f-\varrho_X) \cong M_{J_4}(2+f-\varrho_X),\;
M_{J_2}(2+f-\varrho_X) \cong M_{J_3}(2+f-\varrho_X).
\end{equation}
\end{prop}

\begin{proof}
$\Phi_{X \to X}^{{\cal P}_1^{\vee}[2]}$ and $D_X$ preserve $(f,G,H)$-semistability
by Theorem \ref{thm:FMpreserve} and Definition \ref{defn:fGH2}.
\end{proof}

By \eqref{eq:Lambda} and Proposition \ref{prop:Lambda},
we get a description of $W_{-\frac{3}{5}}$ from $W_{-3}$. 
As a summary of our description of walls,
we present a general members of the moduli spaces. 
\begin{prop}
\begin{enumerate}
\item
A general $E \in M_{J_2}(2+f-\varrho_X)=M_{H_f}(2+f-\varrho_X)$ fits in 
ian exact sequence
$$
0 \to  {\cal O}_X \to E \to I_Z(f) \to 0,
$$
where $I_Z$ is an ideal sheaf of three points. 
\item
A general $E \in M_{J_2}(2+f-\varrho_X)$ fits in an exact sequence
$$
0 \to I_Z(f) \to E \to {\cal O}_X \to 0,
$$
where $I_Z$ is an ideal sheaf of three points. 
\item
A general $E \in M_{J_3}(2+f-\varrho_X)$ fits in an exact sequence
$$
0 \to \oplus_{i=1}^3 L_i \to E \to {\cal O}_X(-f)^{\oplus 2} \to 0,
$$
$L_i \in M_{H_f}(f-\varrho_X)$.
\item
A general $E \in M_{J_4}(2+f-\varrho_X)$ fits in an exact sequence
$$
0 \to F \to E \to {\cal O}_X(-2f) \to 0,
$$
where
$F \in M_{J_5'}(1+2f-2\varrho_X)$
\end{enumerate}

\end{prop}

\begin{proof}
It is sufficient to prove (iv).
$\Lambda$ induces an isomorphism
$$
M_{J_1}(1+f-2\varrho_X) \cong M_{J_4}(1+2f-2\varrho_X)
$$
and $\Lambda({\cal O}_X)={\cal O}_X(-2f)$.
Since $W_{-1}$ is the unique totally semistable wall for
$1+2f-2\varrho_X$
in $s \leq -\frac{1}{2}$,
$M_{J_4}(1+2f-2\varrho_X) \cap M_{J_5'}(1+2f-2\varrho_X) \ne \emptyset$.
Therefore (iv) follows from (i).
\end{proof}

\begin{NB}
$0 \to {\cal O}_X(-3f) \to F \to \oplus_{i=1}^3 L_i' \to 0$,
$L_i' \in M_H(2f-\varrho_X)$. 
\end{NB}

\begin{NB}
$\Phi_{X \to X}^{{\cal P}_1^{\vee}[2]}({\cal O}_X (f))={\cal O}_X$.
$(\Phi_{X \to X}^{{\cal P}_1^{\vee}[2]})^{-1}({\cal P}_{1| X \times \{ x \}}^{\vee}[1])={\Bbb C}_x[-1]$. 
$(\Phi_{X \to X}^{{\cal P}_1^{\vee}[2]})^{-1}({\cal O}_X(-f))={\cal O}_X$.
$D_X({\Bbb C}_x[-1])={\Bbb C}_x[-1]$ and $D_X({\cal O}_X)={\cal O}_X$.
Hence $\Lambda({\cal P}_{1| X \times \{ x \}}^{\vee}[1])={\cal P}_{1| X \times \{ x \}}^{\vee}[1]$
and $\Lambda({\cal O}_X(-f))={\cal O}_X(-f)$.
\end{NB}

\begin{NB}
$s=-\frac{1}{2}$.
$v=(0,2f,-1)+(2,-f,0)$. Uhlenbeck contraction.

$s=-\frac{3}{7}$.
$v=(0,7f,-3)+2(1,-3f,1)$.

$s=-\frac{1}{3}$.
$v=3(0,3f,-1)+2(1,-4f,1)$.
Hilbert-Chow contraction.
\end{NB}

\begin{NB}

Assume that $v=(2,f,-3)$.
Then the walls in $s \leq -\frac{1}{2}$ are the following.

\begin{NB2}
$v_1=(1,k_1 f,a_1)$ defines a wall if
$-4 \leq a_1 \leq 1$.
$s=\frac{5}{2k_1-1}$.
\end{NB2}

\begin{enumerate}
\item
$s=-5$. $v_1=(0,f,-5)$.
\begin{NB2}
${\Bbb P}^2$-bundle over $M_H(1,f,-4)$:
$$
0 \to {\cal O}_X \to E \to I_Z(f) \to 0.
$$ 
\end{NB2}
\item
$s=-3$.
$v_1=(0,f,-3)$.
\item
$s=-2$.
$v_1=(0,f,-2)$.
\item
$s=-\frac{5}{3}$.
$v_1=(0,3f,-5)$.
\begin{NB2}
${\Bbb P}^2$-bundle over $M_H(1,2f,-4)$:
$$
0 \to {\cal O}_X(-f) \to E \to I_Z(2f) \to 0.
$$ 
\end{NB2}
\item
$s=-\frac{3}{2}$. $v_1=(0,2f,-3)$.
\item
$s=-1$.
$v_1=(0,f,-1)$. Uhlenbeck contraction.
\item
$s=-\frac{3}{4}$.
$v_1=(0,4f,-3)$.
\item
$s=-\frac{5}{7}$.
$v_1=(0,7f,-5)$.
\item
$s=-\frac{2}{3}$.
$v_1=(0,3f,-2)$.
\item
$s=-\frac{3}{5}$.
$v_1=(0,5f,-3)$.
\item
$s=-\frac{5}{9}$.
$v_1=(0,9f,-5)$.
\end{enumerate}

\begin{NB2}
$$
\Lambda'=\Phi_{X \to X}^{{\cal P}_1^{\vee}[2]} \circ \otimes {\cal L}_2 \circ D_X
\circ (\Phi_{X \to X}^{{\cal P}_1^{\vee}[2]})^{-1},
$$
where $c_1({\cal L}_2)=-2f$.
$$
\Lambda'(r+dH+kf+a \varrho_X)=
(r+2d)-dH-(2r+4d+k+2a)f+(2d+a)\varrho_X.
$$
$\Lambda'(2+f-3\varrho_X)=2+f-3\varrho_X$.
$s'=\frac{-s}{1+2s}$.

$\Lambda'((-\infty,-1))=(-1,-\frac{1}{2})$.
For $s=-5,-3,-2,-\frac{5}{3},-\frac{3}{2}$, we have flops.

For $s=-1$, we have an isomorphism
$M_{(-\frac{3}{2},-1)}(2+f-3\varrho_X) \cong M_H(2-2f-3\varrho_X)$.
Hence we have a divisorial contraction.

\end{NB2}

\begin{rem}
As in \cite{BM:2},
we have an example of wall such that the associated lattice contains infinitely many
$(-2)$-vectors.

Let $\pi:X \to {\Bbb P}^1$ be an elliptic K3 surface such that
$\NS(X)={\Bbb Z}\sigma+{\Bbb Z}f+{\Bbb Z}D$, where $(D \cdot \sigma)=(D \cdot f)=0$ and
$(D^2)=-2(n+1)$ $(n \geq 4)$.

We set $u_1=1-f+\varrho_X, u_2=1+D-n \varrho_X$.
Then $\langle u_1^2 \rangle=\langle u_2^2 \rangle=-2$ and $\langle u_1,u_2 \rangle=(n-1)$.
$L:={\Bbb Z}u_1+{\Bbb Z}u_2$ is a hyperbolic lattice containing infinitely many $(-2)$-vectors.
For $v=xu_1+yu_2$ with $x,y>0$ and 
$-x^2+(n-1)xy-y^2>0$, $u_1$ defines a wall such that the associated lattice is $L$.
 \end{rem}

\end{NB}

\subsection{Examples for an elliptic abelian surface}

Let $\pi:X \to C$ be an elliptic abelian surface.
\begin{NB}
We set $v=r+kf-\varrho_X$.
Then $(1,k_1 f,0)$ defines a totally semistable walls 
$s=-\frac{1}{k-rk_1}$.

Assume that $k=1$.
If $s<-1$, then we have an exact sequence
$$
0 \to E_1 \to E \to L \to 0
$$ 
where $L \in {\cal M}_H(0,f,-1)$ and $E_1 \in {\cal M}_H(r,0,0)$.
We set $X':=M_H(0,2f,-1)$.
Let ${\cal P}$ be a universal family.
Then
$\Phi_{X \to X'}^{{\cal P}^{\vee}[1]}(L)[2]$ is a stable sheaf and
$\Phi_{X \to X'}^{{\cal P}^{\vee}[1]}(E_1)[1]$ is a stable sheaf.
Hence $\Phi_{X \to X'}^{{\cal P}^{\vee}[1]}(E)$ is a two-term complex.
\end{NB}
\begin{NB}
We set $v=r+(kf+D)-\varrho_X$, with
$(D^2)=-2m<0$.
Then $\langle v^2 \rangle=2(r-m)$ and
we have 
$$
v=(r-m)(1+k_1 f)+(m+(k-rk_1+mk_1)f+D-\varrho_X).
$$
$1+k_1 f$ defines a totally semistable wall 
$s=-\frac{1}{k-rk_1}$.
\end{NB}
Let 
$v=r+\xi+a \varrho_X$ be a primitive Mukai vector.
Assume that $l:=\gcd(r,(\xi \cdot f))>1$.
An isotropic Mukai vector
$u:=p f+q \varrho_X$ defines a totally semistable wall if and only if
$\langle v,u \rangle=l$, $l \mid (v-u)$ and $\langle (v-u)^2 \rangle=0$.
Thus there is an isotropic Mukai vector 
$w=r_0+\xi_0+a_0 \varrho_X$ such that 
$p (\xi_0 \cdot f)-q r_0=1$ and $v=lw+u$.
In this case, $r=lr_0$, $(\xi \cdot f)=l(\xi_0 \cdot f)$ and 
$\langle v^2 \rangle=2l \leq 2r$.
Let $u'=p' f+q' \varrho_X$ be another isotropic vector 
defining totally semistable wall.
Then we see that $u'=u+nl(r_0 f+(\xi_0 \cdot f) \varrho_X)$ ($n \in {\Bbb Z})$.
We take $p$ with $0<p<r$.
\begin{NB}
$v \equiv u \mod l$ and $v \equiv u' \mod l$ implies
$l \mid (u'-u)$.

$\frac{q}{p}<\frac{q+(\xi \cdot f)n}{p+rn}<\frac{(\xi \cdot f)}{r}$.
\end{NB}

Assume that $v(G)=e^{xH+\alpha}$, where $(x,\alpha) \in {\cal G}(v)$.
If $x<\frac{q}{p}$, then all totally semistable walls for $(f,G,H)$-semistability are
\begin{equation}
v=l v_1+v_2, \langle v_1,v_2 \rangle=1,\;v_i=r_i e^{\eta_i}, (\eta_i \cdot f)=\frac{(\xi \cdot f)}{r}. 
\end{equation}
Since $r_1 r_2((\eta_1-\eta_2)^2)=-2$,
$\eta_1-\eta_2$ defines a totally semistable wall with respect to $\mu$-semistability.
\begin{NB}
there is an ample ${\Bbb Q}$-divisor $L$ such that
$(\eta_1 \cdot L)=(\eta_2 \cdot L)$.
We note that $(\eta_1 \cdot (L+mf))=(\eta_2 \cdot (L+mf))$.
\end{NB}
For each chamber, there is an ample divisor $L$ such that 
${\cal M}_{(f,H)}^G(v) \cap {\cal M}_L(v) \ne \emptyset$.  
Since $((\eta_1-\eta_2) \cdot f)=0$,
${\cal M}_L(v)={\cal M}_{L_f}(v)$, and hence
${\cal M}_{(f,H)}^G(v) \cap {\cal M}_{L_f}(v) \ne \emptyset$.

Assume that 
$\langle v,(r' f+d' \varrho_X) \rangle=lt$.
Then there is an integer $n$ such that
$r'=tp+n r_0$ and $d'=tq+n(\xi \cdot f)$.
$$
\frac{d'}{r'}-\frac{q}{p}=\frac{n}{pr'}.
$$
If $r' \leq tp$, then $n \leq 0$.
For $X'=M_{H_f}(r' f+d' \varrho_X)$, 
$\Phi_{X \to X'}^{{\cal P}^{\vee}[1]}$ induces an isomorphism
${\cal M}_{(f,H)}^G(v) \cong {\cal M}_{H'_{f'}}(v')$.
Since $\tfrac{d'}{r'}<\tfrac{q}{p}$, we get 
a birational map
$$
{\cal M}_{L_f}(v) \cdots \to {\cal M}_{H'_{f'}}(v').
$$

\begin{lem}
Assume that $l>1$ and $\langle v^2 \rangle>2r$.
Then there is no totally semistable wall.
\end{lem}

\begin{proof}
Assume that there is a decomposition $v=\ell v_1+v_2$ such that 
$\langle v_1^2 \rangle=\langle v_2^2 \rangle=0$
and $\langle v_1,v_2 \rangle=1$, where
$v_i=r_i+\xi_i+a_i \varrho_X$ with $r_i \geq 0$.
If $r_1=0$, then we get
$\xi_1=mf$ ($m \in {\Bbb Z}_{>0}$) and
$(\xi \cdot \xi_1)-ra_1=\langle v_1,v \rangle=\langle v_1,v_2 \rangle=1$,
which implies $l=1$.
Hence $r_1>0$ and $r \geq \ell=\langle v^2 \rangle/2$.
Therefore there is no totally semistable wall.
\end{proof}

\begin{rem}
If $r_2=0$, then $\xi_2=m f$ $(m \in {\Bbb Z})$ and
$m(\xi_1 \cdot f)-r_1 a_2=1$. Hence
$l=\ell$.
\end{rem}

   
\subsubsection{An example for $X=C \times C$.}
Let $C$ be an elliptic curve with $\End(C) \cong {\Bbb Z}$
and
$X=C \times C$.
Let $\pi:C \times C \to C$ be a projection,
$f$ a fiber of $\pi$ and 
$H$ the $0$-section of $\pi$.
Then $(H^2)=(f^2)=0$ and $(H \cdot f)=1$.
We set $\xi:=\Delta-H-f$, where $\Delta$ is the diagonal.
Then $\xi \in N_H$, $(\xi^2)=-2$ and 
$\NS(X)={\Bbb Z}H+{\Bbb Z}f+{\Bbb Z}\xi$.
\begin{NB}
Let $\Gamma_n$ be the graph of $n:C \to C$.
Then $\Gamma_n=H+n^2 f+n\xi$.
\end{NB}

We set
$v=2+f+\xi-2\varrho_X$.
Then
$\langle v^2 \rangle=6$.
We shall give a chamber decomposition of 
$$
{\cal G}(v)=\{(x,y\xi) \mid x \ne 0 \}.
$$
By simple calculations, we see that there are three families of decompositions
$$
v=u_1^k+u_2^k,\;v=v_1^k+v_2^k,\; v=w_1^k+w_2^k,\; k \in {\Bbb Z}, 
$$
where 
\begin{equation}
\begin{split}
u_1^k:=& kf-\varrho_X,\;\; u_2^k:=2+(1-k)f+C-\varrho_X,\\
v_1^k:=& e^{kf+\xi}=1+(kf+\xi)-\varrho_X,\;\;v_2^k:= 1+(1-k)f-\varrho_X,\\
w_1^k:=& 1+(kf+\xi)-2\varrho_X,\;\;w_2^k:= e^{(1-k)f}=1+(1-k)f.
\end{split}
\end{equation}
Thus we have three families of walls $U_k,V_k,W_k$ defined by $u_1^k,v_1^k,w_1^k$ respectively. 
$$
U_k:x=-\frac{1}{k},\;V_{k}:y=\frac{2k-1}{2}x,\;
W_{k}:y=\frac{2k-1}{2}x+1.
$$
In particular there is no totally semistable wall and each wall 
corresponds to a Donaldson-Uhlenbeck type divisorial contraction.
We note that
\begin{equation}
\begin{split}
-\langle e^{xH+y\xi},v_1^k \rangle<-\langle e^{xH+y\xi},v_2^k \rangle
& \iff y<\frac{2k-1}{2}x.\\
-\langle e^{xH+y\xi},w_1^k \rangle<-\langle e^{xH+y\xi},w_2^k \rangle
& \iff y>\frac{2k-1}{2}x+1.
\end{split}
\end{equation}

\begin{NB}

$$
U_k:x=-\frac{1}{k}.
$$

$v_1^k$ defines a wall 
$$
V_{k}:y=\frac{2k-1}{2}x.
$$
\begin{equation}
-\langle e^{xH+y\xi},v_1^k \rangle<-\langle e^{xH+y\xi},v_2^k \rangle
\iff y<\frac{2k-1}{2}x.
\end{equation}

$w_1^k$ defines a wall 
$$
W_{k}:y=\frac{2k-1}{2}x+1.
$$
\begin{equation}
-\langle e^{xH+y\xi},w_1^k \rangle<-\langle e^{xH+y\xi},w_2^k \rangle
\iff y>\frac{2k-1}{2}x+1.
\end{equation}
\end{NB}

By using \eqref{eq:auto}, we shall describe chambers in $x<-\frac{1}{2}$. 
Assume that $x<-1$.
Then there are two kind of chambers.
Their closure intersect the line at infinity and 
they parameterizes Gieseker semistable sheaves (Proposition \ref{prop:unbounded}).
\begin{enumerate}
\item
${\cal C}_k$:$\frac{2k+1}{2}x+1<y<\frac{2k-1}{2}x$.
\begin{NB}
It is a triangle bounded by $W_{k+1}$, $V_k$ and the line at infinity.
We set
$D_1^+:=\{E \mid E_2 \subset E, E_2 \in {\cal M}_{(f,H)}^G(w_2^{k+1}) \}$.
$D_1^+$ is a ${\Bbb P}^1$-bundle over ${\cal M}_{(f,H)}^G(w_1^{k+1}) \times {\cal M}_{(f,H)}^G(w_2^{k+1})$.
It is contracted on $W_{k+1}$.

We set
$D_2^-:=\{ E \mid E_1 \subset E, E_1 \in {\cal M}_{(f,H)}^G(v_1^k) \}$.
It is contracted on $V_k$.
\end{NB}
We have 
${\cal M}_{(f,H)}^G(v)={\cal M}_{(H+p\xi)_f}(v)$ 
for $\frac{2k-1}{2}<p<\frac{2k+1}{2}$.

\item
${\cal C}_k'$:
$\frac{2k-1}{2}x<y<\frac{2k-1}{2}x+1$.
\begin{NB}
It is a triangle bounded by $U_1$, $V_k$ and $W_k$.
$D_2^+:=\{ E \mid E_2 \subset E, E_2 \in {\cal M}_{(f,H)}^G(v_2^k) \}$.
It is contracted on $V_k$.

$D_1^-:=\{ E \mid E_1 \subset E, E_1 \in {\cal M}_{(f,H)}^G(w_1^k) \}$.
$D_1^-$ is a ${\Bbb P}^1$-bundle over ${\cal M}_{(f,H)}^G(w_1^k) \times {\cal M}_{(f,H)}^G(w_2^k)$.
It is contracted on $W_k$.
\end{NB}
We have
${\cal M}_{(f,H)}^G(v)={\cal M}_{(H+\frac{2k-1}{2}\xi)_f}^G(v)$.
\end{enumerate}

For $-1<x<-\frac{1}{2}$, we use the following Fourier-Mukai transform to describe chambers.
Let ${\cal L}_{-f+\xi}$ be a line bundle with $c_1({\cal L}_{-f+\xi})=-f+\xi$.
Then
$$
\Xi:=\Phi_{X \to X}^{{\cal P}_1^{\vee}[2]} \circ \otimes {\cal L}_{-f+\xi} \circ D_X \circ 
(\Phi_{X \to X}^{{\cal P}_1^{\vee}[2]})^{-1}
$$
 induces an isomorphism
$$
{\cal M}_{x,y}(v) \stackrel{\sim}{\longrightarrow} {\cal M}_{\frac{-x}{1+2x},\frac{1+x-y}{1+2x}}(v),
$$
where ${\cal M}_{x,y}(v):={\cal M}_{(f,H)}^G(v)$ with $v(G)=e^{xH+y \xi}$.
Hence
$\Xi({\cal C}_k), \Xi({\cal C}_k')$ $(k \in {\Bbb Z})$ are chambers in $-1<x<-\tfrac{1}{2}$
 and they are described as follows:
$$
\Xi({\cal C}_k):
x>-\frac{1}{2},\;y<\frac{2k+1}{2}x+1,\; y<\frac{2k-1}{2}x,
$$
$$
\Xi({\cal C}_k'):
x<-1,\;y>\frac{2k+1}{2}x+1,\; y<\frac{2k-3}{2}x.
$$

\begin{NB}
$\Xi({\cal C}_k)$: triangle with the virtices
$$
V_{k} \cap W_{k+1}=\{(-1,\tfrac{1-2k}{2})\},\;
V_{k} \cap U_2=\{(-\tfrac{1}{2},\tfrac{1-2k}{4})\},\;
U_2 \cap W_{k+1}=\{(-\tfrac{1}{2},\tfrac{3-2k}{4})\}.
$$
\end{NB}

\begin{NB}
$\Xi({\cal C}_k')$: triangle with the virtices
$$
V_{k-1} \cap W_{k+1}=\{(-\tfrac{1}{2},\tfrac{3-2k}{4})\},\;
V_{k-1} \cap U_1=\{(-1,\tfrac{3-2k}{2})\},\;
U_1 \cap W_{k+1}=\{(-1,\tfrac{1-2k}{2})\}.
$$
\end{NB}

\begin{NB}
$\Xi:{\cal M}(v_1^k) \cong {\cal M}(w_2^{k+1})$.

For $E \in {\cal M}_{H_f}(v_1^k)$,
we have an exact sequence
$$
0 \to F_1 \to \Xi(E) \to F_2 \to 0
$$ 
$F_1 \in {\cal M}_{H_f}(f-\varrho_X)$ and
$F_2 \in {\cal M}_{H_f}(1+(k-1)f+\xi-\varrho_X)$.
$\Xi(E)$ is not torsion free.

$\Xi$ acts on ${\cal G}(v)$ as an involution with $\Xi_{|U_1}=1_{U_1}$ and
$\Xi(V_k)=W_{k+1}$.
$V_k \cap W_{k+1}=\{(-1,\frac{1-2k}{2})\}$.
$V_k \cap U_2=\{(-\frac{1}{2},\frac{1-2k}{4})\}$.
$W_{k+1} \cap U_2=\{(-\frac{1}{2},\frac{3-2k}{4})\}$.
\end{NB}

\begin{NB}
For $v'=2+f+\xi-\varrho_X$, walls are defined by $e^{kf+\xi}$.
\end{NB}

\begin{NB}
For $H+\frac{2k-1}{2}\xi$,
$((kf+\xi) \cdot (H+\frac{2k-1}{2}\xi))=((1-k)f \cdot (H+\frac{2k-1}{2}\xi))$.
\end{NB}

\begin{NB}

\section{}

Assume that 
$X':=M_H^G(0,r_1 f,d_1)$ is a fine moduli space,
where $\rk G d_1-r_1 (c_1(G) \cdot f)=0$.
Then $F$ is $(G,H)$-semistable if and only if 
$\Hom_{p_{X'}}(F,{\cal P}[1])$ is semistable.
\begin{NB2}
$\Hom(F,{\cal P}_{|X \times \{ x' \}}[2])=
\Hom({\cal P}_{|X \times \{ x' \}},F)^{\vee}=0$.
\end{NB2}

\begin{lem}[{\cite[Lem. 1.5, Rem. 1.6]{Y:elliptic}}]\label{lem:G-stability}
Let $A$ be a $G$-twisted stable sheaf supported on a fiber.
\begin{enumerate}
\item[(1)]
Assume that 
$\chi(G,A)<0$. Then $\Phi_{X \to Y}^{{\cal P}^{\vee}}(A)[2]$ is a $G'$-twisted stable sheaf. 
\item[(2)]
Assume that $\chi(G,A)>0$. Then $\Phi_{X \to Y}^{{\cal P}^{\vee}}(A)[1]$ is a $G'$-twisted stable sheaf. 
\end{enumerate}
\end{lem}

\begin{rem}\label{rem:G-stability}
\begin{enumerate}
\item[(1)]
By \eqref{eq:PQ}, Lemma \ref{lem:G-stability} implies that
$A \in {\bf D}(X)$ is a $G$-twisted stable sheaf supported on a fiber
such that $\chi(G,A)<0$
if and only if $A':=\Phi_{X \to Y}^{{\cal P}^{\vee}}(A)[2]$
is a $G'$-twisted stable sheaf supported on a fiber such that $\chi(G',A')>0$.
\item[(2)]
$G'$-twisted stability is also defined for 0-dimensional sheaves and 
$G$-twisted stable sheaf $A$ supported on a fiber with $\chi(G,A)=0$
corresponds to a structure sheaf of a point via $\Phi_{X \to Y}^{{\cal P}^{\vee}[2]}$.
\end{enumerate}
\end{rem}

\begin{rem}\label{rem:G-stability2}
Let $A$ be an object of ${\bf D}(X)$.
\begin{enumerate}
\item[(1)]
$A$ is a $G$-twisted stable sheaf supported on a fiber
such that $\chi(G,A)<0$
if and only if $A'':=\Phi_{X \to Y}^{{\cal P}^{\vee}[1]}(A)^{\vee}$
is a $G''$-twisted stable sheaf supported on a fiber such that $\chi(G'',A'')<0$.
\item[(2)]
$A$ is a $G$-twisted stable sheaf supported on a fiber
such that $\chi(G,A)>0$
if and only if $A'':=\Phi_{X \to Y}^{{\cal P}^{\vee}[1]}(A)^{\vee}[1]$
is a $G''$-twisted stable sheaf supported on a fiber such that $\chi(G'',A'')>0$.
\end{enumerate}
\end{rem}

We have an order
$$
{\Bbb C}_x[-1]<A_1<{\cal P}_{|X \times \{ y \}} < A_2<{\Bbb C}_x
$$
where $A_1,A_2$ are $\alpha$-twisted stable sheaves with 
$\chi(G,A_1)<0<\chi(G,A_2)$.

\section{}

For a coherent sheaf $E$ with $\rk E>0$, we set
$$
\mu_f^\beta (E):=((c_1(E)-\rk E \beta) \cdot f)/\rk E.
$$
For $E$, we set
$$
\mu_{\min,f}^\beta(E):=\min \{\mu_f^\beta(F) \mid
\text{ $F$ is a quotient of $E$ with $\rk F>0$} \}.
$$

For $t \gg s \gg 0$,
we set $\omega:=t^{-1}H+ts f$.
\begin{NB2}
$(s,t)=(\frac{n}{r_0^2},r_0^2 m)$.
\end{NB2}

We note that
\begin{equation}
\begin{split}
\widehat{Z}_{(\beta,\omega_{s,t})}(E)=&
\langle e^{\beta+\sqrt{-1}(t^{-1}H+st f)},e^\beta(r+\eta^\beta+a^\beta) \rangle\\
=&
sr-a^\beta+\sqrt{-1}((t^{-1}H+stf) \cdot \eta^\beta),
\end{split}
\end{equation}
where $v(E)=e^\beta(r+\eta^\beta+a^\beta \rho_X)$.

For $v=e^\beta(r+\eta^\beta+a^\beta \rho_X), v_1=e^\beta(r_1+\eta_1^\beta+a_1^\beta \rho_X)$,

\begin{equation}\label{eq:slope}
\begin{split}
& \frac{a^\beta-sr}{st(\eta^\beta \cdot f)+t^{-1}(\eta^\beta \cdot H)}-
\frac{a_1^\beta-sr_1}{st(\eta_1^\beta \cdot f)+t^{-1}(\eta_1^\beta \cdot H)}\\
=& \frac{s^2 t(r_1(\eta^\beta \cdot f)-r(\eta_1^\beta \cdot f))+
st((-a_1^\beta (\eta^\beta \cdot f)-(-a)(\eta_1^\beta \cdot f))}
{(st(\eta^\beta \cdot f)+t^{-1}(\eta^\beta \cdot H))(st(\eta_1^\beta \cdot f)+
t^{-1}(\eta_1^\beta \cdot H))}\\
& +\frac{s t^{-1}(r_1(\eta^\beta \cdot H)-r(\eta_1^\beta \cdot H))+
t^{-1}((-a_1^\beta (\eta^\beta \cdot H)-(-a)(\eta_1^\beta \cdot H))}
{(st(\eta^\beta \cdot f)+t^{-1}(\eta^\beta \cdot H))(st(\eta_1^\beta \cdot f)+t^{-1}(\eta_1^\beta \cdot H))}.
\end{split}
\end{equation}

\begin{equation}
\begin{split}
& \frac{a^\beta-sr}{st(\eta^\beta \cdot f)+t^{-1}(\eta^\beta \cdot H)}-
\frac{a_1^\beta-sr_1}{nt(\eta_1^\beta \cdot f)+t^{-1}(\eta_1^\beta \cdot H)} >0 \;(t \gg s \gg 0)\\
\iff & 
(r_1(\eta^\beta \cdot f)-r(\eta_1^\beta \cdot f),-a_1^\beta(\eta^\beta \cdot f)+a(\eta_1^\beta \cdot f),
r_1(\eta^\beta \cdot H)-r(\eta_1^\beta \cdot H))>0
\end{split}
\end{equation}

\begin{prop}
Let $E$ be an object of ${\bf D}(X)$ with
$v(E)=e^\beta(r+\eta^\beta+a^\beta \rho_X)$. 
Assume that $(\eta^\beta \cdot f)>0$ and
$E$ is $\widehat{\sigma}_{(\beta,\omega_{s,t})}$-stable for all $t \gg s \gg 0$.
Then
$E \in \Coh(X)$ and one of the following holds for all subsheaf $E_1$ of $E$ with
$v(E_1)=e^\beta(r_1+\eta_1^\beta+a_1^\beta \rho_X)$:
\begin{enumerate}
\item
$(\eta_1^\beta \cdot f)<\frac{r_1}{r}(\eta^\beta \cdot f)$
\item
$(\eta_1^\beta \cdot f)=\frac{r_1}{r}(\eta^\beta \cdot f)$ and 
$a_1^\beta <\frac{r_1}{r}a^\beta $
\item
$(\eta_1^\beta \cdot f)=\frac{r_1}{r}(\eta^\beta \cdot f)$, 
$a_1^\beta =\frac{r_1}{r}a^\beta $ and
$(\eta_1^\beta \cdot H)<\frac{r_1}{r}(\eta^\beta \cdot H)$.
\end{enumerate} 
\end{prop}

\begin{proof}
Let $E$ be a $\widehat{\sigma}_{(\beta,\omega_{s,t})}$-semi-stable object with
 $v(E)=e^\beta(r+\eta^\beta+a^\beta \varrho_X)$
with $(\eta^\beta \cdot f)>0$.
Assume that $H^{-1}(E) \ne 0$.
We set $v(H^{-1}(E)[1]):=e^\beta(r_1+\eta^\beta_1 +a_1^\beta \varrho_X)$.
Then $r_1<0$ and $(\eta^\beta_1 \cdot (H+t^2 s f)) \geq 0$ for all $t \gg s \gg 0$.
In particular $(\eta^\beta_1 \cdot f) \geq 0$.
Then $r_1(\eta^\beta \cdot f)-r(\eta^\beta_1 \cdot f)<0$, which implies $E$ is not semi-stable.
Therefore $E$ is a sheaf.

Assume that there is a subsheaf $E_1$ of $E$ such that
$v(E_1)=e^\beta(r_1+\eta_1^\beta+a_1^\beta \rho_X)$ with 
$r_1<r$ and
$(\eta_1^\beta \cdot f) > \frac{r_1}{r}(\eta^\beta \cdot f)$.
We set $E_2:=E/E_1$.
We may assume that
$\mu_{\min, f}^\beta(E_1)>\mu_f^\beta(E_2) \geq 0$.
Then we have $E_1 \in {\cal A}_{(\beta,\omega_{s,t})}$ for $t \gg s \gg 0$.
Hence we have an exact sequence 
\begin{equation}
0 \to E_1 \to E \to E_2 \to 0
\end{equation}
in ${\cal A}_{(\beta,\omega_{s,t})}$ ($t \gg s \gg 0$).
\begin{NB2}
Since $E \in {\cal A}_{(\beta,\omega_{s,t})}$ for $t \gg s \gg 0$,
$(\eta^\beta_2 \cdot(H+st^2 f))>0$ for $t \gg s \gg 0$.
Thus $(\eta^\beta_2 \cdot f) \geq 0$.
If $E$ is not $f$-semi-stable, then by taking the Harder-Narashimhan filtration
of $E$, 
we may assume that $\mu_{\min, f}^\beta(E_1)>\mu_f^\beta(E_2) \geq 0$.
Then we have $E_1 \in {\cal A}_{(\beta,\omega_{s,t})}$,
\end{NB2}
By $\widehat{\sigma}_{(\beta,\omega_{s,t})}$-semi-stability of $E$,
$r_1 (\eta^\beta \cdot f)-r(\eta^\beta_1 \cdot f) \leq 0$, which is a contradiction.
Therefore $(\eta_1^\beta \cdot f) \leq  \frac{r_1}{r}(\eta^\beta \cdot f)$
for all subsheaves $E_1$ with $v(E_1)=e^\beta(r_1+\eta_1^\beta+a_1^\beta \rho_X)$ and 
$r_1<r$.
Assume there is a subsheaf $E_1$ of $E$ such that
$v(E_1)=e^\beta(r_1+\eta_1^\beta+a_1^\beta \rho_X)$ with 
$r_1<r$ and
$(\eta_1^\beta \cdot f) = \frac{r_1}{r}(\eta^\beta \cdot f)$.
We set $E_2:=E/E_1$.
Then $\mu_{\min, f}^\beta(E_1)=\mu_f^\beta (E)>0$ and 
$E_1 \in {\cal A}_{(\beta,\omega_{s,t})}$ for $t \gg s \gg 0$.
Hence we have an exact sequence 
\begin{equation}
0 \to E_1 \to E \to E_2 \to 0
\end{equation}
in ${\cal A}_{(\beta,\omega_{s,t})}$ ($t \gg s \gg 0$).
The we also see that (ii) or (iii) holds.
\end{proof}

\begin{rem}
If $E$ satisfies the conditions, then $E \in {\cal A}_{(\beta,\omega_{s,t})}$ for $t \gg s \gg 0$.
\end{rem}

Assume that $\pi:X \to C$ is an elliptic abelian surface.
Then $(0,f,0)$ defines a fibration $M_H(v) \to S^{\ell} C$.

Mukai vectors $u$ such that
$\langle u^2 \rangle=0$ and 
$0 < \langle v,u \rangle \leq \ell$,
$0<\langle u^2 \rangle$ and $2\langle u^2 \rangle+1 \leq \langle v,u \rangle \leq \ell$.

We set $X':=M_H(0,f,-1)$.
Let ${\cal P}$ be a universal family on $X \times X'$.
Since $\Phi_{X \to X'}^{{\cal P}^{\vee}[1]}({\cal O}_X)$ is a line bundle on $X'$, we may assume that
$\Phi_{X \to X'}^{{\cal P}^{\vee}[1]}({\cal O}_X) \cong {\cal O}_{X'}$.
Then we see that
\begin{equation}
\begin{split}
\Phi_{X \to X'}^{{\cal P}^{\vee}[1]}(1)=& 1\\
\Phi_{X \to X'}^{{\cal P}^{\vee}[1]}(H)=& 1+H\\
\Phi_{X \to X'}^{{\cal P}^{\vee}[1]}(f)=&f\\
\Phi_{X \to X'}^{{\cal P}^{\vee}[1]}(\varrho_X)=& f+\varrho_{X'}. 
\end{split}
\end{equation}

Since $X' \cong X$, we regard $\Phi_{X \to X'}^{{\cal P}^{\vee}[1]}$ as an autoequivalence of ${\bf D}(X)$.

\begin{NB2}
$\Phi_1:=(T_{f} \circ \Phi_{X \to X'}^{{\cal P}^{\vee}[1]} \circ \Phi_{X \to X'}^{{\cal P}^{\vee}[1]})^{-1}$
is 
\begin{equation}
\begin{split}
\Phi_1(1)=& 1-f\\
\Phi_1(H)=& -2+H+2f-\varrho_X \\
\Phi_1(f)=&f \\
\Phi_1 (\varrho_X)=& -2f+\varrho_{X} 
\end{split}
\end{equation}

By the action of $\Phi_1$,
$\Phi_1(e^{sH})=(1-2s)e^{\frac{s}{1-2s}H}$.
We set $\varphi(s):=\frac{s}{1-2s}$.
Then $(-\infty,-\frac{1}{2})$ is the fundamental domain.
$x=0$ is the fixed point of $\varphi$.
\end{NB2}

We set
$\Phi_r:=(\Phi_{X \to X'}^{{\cal P}^{\vee}[1]})^{-r} \circ T_{-af}$.
Then we see that
\begin{equation}
\begin{split}
\Phi_r(1)=& 1-a f\\
\Phi_r(H)=& -r+H+arf-a\varrho_X \\
\Phi_r(f)=&f \\
\Phi_r (\varrho_X)=& -rf+\varrho_{X}. 
\end{split}
\end{equation}

We set $v:=(r,kf,-a)$. Then
$\Phi_r(v)=v$.

Assume that $a \geq 1$. Then $v_1=(0,rf,-1)$ satisfies
$\langle v,v_1 \rangle=2r$ and $\langle (v-v_1)^2 \rangle=2r(a-1)$.
Hence $v_1$ defines a wall $s=-\frac{1}{r}$. 
We set $\varphi_r (s)=\frac{s}{1-rs}$.
Then $\Phi_r(e^{sH})=(1-rs) e^{\varphi_r (s) H-af}$.
$\varphi_r$ is regarded as an automorphism of ${\Bbb P}^1={\Bbb R} \cup \{ \infty \}$. 
For any chamber $I$, there is an integer $n$ such that $\varphi_r^n(I) \subset (-\infty,-\frac{1}{r})$. 

$x=\varphi_r^n(-\frac{1}{r})$ $(n \in {\Bbb Z})$ are walls for $v$.
Thus there are infinitely many walls.

Assume that $v=(r,kf,-a)$ is primitive, that is, $\gcd(r,k,a)=1$.

Assume that $a>r \geq 3$.
Then $\Phi_r$ induces a birational map of $M(r,kf,-a)$.

\end{NB}

\end{document}